\documentclass[a4paper,11pt,
 leqno
]{article}

\usepackage{times}
\usepackage{comment}
\usepackage[utf8]{inputenc}
\usepackage[T1]{fontenc}
\usepackage{amsthm}
\usepackage{amssymb}
\usepackage{bbm}
\usepackage{mathrsfs}
\usepackage{amsmath}
\usepackage{amsfonts}
\usepackage{mathtools}
\usepackage{graphicx}
\usepackage{float}
\usepackage{dsfont}
\usepackage{enumerate}
\usepackage{enumitem}
\setlist{topsep=3pt,itemsep=1pt,parsep=0pt}
\usepackage[a4paper]{geometry}
\usepackage{csquotes}
\usepackage{IEEEtrantools}
\usepackage{appendix}
\usepackage{constants}
\usepackage[obeyDraft]{todonotes}
\usepackage{tikz}
\usetikzlibrary{calc,arrows.meta,decorations.pathmorphing,positioning}
\usepackage{enumerate}
\usepackage{mathabx}
\usepackage
{hyperref}
\hypersetup{linktocpage,
  colorlinks   = true, 
  urlcolor     = blue, 
  linkcolor    = blue,
  citecolor   = blue 
}\usepackage[all]{hypcap}  
\usepackage{breakurl}

\usepackage{titlesec}

\titleformat{\subsection}[runin]{\normalfont\bfseries}{\thesubsection.}{.5em}{}[.]
\titlespacing{\subsection}{0pt}{2ex plus .1ex minus .2ex}{.8em}

\numberwithin{equation}{section}
\allowdisplaybreaks[1] % let long multi-line displays break across pages if needed

\usepackage{caption}
\title{
\normalsize
\textbf{{ANOMALOUS SCALING LAW FOR THE TWO-DIMENSIONAL GAUSSIAN FREE FIELD
 }}}
\author{}
\date{}
\makeatletter
\newcommand{\1}{\mathbbm 1}
\newcommand{\Z}{\mathbb Z}
\newcommand{\R}{\mathbb R}
\renewcommand{\P}{\mathbb{P}}
\newcommand{\E}{\mathbb{E}}
\newcommand{\F}{\mathcal F}

\renewcommand{\epsilon}{\varepsilon}

\newcommand{\capacity}{\textnormal{cap}}

\newcommand{\sK}{\mathcal{E}}

\newtheorem{theorem}{Theorem}[section]
\newtheorem{lemma}[theorem]{Lemma}
\newtheorem{proposition}[theorem]{Proposition}

\theoremstyle{definition}

\newtheorem{remark}[theorem]{Remark}

\newcommand{\cref}[1]{\text{\upshape\ref{#1}}}

\newcommand{\overbar}[1]{\mkern 1.5mu\overline%
{\mkern-1.5mu#1\mkern-1.5mu}\mkern 1.5mu}

\newcommand{\Q}{\mathbb{Q}}

\newcommand{\K}{\mathcal{K}}
\newcommand{\A}{\mathcal{A}}

\newcommand{\cL}{\mathcal{L}}

\newcommand{\cP}{\mathcal{P}}

\newcommand{\cC}{\mathcal{C}}
\newcommand{\tC}{D}
\newcommand{\tcC}{\tilde{\mathcal{C}}}
\newcommand{\bC}{C}
\newcommand{\haC}{U}

\newcommand{\diam}{\text{diam}}

\newcommand{\ba}{\bar{a}}

\newcommand{\tB}{\tilde{B}}
\newcommand{\U}{O}
\newcommand{\Cir}{\text{Circ}}
\newcommand{\Con}{\text{Conn}}
\newcommand{\na}{\alpha}
\newcommand{\bna}{\overbar{\na}}

\newcommand{\I}{{\cal I}}
\newcommand{\V}{{\cal V}}

\renewcommand{\phi}{\varphi}
\renewcommand{\tilde}{\widetilde}
\renewcommand{\hat}{\widehat}
\renewcommand{\epsilon}{\varepsilon}

\newconstantfamily{c}{symbol=c}

\newcommand{\defeq}{\stackrel{\text{def.}}{=}}

\newcommand{\cS}{\mathcal{L}}

\newcommand{\tr}{\text{tr}}

\usepackage{color}
\definecolor{Red}{rgb}{1,0,0}
\definecolor{Blue}{rgb}{0,0,1}
\definecolor{Olive}{rgb}{0.41,0.55,0.13}
\definecolor{Yarok}{rgb}{0,0.5,0}
\definecolor{Green}{rgb}{0,1,0}
\definecolor{MGreen}{rgb}{0,0.8,0}
\definecolor{DGreen}{rgb}{0,0.55,0}
\definecolor{Yellow}{rgb}{1,1,0}
\definecolor{Cyan}{rgb}{0,1,1}
\definecolor{Magenta}{rgb}{1,0,1}
\definecolor{Orange}{rgb}{1,.5,0}
\definecolor{Violet}{rgb}{.5,0,.5}
\definecolor{Purple}{rgb}{.75,0,.25}
\definecolor{Brown}{rgb}{.75,.5,.25}
\definecolor{Grey}{rgb}{.7,.7,.7}
\definecolor{Black}{rgb}{0,0,0}

\usepackage{centernot}

\usepackage{titletoc}

\dottedcontents{section}[1.8em]{\addvspace{0.4em}}{1.8em}{0.7pc}
\dottedcontents{subsection}[4.3em]{}{2.5em}{0.7pc}

\usepackage{multicol}
\usepackage{parcolumns}

\usepackage{tabu}
\usepackage{caption}
\begin{document}
\maketitle
\vspace{0.1cm}
\begin{center}
\vspace{-1.9cm}
Pierre-Fran\c cois Rodriguez\footnote{Imperial College London, Department of Mathematics, London SW7 2AZ, UK, and Centre for Mathematical Sciences, University of Cambridge, Cambridge CB3 0WB, UK, \url{pfr26@cam.ac.uk}} and Wen Zhang\footnote{Imperial College London, Department of Mathematics, London SW7 2AZ, UK, \url{kate.zhang23@imperial.ac.uk}}

\end{center}
\vspace{0.1cm}
\begin{abstract}
\centering
\begin{minipage}{0.80\textwidth}
{We consider level-set percolation for the Gaussian free field $\varphi$ on $\Z^2$ regularised by a mass of order $N^{-2}$, and investigate the probability $\theta_N(a)$ that the cluster of the origin in $\{\varphi \geq  \sqrt{g_N} a\}$ reaches distance $N$ while remaining finite, where $g_N$ is the variance of $\varphi_0$. We prove a scaling law for $\theta_N(a)$ which is quantitative simultaneously in $a$ and $N$. It asserts that $\theta_N(a)=\theta_{\xi}(a_*)(N/\xi)^{-\vartheta+o(1)}$ as $N/\xi\to\infty$, where the correlation length $\xi$ diverges exponentially as $a$ approaches the critical level $a_*$, the critical one-arm probability $\theta_{\xi}(a_*)$ decays fractional-logarithmically, and the exponent $\vartheta$ is explicit and given by a variational problem involving the logarithmic energy. Thus $\theta_N(a)$ crosses over from fractional logarithmic decay at and near criticality ($N\lesssim\xi$) to polynomial decay away from it ($N\gg\xi$), in stark contrast with dimension three and higher, where similar observables exhibit polynomial decay at and near criticality and rapid decay away from it. The proof relies on a coarse-graining scheme capturing a polynomial rate with sharp exponent, and on sharp logarithmic capacity estimates for planar ``porous'' sets.}
\end{minipage}
\end{abstract}
\thispagestyle{empty}

\vspace{-0.4cm}

{\setcounter{tocdepth}{1}
\tableofcontents}

\vfill
\begin{flushright}

%\today
\bigskip
\end{flushright}

\newpage

\section{Introduction}

A common principle in the theory of critical phenomena \cite{fisher1972scaling} asserts that, near a continuous phase transition, an observable $\theta_N(a)$ of a system of size $N$ at control parameter $a$ is well approximated by
\begin{equation}\label{eq:scaling-normal}
\theta_N(a) \approx N^{-\sigma}f (N/\xi),
\end{equation}
where $\sigma>0$, $f$ is rapidly decaying, and the correlation length $\xi=\xi(a)$ diverges polynomially as $a$ approaches the critical point $a_*$. Thus $\theta$ decays polynomially at and near criticality ($N\lesssim\xi$), and rapidly in the off-critical regime $N\gg\xi$. For level-set percolation of the Gaussian free field (GFF) and related models with long-range correlations in dimension three, scaling of this type has recently been established at various levels of precision \cite{goswami_radius_2022,muirhead_percolation_2024-1,prevost_first_2025,GRS24.1,GRS24.2,ding_percolation_2020,drewitz_arm_2023}, with $\sigma=\tfrac12$ for GFF on the cables \cite{drewitz_critical_2025,cai_one-arm_2024}.

In the present article we investigate the same question in dimension two, the lower-critical dimension for the GFF. We prove that \eqref{eq:scaling-normal} must then be replaced by
\begin{equation}\label{eq:scaling-new}
\theta_N(a) \approx (\log  N)^{-1/2}(N/\xi)^{-\vartheta},
\end{equation}
with an explicit exponent $\vartheta\in(0,\infty)$ and a correlation length $\xi$ diverging \emph{exponentially} near $a_*$. Thus the nature of the transition changes: polynomial decay no longer characterises critical but rather off-critical behaviour, while the former exhibits \emph{fractional logarithmic} decay, and rapid decay is absent altogether.

\subsection{Model and main result}\label{sec:intro_model}
We now describe our results, and introduce to this end the two-dimensional analogue of the model studied in \cite{rodriguez_phase_2013} for $d\geq3$. Since the planar random walk is recurrent, the free field on $\Z^2$ requires a cut-off at large distances, which we implement through a mass. For $N\geq1$, let $g_N(\cdot,\cdot)$ denote the Green's function of the continuous-time simple random walk $X$ on $\Z^2$ with unit jump rate and killing rate $N^{-2}$, see Section~\ref{sec:facts}, and let $\P_N$ be the law of the centred Gaussian field $\varphi=(\varphi_x)_{x\in\Z^2}$ with covariance $g_N$. Its expectation is $\E_N[\cdot]$. The mass has the benefit of preserving translation invariance, as well as the other lattice symmetries of $\Z^2$, and together with the positive association of $\P_N$,  gives immediate access to various classical tools of planar percolation; cf.~ in particular~\cite{kesten_percolation_1982, gandolfi_uniqueness_1988, kohler-schindler_crossing_2023} which will come in handy (see for instance the proof of Proposition~\ref{thm:crit_pt}). By \cite[Theorem~2.1]{rz25a} applied with $h_N=1$, one knows that
\begin{equation}\label{eq:g_N}
g_N \stackrel{\text{def.}}{=} g_N(0,0) = \E_N[\varphi_0^2]\sim \tfrac2{\pi} \log N, \quad \text{as $N \to \infty$}
\end{equation}
(the walk in \cite{rz25a} has jump rate $\frac23$ and killing rate $1$, which accounts for minor adjustments when applying results from \cite{rz25a}; e.g.~the constant in \eqref{eq:g_N} is $\frac23\times\frac3\pi$, cf.~\cite[(1.6)]{rz25a}). It is natural to measure heights in units of the typical fluctuations $\sqrt{g_N}$, and we write $\bar q=\sqrt{g_N}\,q$ for $q\in\R$. For $A,B\subset\Z^2$, we write $A\stackrel{\geq a}{\longleftrightarrow}B$ if a nearest-neighbour path in the excursion set $\{\varphi\geq a\}$ connects $A$ and $B$, and $0\stackrel{\geq a}{\longleftrightarrow}\infty$ if the cluster of $0$ in $\{\varphi\geq a\}$ is unbounded. The order parameter of the model is the percolation probability $\theta_\infty(a)\coloneqq\P_N(0\stackrel{\geq \ba}{\longleftrightarrow}\infty)$. As $a\mapsto\theta_\infty(a)$ is non-increasing, the critical level
\begin{equation}\label{eq:a_*}
a_* \defeq \inf\big\{a\in \R:\theta_\infty(a)=0\big\}
\end{equation}
(with the convention $\inf \emptyset= \infty$)
is well defined. We show in Proposition~\ref{thm:crit_pt} that the transition is non-trivial for every $N$, and that, for some absolute constant $\Cl[c]{C:a*}\in(0,\infty)$,
\begin{equation}\label{eq:a_*-bound}
-\Cr{C:a*} \leq \ba_* \leq 0 \quad \text{for all $N\geq 1$};
\end{equation}
in particular $a_*\to0$ as $N\to\infty$ (recall that $a_*$ depends implicitly on $N$), at rate at most $c(\log N)^{-1/2}$. It is also a straightforward matter (see Section~\ref{sec:nearcrit}) to see that $\theta_\infty(a_*)=0$, i.e.~there is no percolation at criticality, and by \cite{gandolfi_uniqueness_1988} the infinite cluster is unique when it exists.

\medskip
\noindent\textbf{Main result.} We consider the truncated one-arm probability (with $\ba=\sqrt{g_N} a$)
\begin{equation}\label{eq:connectivity_function}
\theta_R(a) \stackrel{\text{def.}}{=} \P_N\big( 0\stackrel{\geq \ba}{\longleftrightarrow}\partial B_R , \,
0 \hspace{-1ex}  \stackrel{\quad \geq \ba}{\mathrlap{\longleftrightarrow}{\mkern15mu{\text{\scriptsize{/}}}}}  \,  \infty \big), \quad a \in \R,\, R,N \geq 1,
\end{equation}
i.e.~the probability that the cluster of the origin in $\{\varphi\geq \ba\}$ reaches the boundary of the Euclidean ball $B_R$ of radius $R$ while remaining bounded (note that $\theta_R(a)$ depends on $N$ through $\P_N$, which we keep implicit). Thus $\P_N(0\stackrel{\geq \ba}{\longleftrightarrow}\partial B_R)=\theta_R(a)+\theta_\infty(a)$ for all $R$. ``Truncation'' simply refers here to the presence of the disconnection event in \eqref{eq:connectivity_function}. In particular, it is immaterial for $a\geq a_*$, where $\theta_\infty(a)=0$. 

In order to state our main result we introduce
the zeroth-order modified Bessel function
\begin{equation}\label{eq:K_0}
    K_0(s)=   \int_{0}^\infty \frac{1}{2 t}
    \exp\Big\{-\frac{s^2}{4t}-t\Big\} \,dt, \quad s > 0.
\end{equation}
and, for a probability measure $\mu$ on $[0,1]$, let
\begin{equation}\label{eq:energy-form}
    \mathscr{E}(\mu) =\frac{ 4}{\pi} \iint_{[0,1]^2} K_0\big( 2|x-y|\big) \,\mu(dx)\,\mu(dy).
\end{equation}
One can check (see Section~\ref{sec:tube} for details) that $(x,y)\in \R^2 \mapsto K_0(2|x-y|)$ is a positive definite kernel %(indeed it is up to negligible pre-factors the transition density of a killed two-dimensional Brownian motion), 
and hence that $\mathscr{E}(\cdot)$ is a positive definite quadratic form.
Let $\log^*(x) = \max\{ 1 , \log x\}$ and $\Cl[c]{gNc}= \inf_N \sqrt{ \log^*(N)/g_N} \in (0,\infty)$, cf.~\eqref{eq:g_N}.
\begin{theorem}\label{T:main}
    For all $N\geq 1$ and $a\in\mathbb{R}$ with $|a-a_*|\leq \Cr{gNc}$ and $a_*$ given by \eqref{eq:a_*}, let
    \begin{equation} \label{def:correlation_length}
    \xi%=\xi(a,N)
    \defeq  N\exp\{-|a-a_*|^2g_N\}.
    \end{equation} 
    Then
    \begin{equation} \label{eq:thm-asymp}
        \lim_{({N}/{\xi})\to\infty} \frac{\log\left(\theta_N(a)/\theta_{\xi}(a_*)\right)}{\log(N/\xi)}= - \big(\inf_{\mu}\mathscr{E}(\mu)\big)^{-1},
    \end{equation}
    where the limit is over any sequence of pairs $(a,N)$ (possibly dependent) such that $
    |a-a_*|^2g_N\to \infty$, and the infimum ranges over all probability measures $\mu$ on $[0,1]$.
\end{theorem}
Theorem~\ref{T:main} follows immediately from Theorems~\ref{thm:upper_bound} and~\ref{prop:lower_bound} below, which provide separate upper and lower bounds on $\theta_N(a)$ quantitative in both $N$ and $a$.

\medskip
\noindent\textbf{Discussion.}
Let us denote by
\begin{equation}\label{eq:tau}
\vartheta \defeq \big(\inf_{\mu}\mathscr{E}(\mu)\big)^{-1}
\end{equation}
the exponent in \eqref{eq:thm-asymp}. It is non-degenerate: as $K_0$ is positive and decreasing, $\mathscr{E}(\mu)\geq\frac4\pi K_0(2)$ for all $\mu$, whence $\vartheta<\infty$; and choosing $\mu$ uniform, $K_0(2|x-y|)\leq|\log|x-y||+c$ (see \eqref{eq:bessel_ub}) and $\iint_{[0,1]^2}|\log|x-y||\,dx\,dy=\frac32$ give $\vartheta>0$. Numerically one finds that $\vartheta\approx 0.748$.

Theorem~\ref{T:main} gives a precise meaning to \eqref{eq:scaling-new}: as $N/\xi\to\infty$, $\theta_N(a)$ decays like $\theta_{\xi}(a_*)(N/\xi)^{-\vartheta}$, where the critical cost $\theta_{\xi}(a_*)$ is of fractional-logarithmic order: by \cite[Theorem 1.1]{rz25a} and Lemma~\ref{lem:near_critical}, it is comparable to $(\log^*(N/\xi)/\log^*N)^{1/2}$. In particular, for fixed $a$ with $|a-a_*|\leq\Cr{gNc}$,
\begin{equation}
\label{eq:easy-asymp}
\theta_N(a) = N^{- \frac2\pi \vartheta |a-a_*|^2 + o(1)}, \quad \text{ as } N \to \infty,
\end{equation}
a two-dimensional analogue of the main result of \cite{goswami_radius_2022} on $\Z^3$, consistent with its extensions in \cite{muirhead_percolation_2024-1}. Unlike these works, which fix $a\neq a_*$ and let $N\to\infty$, Theorem~\ref{T:main} is quantitative \emph{simultaneously} in $a$ and $N$, and covers the whole crossover regime $|a-a_*|\to0$ with $N/\xi\to\infty$.

The correlation length $\xi=N^{1-\frac2\pi(a-a_*)^2+o(1)}$ depends on the distance to criticality through a Gaussian factor in the natural variables $\ba,\ba_*$ rather than through a power law; it remains comparable to $N$ in a window of width $O((\log N)^{-1/2})$ around $a_*$, where $|a-a_*|\asymp(\log\xi)^{-1/2}$. This confirms rigorously the physics prediction of \cite{cao_level_2021} of an exponentially diverging correlation length, and Theorem~\ref{T:main} goes beyond it by describing the genuinely off-critical regime $N/\xi\to\infty$.

All our results remain valid for the continuous extension $\tilde\varphi$ of $\varphi$ to the cable system $\tilde\Z^2$, for which $\tilde a_*<0$, see Remark~\ref{rem:cableBara}; our arguments are written for the discrete model and we indicate the minor modifications needed along the way.

The methods also determine the arm exponent at every macroscopic scale. With $\vartheta(\varepsilon)\coloneqq(\inf_\mu\mathscr{E}(\mu))^{-1}$, where the infimum now ranges over probability measures on $[0,\varepsilon]$, $\varepsilon\in(0,1]$ (so that $\vartheta(1)=\vartheta$), one can show by inspecting the proof that
\begin{equation}\label{eq:profile}
\lim_{(N/\xi)\to\infty}\frac{\log\big(\theta_{\varepsilon N}(a)/\theta_{\xi}(a_*)\big)}{\log (N/\xi)}=-\vartheta(\varepsilon).
\end{equation}
As $\vartheta(\varepsilon)\sim\frac{\pi/4}{\log(1/\varepsilon)}\to0$ when $\varepsilon\to0$, the cost of the arm is \emph{concentrated at the top scales}: the cluster travels mesoscopic distances essentially for free and only pays when approaching scale $N$. In dimension three, by contrast, the cost to reach $\varepsilon N$ is (up to logarithmic corrections) a fraction $\varepsilon$ of the total, governed by the capacity of a segment of length $\varepsilon N$ \cite{goswami_radius_2022}. This hints at the completely different coarse-graining design underlying the proof, which is explained below.

\subsection{Ideas of the proof}\label{sec:intro_proof}
Coarse-graining techniques inspired by higher dimensions, see e.g.~\cite{zbMATH07704060} for the state-of-the-art, are an important tool, but in two dimensions they must be implemented differently, since the target is a \emph{polynomial} rate with an exact exponent, rather than (stretched-)exponential decay. Owing to this, the coarse-graining we describe below has a \emph{critical} nature, in contrast with the \emph{off-critical} coarse-graining of e.g.~\cite{zbMATH07704060}.

One important message here is that these techniques \emph{can} be successfully developed at all. Indeed, in two dimensions, decoupling typically forces a huge separation between local events, and combining them with a naive renormalisation scheme (e.g.~as in~\cite{rodriguez_phase_2013}) (which for $a>a_*$ would for instance replace the entire path by a fractal ``carpet'') is much too wasteful to deduce a sharp result like \eqref{eq:thm-asymp}. The remedy is a bootstrap argument with very carefully chosen scales, see \eqref{eq:scales} below, which in a single round loses almost no information on the connectivity event defining $\theta_N(a)$ (cf.~\eqref{eq:connectivity_function}). In a certain sense there is \emph{very little} (and careful) coarse-graining. We describe this in more detail below. 

\smallskip
\noindent\textbf{Near-critical input as starting point.} The upper bound (Theorem~\ref{thm:upper_bound}) involves three scales
\begin{equation}\label{eq:scales}
\xi \ll L \ll ML,
\end{equation}
chosen quantitatively in terms of $a$ and $N$, see \eqref{eq:ML_choice_ub}. Conditioning on the field at scale $L$ exhibits the critical cost $\theta_{\xi}(a_*)$, which is supplied by the near-critical estimate of Lemma~\ref{lem:near_critical} (building on \cite{rz25a,bi_2026_onear}). This decoupling must be \emph{multiplicative}: since $a$ may tend to $0$ as $N\to\infty$, even an additive error that's exponentially small in $|a-a_*|^2g_N$ would swamp the critical cost.  
Lemma~\ref{lem:near_critical} also shows that $\xi$ is maximal such that the near-critical arm probability is comparable to the critical one, which identifies $\xi$ as the correlation length.

\smallskip
\noindent\textbf{A coarse-graining that sees a polynomial rate.} Beyond scale $L$, the connection is coarse-grained into a family of boxes of side length $L$ at mutual distance at least $ML$ (Proposition~\ref{prop:coarse_graining}). The scheme must be {parsimonious}, since the entropy of many boxes would overwhelm a polynomially small cost: only $n=N/(ML)$ boxes are used, a bounded number when $|\ba-\ba_*|^2$ is a large constant, in contrast with higher dimensions where the number of boxes diverges. Yet it must be \emph{rich enough}, since a single box only contributes a critical cost, and it is the union of the few boxes that produces the full exponent. The resulting bound,
\begin{equation}\label{eq:3terms}
 \P_N\big(
    \partial B_{L}\stackrel{\geq \ba}{\longleftrightarrow} \partial B_N
    \,\big|\,\mathcal{F} \big) \leq a_1 + a_2 +a_3, \qquad a>a_*,
\end{equation}
with $\mathcal F$ the field on $\partial B_L$, balances \emph{three} terms: a local term $a_1$ (many independent crossings at scale $L$), and two deviations of harmonic averages, $a_2$ due to the anchoring at scale $L$ and $a_3$ due to the decoupling of the boxes. The latter are two-dimensional, non-large-deviation analogues of the estimates of \cite{sznitman_disconnection_2015} (cf.~the contrast between disconnection in two \cite{lsw01} and higher \cite{zbMATH07704060} dimensions); their sharp variance bound, which is where $\vartheta$ enters, is proved in Section~\ref{sec:harm_avg}. To our knowledge, this is the first instance of a  coarse-graining that detects a sharp polynomial rate.

\smallskip
\noindent\textbf{Disconnection and multi-scale porous sets.} A genuinely new feature arises in the supercritical regime $a<a_*$. There, on account of \eqref{eq:connectivity_function}, the cluster of the origin must reach distance $N$ while \emph{remaining finite}, i.e.~it must be disconnected from infinity --- a disconnection problem with no counterpart in the symmetric case $\tilde a_*=0$ of the cable system of $\Z^3$ considered in \cite{drewitz_critical_2023}, where the supercritical regime can be reduced to the subcritical one by symmetry. We handle it through a second coarse-graining (see Lemma~\ref{lem:disc}, its proof in Section~\ref{sec:disc}, and Figure~\ref{fig:cg_disc}). Among the scale-$L$ boxes which witness a crossing, those with an atypical harmonic average are retained directly; in the remaining ones, the finiteness of the cluster forces a large portion of them to be disconnected from $\partial B_{kN}$ for $k \gg 1$, and we coarse-grain this disconnection further by extracting a sub-collection of much smaller balls of radius $L'\ll L$ with poor connectivity and controlled capacity, which together obstruct the connection. This follows the approach, initiated by Sznitman \cite{sznitman_disconnection_2015} for the disconnection of a box and extended to general sets by Nitzschner and Sznitman \cite{nitzschner_solidification_2020}, of coarse-graining a disconnection into a porous family of local obstacles; see also \cite{zbMATH07704060}, where these ideas are developed to high art in dimensions $d \geq 3$. A novelty here is that both coarse-grainings are combined within a single multi-scale construction: the resulting object is a hybrid, multi-scale ``porous set'', comprising pieces that encode both connection and disconnection, which is ultimately controlled through estimates for harmonic averages in which capacity plays a central role. 

\smallskip
\noindent\textbf{Sharp   capacity of multi-scale porous sets.} The capacity of the porous sets that arise (which is of logarithmic type) must be controlled up to the \emph{exact} constant in order to recover $\vartheta$ via the variational principle \eqref{eq:tau}. This is achieved in Proposition~\ref{prop:tube}, which is of independent interest. It is a planar, sharp-constant counterpart of the capacity bounds of \cite{goswami_radius_2022,prevost_first_2025} in transient setups, and sharpens the planar bound of \cite{rz25a} by keeping track of the proximity between blocks. By contrast, the solidification estimates of \cite{nitzschner_solidification_2020} (see also \cite{nitzschner_disconnection_2018,zbMATH07704060,chiarini_entropic_2020,chiarini_solidification_2026}), which were instrumental in pinning down the constants governing disconnection in dimension $d\geq3$, have a largely topological counterpart in the plane, where any set surrounding a box is hit by every path leaving it. We expect these methods to yield sharp macroscopic disconnection estimates in two dimensions, the planar counterparts of \cite{sznitman_disconnection_2015,li_lower_2014}.

\smallskip
\noindent\textbf{Lower bounds.} The lower bounds (Theorem~\ref{prop:lower_bound}) are obtained by a change of measure along a tube, in the spirit of \cite{drewitz_critical_2023,drewitz_arm_2023, goswami_radius_2022, GRS24.2}, but must retain the exact exponent and the critical cost up to constant factors. The main difficulty lies in the supercritical regime $a<a_*$. In \cite{drewitz_critical_2023}, which provides bounds of comparable strength (i.e.~quantitative in $a$, and not only in $N$) for the cable system of $\Z^3$, the symmetry $\tilde a_*=0$ allows one to reduce the supercritical regime to the subcritical one. No such reduction is available here: the cluster of the origin must be forced to reach distance $N$ and, simultaneously, to be disconnected from infinity, and it is the disconnection which carries the leading cost $(N/\xi)^{-\vartheta}$. Our strategy is loosely inspired by \cite{goswami_radius_2022, GRS24.2}, but these results hold in dimension $d \geq 3$ and apply to a fixed level $a \neq a_*$ only, i.e.~,they are \emph{not} quantitative in $a$. We decouple the two opposite connectivity requirements inherent to $\theta_N(a)$ when $a< a_*$ by means of the Markov property: the disconnection is enforced by a change of measure lowering the field in an annular region around a tube containing the segment from $0$ to $\partial B_N$, at a polynomial cost producing the exponent $\vartheta$, while the connection is constructed inside the tube for the field killed outside of it (Proposition~\ref{prop:lb_general}), at a cost which is of lower order.

In both sub- and super-critical regimes, exhibiting the critical cost $\theta_\xi(a_*)$ requires conditioning on the cluster of the origin inside a ball of radius $\xi_\epsilon\ll N$, cf.~\eqref{def:xi_delta}, which is the counterpart of the multiplicative decoupling in the upper bound. This decoupling produces new error terms, mirroring the multiple $a_i$'s in \eqref{eq:3terms}; fortunately, they turn out to be of lower order, see Lemmas~\ref{lem:cap_ll} and~\ref{lem:supcritCap1}. The harmonic averages arising in this decoupling are controlled by the critical estimates of Lemma~\ref{lem:critHarm}. In both regimes, the slow decay of correlations forces separation scales much larger than multiples of $\xi$. The various scales \eqref{eq:scales} involved in the upper bounds fittingly re-appear here. Planarity also helps: crossing-type estimates for loops/interlacements (Lemma~\ref{lem:Iu_closed_cont} and \eqref{eq:conChain}) replace the local uniqueness input of \cite{drewitz_critical_2023}.

\subsection{Related works and organisation}\label{sec:intro_related}
Beyond the three-dimensional results cited above, level-set percolation of the free field in $d\geq3$ has been studied extensively through the cable-system framework \cite{lupu_loop_2016,ding_percolation_2020,drewitz_cluster_2022}, leading to critical exponents \cite{drewitz_critical_2023}, one-arm exponents \cite{drewitz_arm_2023,drewitz_critical_2025,cai_one-arm_2024} and cluster volumes \cite{drewitz_cluster_2024,cai_incipient_2025}; coarse-graining methods akin to ours were developed in \cite{goswami_radius_2022,prevost_first_2025,GRS24.1}. The concentration of the cost at the top scales is mirrored in the geometry of clusters: in $d\geq3$, clusters conditioned to be large carry many pivotal edges at all scales \cite{cai_separation_2025,cai_heterochromatic_2025}, whereas in two dimensions critical cluster boundaries converge to non-touching $\mathrm{CLE}_4$ loops \cite{lupu_convergence_2019,sheffield_conformal_2012} and the four-arm estimates of \cite{gao_percolation_2024-2,bi_arm_2025} place pivotal points at dimension zero (see \cite{nolin_near_2008,duminil-copin_near_2021,duminil-copin_planar_2022} for planar Bernoulli and FK percolation). Further planar results include \cite{ding_percolation_2020,ding_chemical_2018} in the discrete setting and, in the continuum, \cite{schramm_contour_2009,aru_first_2020,jego_crossing_2023,aru_renormalised_2026}. Up-to-constants control of near-critical quantities, as provided by \cite{rz25a}, which interpolates between dimensions two and three, is imperative for Theorem~\ref{T:main}: the off-critical factor $(N/\xi)^{-\vartheta}=e^{-\vartheta|\ba-\ba_*|^2}$ may decay arbitrarily slowly in $N$ and would be obscured by multiplicative errors $N^{o(1)}$. Finally, \eqref{eq:scaling-new} exhibits logarithmic corrections at a critical dimension, as for the four-dimensional weakly  self-avoiding walk and $|\varphi|^4$ model \cite{bauerschmidt_logarithmic_2015,bauerschmidt_scaling_2014}, one-arm probabilities in dimensions four and six \cite{vanengelenburg_one-arm_2025}, the four-dimensional uniform spanning tree \cite{hutchcroft_logarithmic_2023} and long-range percolation \cite{hutchcroft_critical_2025}; here, however, the critical dimension is the \emph{lower} one. During the final stages of this project we learned that Yijie Bi \cite{bi_2026_onear} and collaborators were investigating related questions.

\medskip
\noindent\textbf{Organisation.} Section~\ref{sec:facts} collects preliminaries. Section~\ref{sec:nearcrit} proves Proposition~\ref{thm:crit_pt} and the near-critical estimates. Section~\ref{sec:ub} proves the upper bounds, using the disconnection estimate of Section~\ref{sec:disc}, the capacity bounds for porous sets of Section~\ref{sec:tube}, and the bounds on harmonic averages of Section~\ref{sec:harm_avg}. Sections \ref{sec:ub},~\ref{sec:disc} and \ref{sec:tube} form the technical core of this paper. Section~\ref{sec:lb} proves the lower bounds. Finally, Section~\ref{sec:harm_avg} gathers the bounds on critical and off-critical harmonic averages used throughout, including the sharp variance estimate responsible for $\vartheta$.

\medskip
In the sequel, $c,c',\dots$ denote positive constants which may change from line to line, and $c_i$ positive constants which remain fixed. Constants have no dependence unless explicitly stated.

\section{Useful facts} \label{sec:facts}

We gather here a few preliminary results concerning random walks and potential theory in two dimensions, including their extensions to the cable system, and the function $K_0(\cdot)$ from \eqref{eq:K_0}, that will be used throughout. In particular, we introduce the killed Green's function $g_N^U$ with killing on $U\subset \Z^2$, see \eqref{eq:gNK}, which extends $g_N=g_N^\emptyset$, review its large distance and large-$N$ behaviour, see \eqref{eq:g2-asymp-unif} and gather key estimates as to the effect of killing on $U$, see Lemma~\ref{L:RWestimates}. 

We then introduce the free field $\P_N^U$ with killing on $U$, together with the corresponding random interlacement process $\bar{\P}_N^U$ and loop soup $\mathbb{Q}_N^U$. The interlacement process comprises bounded excursions terminating either at the ``ghost'' vertex $\Delta$, or upon hitting $\partial U$ when $U\neq\emptyset$, while $\mathbb{Q}_N$ is the Poissonian ensemble of loops associated with the same killed diffusion. These three  objects are intimately linked through isomorphism theorems which we present towards the end of the section. 
We conclude with an RSW-type estimate for the interlacement process and loop soup which provides a gluing mechanism that chains local connections to form a continuous macroscopic
connection, see Lemma~\ref{lem:Iu_closed_cont} and \eqref{eq:conChain}.

\medskip
We start with some facts concerning the random walk $X_{\cdot}$ introduced above \eqref{eq:g_N}. Formally, this is the Markov chain on the weighted graph $(\Z^2, \lambda, \kappa)$ with generator $Lf(x)=\sum_y \lambda_{x,y} (f(y)-f(x))-\kappa_x f(x)$, where $\lambda_{x,y}=\tfrac14$ for neighbors $x,y \in\Z^2$, and $\kappa_x=N^{-2}$, where $\kappa_x$ is interpreted as jump rate to the cemetery state $\Delta$, which is absorbing, i.e.~$X_t=\Delta$ for all $t \geq H_{\Delta}$, where $H_K \coloneqq \inf\{ t \geq 0: X_t \in K \}$. We write $\lambda_x=\kappa_x+\sum_{y}\lambda_{x,y}=1+N^{-2}$ in the sequel.
The canonical law of $X$ with starting point $X_0=x \in \Z^2$ is denoted by $P_x$. It depends implicitly on $N$.

 For $K\subset \mathbb{Z}^2$, we introduce the equilibrium  measure of $K$, $e_{K,N}(x)= \lambda_x P_x(\tilde{H}_K=  \infty)$ for $x\in K$
where $\tilde{H}_K\coloneqq \inf\{t>0:X_t\in K 
\text{ and there exists } s\in (0,t) \text{ with } X_s\neq X_0\}$ is
the hitting time of $K$. The total mass of the equilibrium measure of $K$ is the capacity of $K$, denoted by $\capacity_N(K)$.
By a version of \cite[(1.57)]{sznitman_topics_2012} on infinite graphs, one has the last-exit decomposition, valid for finite $K\subset \mathbb{Z}^2$,
\begin{equation} \label{eq:last_exit}
    P_x(H_K< \infty)=\sum_{y\in \mathbb{Z}^2}g_N(x,y)e_{K,N}(y), \quad x \in \Z^2,
\end{equation}
where $H_K\coloneqq \inf\{t\geq0:X_t\in K \}$ is
the entrance time of $K$ and $g_N$ denotes the Green's function of $X$. By \cite[(3.9)]{rz25a}, one has that for all $1 \leq r \leq N$, denoting by $B(x,r)=\{z \in \Z^2: |z-x| \leq r\}$ the closed Euclidean ball of radius $r \geq 0$ and abbreviating $B_r=B(0,r)$,
\begin{equation}\label{eq:capB}
c \log \big((N/r)\vee 2\big)^{-1} \leq \capacity_N(B_r) \leq c^\prime \log \big((N/r)\vee 2\big)^{-1}.
\end{equation}
It follows from a slight adaptation of \cite[Proposition 2.5]{rz25a} applied to the case $h_N=1$ that $g_N$ satisfies the following: for all $\varepsilon \in (0,1)$ $k\geq 1$ and $ N \geq C(\varepsilon, k) $,
    \begin{equation}\label{eq:g2-asymp-unif}
%\lim_R
 \sup_{ |x-y|  \leq k N }\left| \frac{g_N(x,y)}{\bar g_N(x,y)}-1\right| < \varepsilon,
  %=0,
\end{equation}
where $|\cdot|$ denotes the Euclidean distance and
$$
\bar g_N(x,y) \defeq
\frac{2}{\pi}  K_0\left({2}\,\frac{|x-y| \vee 1}{N}\right), \quad \text{for } x,y \in \Z^2,
$$
and $K_0(\cdot)$ is given by \eqref{eq:K_0}. The discrepancy between $\bar g_N$ above and its counterpart in \cite[(2.21)]{rz25a} stems from the application of the local limit theorem in \cite[(2.23)]{rz25a}, for which the jump rate $r=1$ in the present context, instead of $r=2/3$ as in \cite{rz25a}. We frequently use (see \cite[Lemma B.1]{rz25a} for a proof) that
  \begin{equation} \label{eq:bessel_ub}
        \log(1/t) \leq K_0(t)\leq (\log(1/t)\vee 0)+c, \quad t >0,
    \end{equation}
which implies, in particular, that $K_0(t)\sim\log({1}/{t})$ as $t\to 0^+$. The following result will prove useful in several places.
Recall that $K_0(\cdot)$ is a decreasing function.
\begin{lemma}\label{lem:bessel_cont}
    For all $\epsilon\in(0,1)$ there exists $\Cl[c]{k0}(\epsilon)\in(0,1/2]$ such that for all $t\in(0,10]$ and $0<\zeta\leq \Cr{k0}t$,
    \begin{equation} \label{eq:bessel_cont}
    \frac{K_0(t+\zeta)}{K_0(t)}\geq 1-\epsilon \quad \text{and}\quad \frac{K_0(t-\zeta)}{K_0(t)}\leq 1+\epsilon .
    \end{equation}
\end{lemma}
\begin{proof}
    Since $K_0'=-K_1$ and $K_1(s)\leq c s^{-1}$ for $s>0$, see \cite[(10.29.3), (10.30.2) and (10.40.2)]{NIST:DLMF}, and $K_0(t)\geq K_0(10)\geq c$, we have
    \[
    \frac{K_0(t)-K_0(t+\zeta)}{K_0(t)}
    \leq c\int_t^{t+\zeta} K_1(s)\,\mathrm{d}s
    \leq c \int_t^{t+\zeta} \frac{\mathrm{d}s}{s}
    = c\log\Big(1+\frac{\zeta}{t}\Big)\leq c\,\frac{\zeta}{t},
    \]
    and similarly $\frac{K_0(t-\zeta)-K_0(t)}{K_0(t)}\leq c\log\big(\frac{1}{1-\zeta/t}\big)\leq 2c\,\frac{\zeta}{t}$ for $\zeta\leq t/2$. Both bounds are at most $\epsilon$ upon taking $\Cr{k0}(\epsilon)$ small enough.
\end{proof}

 We now consider killing on a set $U \subset \Z^2$ with finite boundary, which will be sufficient for our purposes. We write $P_x^U$ for the canonical law of the diffusion $X$ with starting point $X_0=x$ killed upon hitting $U$. 
 We denote the Green's function killed upon hitting $U$ by
    \begin{equation} \label{eq:gNK}
      g_N^U(x,y) \defeq\frac{1}{\lambda_{x}}E_{x}^U\Big[\int_0^{\infty} \1_{\{X_t=y\}} \,dt\Big]=      \frac{1}{\lambda_{x}}E_{x}\Big[\int_0^{H_U} \1_{\{X_t=y\}} \,dt\Big], \quad x,y \in \Z^2,
    \end{equation}
    so that $g_N= g_N^{\emptyset}$. Note that $ g_N^U(x,y)$ is no longer translation invariant. We write $\capacity_N^U(K):=\sum_{x\in K} e^U_{K,N}(x):=\sum_{x\in K} \lambda_x P_x^U (\widetilde{H}_K =\infty)$ for the corresponding capacity and equilibrium measure, i.e., the capacity and equilibrium measure associated to the process $X_{\cdot \wedge H_U}$. The following results are essentially obtained by specialising \cite[(5.2) and (5.4)]{rz25a} to the case $h_N=1$.

 \begin{lemma}\label{L:RWestimates}
   For all $N\geq 1$ and $2 \leq r \leq N$, $x_0 \in \Z^2$  the following hold:
   \begin{enumerate}
   \item[(i)] if $U= \Z^2 \setminus B(x_0,r)$, one has 
          \begin{equation} \label{eq:killed_green_smallh} 
          c   \log\left(\frac{r}{(|x-y|\vee 1) \wedge \tfrac r2 }\right)            \leq g_N^U(x,y) 
            \leq 
            c^\prime              \log\left(\frac{r}{(|x-y|\vee 1)  \wedge \tfrac r2 }\right), \quad x,y\in B(x_0,r/2); %1_{r \geq 2 h_N}.
          \end{equation}

 \item[(ii)] if $r \leq r^\prime \leq N$ and $U\subseteq \Z^2 \setminus B(x_0,2r^\prime)$ is such that $U\supseteq \partial B(x_0,2r^\prime)$ (the outer boundary of $B(x_0,2r^\prime)$; e.g.~$U=\partial B(x_0,2r^\prime)$ or $U=\Z^2\setminus B(x_0,2r^\prime)$),
        \begin{equation} \label{eq:cap-killed}
       {c}{\log({2r^\prime}/{r})}^{-1}
        \leq
        \capacity_N^U\big(B(x_0,r)\big) \leq
       {c^\prime}{\log({2r^\prime}/{r})}^{-1}.        \end{equation}
       \end{enumerate}
(The upper bound in \eqref{eq:cap-killed} holds for any $U\subseteq \Z^2 \setminus B(x_0,2r^\prime)$ by monotonicity of $\capacity_N^U$ in $U$; the condition $U\supseteq \partial B(x_0,2r^\prime)$ is needed for the lower bound, which fails e.g.~for $U=\emptyset$ and $r'\ll N$, cf.~\eqref{eq:capB}.)
\end{lemma}

We now come to extensions to the cable setup.
One naturally associates to $\Z^2$ a cable system $\tilde{\Z}^2$, which is formally attached to the weighted graph $(\Z^2, \lambda, \kappa)$ introduced above; see \cite[Section 2]{drewitz_cluster_2022} for details. The metric space $\tilde{\Z}^2$ is obtained by replacing each edge between neighbouring vertices of $\Z^2$ by a line segment of unit length, glued to the vertices at its endpoints, and adding a half open interval glued to each vertex of length $N^2$ (intuitively one can think of these intervals to be all glued together at the point $\Delta$, but it is more convenient not to perform this wiring). The random walk $X_{\cdot}$ extends to a diffusion on $\tilde{\Z}^2$, which is killed when reaching the endpoint of a half-open interval or escaping all finite sets, though the latter has vanishing probability within the current setup, and can thus be disregarded. With a slight (but in practice harmless) abuse of notation, we use $X_{\cdot}$ to denote this diffusion and $P_x$, $x \in \tilde{\Z}^2$, its law when $X_0=x$. The Green's function $g_N^U(\cdot,\cdot)$ extends to $\tilde{\Z}^2\times \tilde{\Z}^2$, and the above estimates, in particular, \eqref{eq:g2-asymp-unif} and those of Lemma~\ref{L:RWestimates}, admit obvious extensions, with $U$ now being allowed a closed, connected subset of $\tilde{\Z}^2$ with finite boundary $\partial U$; this will be sufficient for our purposes. 

For $U$ as in the previous paragraph, we write $\P_N^U$ for the law of the centred Gaussian free field $\varphi$ on ${\Z}^2$ with covariance $g_N^U(\cdot,\cdot)$, so $\varphi(x)=0$ for all $x \in U$. Hence $\P_N=\P_N^\emptyset$ refers to the probability measure introduced above \eqref{eq:g_N}. 
For $K\subset  \Z^2$ and $x\in \Z^2$, we introduce the almost sure decomposition
\begin{equation}\label{eq:field_decomp}
\eta^K(x)=E_x\big[\varphi({X_{H_K}})\1\{H_K<\infty\}\big] \quad\text{and}\quad \psi^K(x)=\varphi(x)-\eta^K(x), \quad x \in {\Z}^2.
\end{equation}
The strong Markov property asserts that if $\varphi$ is a sample of $\P_N^U$, then the fields $\eta^K$ and $\psi^K$ are independent centred Gaussian fields under $\P_N^U$, and $\psi^K$ has law $\P_N^{U\cup K}$. When no risk of confusion arises, we will also write arguments as subscripts, i.e. $\varphi_x=\varphi(x)$ etc.

The probability measure $\bar{\P}_N^U$ is the corresponding interlacement measure, which is a Poisson point process of labelled two-sided ${\Z}^2$-valued trajectories modulo time-shift, with intensity measure $ \mathrm{d}u \otimes \nu_N^U$, where $ \mathrm{d}u$ denotes Lebesgue measure on $[0,\infty)$, and $\nu_N^U$ is characterised by the property that the trajectories hitting a closed set $K\subset ({\Z}^2 \setminus U)$ has ``law'' 
\begin{equation}\label{eq:intensity}
 \int P_x^U(X \in \cdot\mid\tilde{H}_K=\infty) P_x^U(X \in \cdot) \,  \mathrm{d}  e_{K,N}^U(x),
\end{equation}
which is a finite measure of total mass $\text{cap}_{N}^U(K)$. The interlacement set $\mathcal{I}^u$ for $u>0$ is the set obtained as the union of all trajectories in a sample of $\bar{\P}_N^U$ with label at most $u$. Its distribution is characterised by the property that 
\begin{equation} \label{def:interlacement}
    \bar\P_N^U(\I^u\cap K=\emptyset)=e^{-u\capacity^U_N(K)}, \text{ for all compact }K\subset {\Z}^2\setminus U.
\end{equation}
%The corresponding occupation time field $(\ell_x^u)_{x\in{\Z}^2}$ accumulates the total local time of the (transient) trajectories with label at most $u$ at $x$, and so $\mathcal{I}^u = \{x\in \tilde{\Z}^2: \ell_x^u > 0 \}$. 
Note that within the present setup, the relevant trajectories comprising $\mathcal{I}^u$ are almost surely finite and their range is a subset of ${\Z}^2 \setminus U$. Their forward and backward parts exit  ${\Z}^2\setminus U$ through either cemetery state $\Delta$, or through $U$ (when $U \neq \emptyset$).

The measures $\P_N^U$ and $\bar{\P}_N^U$ are intimately linked. In particular, a version of Lupu's coupling~\cite[Theorem~3]{lupu_loop_2016} (see \cite[(Isom) and Theorem 1.1]{drewitz_cluster_2022} for a formulation on general graphs that applies in the present setup, because the condition \cite[(Cap)]{drewitz_cluster_2022} holds) allows one to couple the free field $\varphi$ with the interlacement set $\I^u$ so that, for each $u>0$, almost surely,
\begin{equation}\label{eqcouplingintergff}
    \{x\in\Z^2\setminus U:\varphi_x>\sqrt{2u}\}
    \subseteq \V^u ,
\end{equation}
with $\V^u=\Z^2\setminus \I^u$. 
By symmetry of the free field, we may equivalently choose the coupling so that
\[
    \I^u\subseteq\{\varphi\geq-\sqrt{2u}\}.
\]
The analogous coupling holds on the cable system, with $\varphi$ and $\I^u$ replaced by their continuous counterparts. We now introduce the corresponding Markovian loop soup. Let $P_{x,x}^{U,t}$ denote the bridge probability measure of the walk $X$ under $P_x^U$. The loop measure associated with the killed walk is the $\sigma$-finite measure,
\begin{equation}\label{eq:loop_measure}
    \mu_N^U(\,\cdot\,)
    \defeq
    \sum_{x\in\Z^2\setminus U}
    \int_0^\infty \frac{\mathrm{d}t}{t}\,
    P_x^U(X_t=x)\,
    P_{x,x}^{U,t}(\,\cdot\,).
\end{equation}
Thus $\mu_N^U$ is a measure on rooted time-parametrised loops
$\gamma:[0,T(\gamma)]\to\Z^2\setminus U$ with
$\gamma(0)=\gamma(T(\gamma))$. We let $\mathbb Q_N^U$ denote the law under which
$\mathscr L$ is a Poisson point process of such loops with intensity
$\frac12\mu_N^U$, and write $\mathbb Q_N=\mathbb Q_N^\emptyset$. Let $(\ell_x)_{x\in\Z^2\setminus U}$ denote the occupation field of $\mathscr L$, where $\ell_x$ is the total local time spent at $x$ by all loops in $\mathscr L$. By \cite[Theorem~1]{lupu_loop_2016}, one may couple $\varphi$ under $\P_N^U$ and $\mathscr L$ under $\mathbb{Q}_N^U$ so that
\begin{equation}\label{eq:lupu_loop_coupling}
    \ell_x=\frac12\varphi_x^2,\qquad x\in\Z^2\setminus U,
\end{equation}
and, moreover, the sign of $\varphi$ is constant on every cluster of loops in $\mathscr L$. In particular, every loop cluster is contained in a sign cluster of the free field. On the cable system, the coupling admits the stronger formulation that the clusters of the continuous loop soup coincide with the sign clusters of the continuous free field; see \cite[Proposition~2.1]{lupu_loop_2016}. 

We now consider the following events in association with loop soups. For any $x\in \Z^2$ and $A\subset  \Z^2$, we consider the events, 
\begin{align}
\label{eq:circuit_event}
    \Cir_{x,L}(A)&\defeq
\bigg\{
\begin{gathered}
\text{$\exists $ a connected path $\gamma\subseteq A$, s.t. $\gamma$ forms} \\
\text{ a closed circuit in $B(x,2L)\setminus B(x,L)$}
\end{gathered}
\bigg\},\\
\label{eq:LU_event}
    \Con_{x,L}(A)&\defeq \Cir_{x,2L}(A)\cap 
%\bigg\{
%\begin{gathered}
%   \text{$\exists$ $\mathscr{C}\subseteq A$ connected, such that $\mathscr{C}\cap  B(x,L)\neq \emptyset$}  \\
%   \text{and $\mathscr{C}\cap \partial B(x,4L)\neq \emptyset$} 
%\end{gathered}
%\bigg\}.
\big\{B(x,L)\stackrel{A}{\longleftrightarrow}\partial B(x,4L)\big\}.
\end{align}
We will also write $\Cir^\star_{x,L}(A)$ for the event analogous to $\Cir_{x,L}(A)$ in which $\gamma$ is required to be star-connected instead of connected. We first record a deterministic gluing consequence of the events $\Con_{x,L}(A)$, which also extends to the cable setting.  
For any $L\geq 2$, let $(x_1,\dots,x_m)\subset \Lambda(L) \coloneqq (\lfloor 2^{-1/2} L \rfloor\vee 1) \Z^2 $ be a star-connected chain in $\Lambda(L)$. For any $A\subseteq {\Z}^2$,
\begin{equation} \label{eq:conChain}
    \bigcap_{i=1}^m \Con_{x_i,L}(A)\subseteq 
    \bigg\{
    \begin{gathered}
    \exists\text{ a connected path $\gamma\subseteq A\cap \big(\textstyle \bigcup_{i=1}^{m} B(x_i,4L)\big)$ }\\
    \text{starting in $B(x_1,L)$ ending in $B(x_m,L)$}
    \end{gathered} 
    \bigg\}.
\end{equation}
We omit the easy proof. For any $A\subset \Z^2$, we let $\I^u_{A}$ denote the union of all trajectories in a sample of $\bar{\P}^U_N$ with label at most $u$ that are fully contained in $A$ (i.e.~ both forward and backward part of the trajectories are killed before exiting the set $A$).

\begin{lemma}\label{lem:Iu_closed_cont}

For all $100N\!\geq\! L\!\geq\! 4$, $K\!\subset\!{\Z}^2$ compact and connected such that $B_{4L}\subseteq {\Z}^2\setminus K$, one has
\begin{align}
    \label{eq:loop_closed_cont}
    &\mu_N^K\big(\Con_{0,L/2}\big(\text{range}(\gamma)\big)\big)\geq c, \quad \text{and}\\
    \label{eq:Iu_closed_cont}
    &\bar{\P}_N^{K}
    \big( \Con_{0,L/2}(\I^u) \big)
    \geq 1-\exp\{-cu\capacity^K_N(B_L)\},\quad u>0.
\end{align}
Further, if $L\in [cN,2cN]$,
\begin{equation}\label{eq:Iu_closed_cont_macro}    
\bar{\P}_N\big(
\Con_{0,L/2}(\I^{u}_{\tB_{4L}})
\big)\geq 1-\exp\{-cu\},\quad u>0.
\end{equation}
\end{lemma}

Lemma~\ref{lem:Iu_closed_cont} continues to hold on $\widetilde{\Z}^2$, for the corresponding loop soup and interlacement process, with $K\subset\widetilde{\Z}^2$ compact and connected, by the same proof. 

\begin{proof}[Proof of Lemma \ref{lem:Iu_closed_cont}]
We first record that, identifying $X$ with its range, for all $x\in B_{L/2}$,
\begin{equation}\label{eq:lb_X}
P_x^{K}\big(\Con_{0,L/2}(X)\big)\geq c,
\end{equation}
and that, if in addition $L\in[cN,2cN]$,
\begin{equation}\label{eq:lb_X_mac}
P_x\big(\Con_{0,L/2}(X), \, H_{B^c_{4L}}=\infty\big)\geq c.
\end{equation}
Estimates like \eqref{eq:lb_X}-\eqref{eq:lb_X_mac} are standard, and we only indicate the events to use to obtain the lower bounds. If $L<10^3$, one simply asks $X$ to follow a fixed path of bounded length realising the event. If $L\geq10^3$, let $L'=L/100$ and fix points $x_1,\dots,x_n\in B_{2L}\setminus B_L$, with $n\leq c$, such that $|x_i - x_{i+1 \,\text{mod}\, n}| < 2L'$, $B(x_i,10L') \subset B_{2L} \setminus B_L$ for $1\leq i \leq n$, and $\partial B_{3L/2}\subset \textstyle\bigcup_{i=1}^n  B(x_i,L'/2)$. Starting from $x\in B_{L/2}$, we require that $X$ (a) exits $B_{2L}$, thereby crossing from $B_{L/2}$ to $\partial B_{2L}$, before time $N^2$ and before being killed; (b) then reaches $\partial B_{3L/2}$, say in $B(x_{i_0},L'/2)$, before exiting $B_{4L}$, and leaves $B(x_{i_0},L')$ after having created a trajectory piece of capacity at least $c$ relative to $B(x_{i_0},10L')$; (c) then visits successively $B(x_{i_0+1},L'),\dots,B(x_{i_0+n-1},L')$ (indices modulo $n$), each before exiting the $10L'$-neighbourhood of the previous one; and (d) finally hits the trajectory piece created in (b) before exiting $B(x_{i_0},10L')$. On this event, the range of $X$ contains a crossing from $B_{L/2}$ to $\partial B_{2L}$ and a closed circuit in $B_{2L}\setminus B_L$, so that $\Con_{0,L/2}(X)$ occurs. By the strong Markov property applied at the successive stopping times, its probability is bounded from below by the product of $n+3$ conditional probabilities, each of which is at least $c$: for (a) by \cite[Proposition~2.4.5]{lawler_random_2010} (recall $L\leq 100N$), for the capacity requirement in (b) by \cite[Proposition~5.4]{rz25a}, and for the remaining hitting events by the analogue of \eqref{eq:last_exit} for the walk killed outside the relevant balls together with Lemma~\ref{L:RWestimates}. This yields \eqref{eq:lb_X}. If $L\in[cN,2cN]$, one additionally requires that $X$ does not exit $B_{4L}$ before being killed, which has conditional probability at least $c$ since the killing time is of order $N^2\asymp L^2$; this yields \eqref{eq:lb_X_mac}.

 We now proceed to show \eqref{eq:loop_closed_cont}, which will easily follow from \eqref{eq:lb_X}.
We consider the same decomposition of loops as in
\cite[(3.8)]{drewitz_arm_2023}. By retaining only loops with
$\kappa(\gamma)=1$ therein, that is, loops which make exactly one excursion from
$B_L$ to $B_{2L}^c$ before returning to $B_L$ and asking the trajectory to form the circuit on its way back to $B_L$, we have that, 
\begin{align*}
\begin{split}
    \mu_N^K\big(\gamma:\Con_{0,L/2}(\text{range}(\gamma))\big)&\stackrel{\phantom{\eqref{eq:lb_X}}}{\geq} \inf_{x\in B_{L/2}}P_x\big(\Con_{0,L/2}(X), H_{B_{L/2}}\circ H_{B_{2L}^c}<\infty\big)\\
    &\stackrel{\eqref{eq:lb_X}}{\geq} c\times \inf_{y\in B_{2L+1}}P_y(H_{B_{L/2}}<\infty)\geq c\times \frac{\log((N/L)\vee 2)}{\log((2N/L)\vee 2)}\geq c,
\end{split}
\end{align*}
where, for the second inequality, we applied the strong Markov property at the time $T$ when $\Con_{0,L/2}(X)$ is first realised (by construction, see (a)--(d) above, $X$ has exited $B_{2L}$ before $T$, and $X_T\in B_{2L+1}$) and used \eqref{eq:lb_X} together with $P_x\geq P_x^K$, and in the last line we used \eqref{eq:last_exit}, \eqref{eq:g2-asymp-unif} and \eqref{eq:capB}.

To conclude the proofs of \eqref{eq:Iu_closed_cont} and \eqref{eq:Iu_closed_cont_macro}, denote by $W_{\text{con}}$ the set of excursions modulo time-shift in the support of the intensity measure $\nu_N^{{ K}}$ (see above~\eqref{eq:intensity}) that cross $B_{2L}\setminus B_{L/2}$ and form a closed circuit in $\tB_{2L}\setminus \tB_L$. Likewise, let $W_{\text{con}}^{4L}$ be the set of excursions modulo time-shift in the support of $\nu_N$ that  cross $B_{2L}\setminus B_{L/2}$ , form a closed circuit in $\tB_{2L}\setminus \tB_L$, and never enter $\tB_{4L}^c$.
By asking the forward trajectory after entering $B_L$ to form the closed circuit, we bound $\nu_N^{{ K}}(W_{\text{con}})$ using \eqref{eq:intensity} and \eqref{eq:lb_X} by
\begin{equation}\label{eq:intensity_lb}
    \nu_N^{{ K}}(W_{\text{con}})\geq \int
    P_x^{{ K}}\left(
    \Con_{0,L/2}(X)
    \right)  \,  \mathrm{d}  e_{B_L,N}^{{ K}}(x)
    \geq c'\, \capacity_N^{{ K}}(B_L).
\end{equation}
To lower bound $\nu_N(W_{\text{con}}^{{4L}})$ when $L\in[cN,2cN]$, we need the backward trajectory to never enter $\tB_{4L}^c$. By the strong Markov property, we have with $\pi_{R}\coloneqq \lambda_x P_x(\tilde{H}_{B_L}=H_{B^c_{R}}=\infty)$, that for any $x\in \partial B_L$, 
 \begin{multline*}
    \pi_{4L}
    \geq
    \pi_{2L}+\sum_{y\in\partial B_{2L}}
    \lambda_x P_x(\tilde{H}_{B_L}> H_{B_{2L}^c},X_{ H_{B_{2L}^c}}=y )P_y(H_{B^c_{4L}}=H_{B_L}=\infty)\\
    \geq \pi_{2L}+ \bar c \lambda_x
    P_x(\tilde{H}_{B_L}> H_{B_{2L}^c}%H_{B_{2L}^c}<\infty
    ) \geq \bar c e_{B_L,N}^{B_{2L}^c}(x),
 \end{multline*}
 where in the second inequality we used that  $P_y(H_{B^c_{4L}}=H_{B_L}=\infty)>\bar{c}$, which follows by a similar argument as above \eqref{eq:lb_X_mac}. The last inequality then uses that $\bar{c}<1$ and $e_{B_L,N}^{B_{2L}^c}(x)=\lambda_xP^{B_{2L}^c}_x(\tilde{H}_{B_L}=\infty)= \lambda_x P_x(\tilde{H}_{B_L}>H_{B_{2L}^c})+\pi_{2L}$. Combining the above display with \eqref{eq:intensity}, \eqref{eq:lb_X_mac} and \eqref{eq:cap-killed} gives,
  \begin{align}\label{eq:intensity_mac_lb}
 \begin{split}
     \nu_N(W_{\text{con}}^{4L})
     &\geq\int
     P_x\left(
    % \begin{gathered}
     %\{
     H_{B^c_{4L}}=\infty,%\}\,\cap \\
     \Con_{0,L/2}(X)%
    % \end{gathered}
     \right) P_x(H_{B^c_{4L}}=\infty|\tilde{H}_{B_L}=\infty) \,  \mathrm{d}  e_{B_L,N}(x)
     \geq c.
 \end{split}
 \end{align}
By \eqref{eq:intensity_lb}, the number of interlacement trajectories at level $u$ that  crosses $B_{2L}\setminus B_{L/2}$ and forms a closed circuit in $\tB_{2L}\setminus \tB_L$ stochastically dominates a $\text{Poisson}$ variable with mean $cu\capacity^K_N(B_L)$, which yields \eqref{eq:Iu_closed_cont}. The bound \eqref{eq:Iu_closed_cont_macro} follows in the same way from \eqref{eq:intensity_mac_lb}.
\end{proof}

\section{Gap size and near-critical estimate}\label{sec:nearcrit}

We now investigate connectivity properties under $\P_N$, defined above \eqref{eq:g_N}.
The main result of this section is the following quantitative estimate on the value of the critical parameter  $a_*$ introduced in \eqref{eq:a_*}; cf.~the discussion around \eqref{eq:a_*-bound}.

\begin{proposition} \label{thm:crit_pt}
For all $N\geq 1$,
\begin{equation}\label{eq:range_crit_pt}
    a_*\in(- \Cr{C:a*}\log(1+N)^{-1/2},0 ],
\end{equation}
and the cluster of the origin in $\{ \varphi \geq \bar{a}_*\}$ is $\P_N$-a.s. bounded.
\end{proposition}
The proof of Proposition \ref{thm:crit_pt} is given in full in the present section, save for a bound on critical harmonic averages, see Lemma~\ref{lem:critHarm}, a result interesting in its own right, which is proved in Section~\ref{sec:harm_avg}. Our starting point, Lemma~\ref{lem:near_critical}, is an upper bound on the one-arm probability that assesses how the critical one-arm estimate (cf.~\cite[Theorem 1.1]{rz25a}) is affected when one is allowed to change the value of $a$ only `\textit{slightly}'.
Let
\begin{equation} \label{eq:f}
    f^L_r(a)\defeq \frac{a_+\vee \sqrt{\log(L/r)}}{\sqrt{\log(L)}}\wedge 1, \quad \text{$a\in\R$,} \quad \text{$ L\geq 2r\geq 2$}.
\end{equation}
\begin{lemma}\label{lem:near_critical}
For all $a\in \R$, $1\leq r\leq N$,
\begin{align}
\label{eq:near_critical_full}
&\P_N(
0\stackrel{{\varphi\geq-a}}{\longleftrightarrow}\partial B_r)
\leq cf_r^{2N}(a),\quad \text{ $1\leq r\leq N$;}
\\
\label{eq:near_critical_killed}
&\P_N(
0\stackrel{{\psi^L\geq-a}}{\longleftrightarrow}\partial B_r)\leq cf^{2L}_r(a),
\quad \text{$1\leq r\leq L/4,L\leq 100N$. }
\end{align}
\end{lemma}
Lemma~\ref{lem:near_critical} remains valid for the corresponding continuous fields, since any connection by $\{\tilde\varphi\geq -a\}$ induces a discrete connection by $\{\varphi\geq -a\}$ between the corresponding vertices. Together with Proposition~\ref{thm:crit_pt} and \cite[Theorem 1.1]{rz25a}, \eqref{eq:near_critical_full} determines the size of the critical cost: for all $N\geq1$ and $a$ with $|a-a_*|\leq\Cr{gNc}$ (so that $1\leq\xi\leq N$),
\begin{equation}\label{eq:crit_cost_size}
c\Big(\frac{\log((N/\xi)\vee 2)}{\log(N\vee 2)}\Big)^{1/2}\leq \theta_\xi(a_*)\leq c'\Big(\frac{\log((N/\xi)\vee 2)}{\log(N\vee 2)}\Big)^{1/2}.
\end{equation}
Indeed, the upper bound follows from \eqref{eq:near_critical_full} with $a=-\ba_*$ and $r=\xi$, together with \eqref{eq:range_crit_pt}. For the lower bound, as there is no percolation at $a_*\leq 0$, $\theta_\xi(a_*)=\P_N(0\stackrel{\geq \ba_*}{\longleftrightarrow}\partial B_\xi)\geq\P_N(0\stackrel{\tilde\varphi\geq 0}{\longleftrightarrow}\partial B_\xi)$, which is bounded from below by \cite[Theorem 1.1]{rz25a}. In particular, by \eqref{def:correlation_length} and \eqref{eq:g_N}, $\theta_\xi(a_*)\geq c|a-a_*|$ whenever $N/\xi\geq 2$. (The bound \eqref{eq:crit_cost_size} is not used in the proof of Proposition~\ref{thm:crit_pt}.)
The proof of Lemma \ref{lem:near_critical} is postponed for a few lines, and we first continue with gathering inputs to the proof of Proposition~\ref{thm:crit_pt}. The following is a simple consequence of Lemma \ref{lem:Iu_closed_cont}. 
Recall that $\bar{\P}^U_N$ denotes the law of the interlacement process with killing on $U$. For $L\geq 1$, abbreviate $\bar{\P}^L_N=\bar{\P}^{\partial B_L}_N$. 
% 
%In fact, the proof of \eqref{lem:aprioriDisc} uses only the weaker form of \eqref{eq:macro_disconnect} in which $\Con_{0,L/2}$ is replaced by $\Cir_{0,L}$.
%
\begin{lemma}\label{lem:aprioriDisc}
For %$k\geq 1$, 
 all $A\subset {\Z}^2$ connected and $u\geq \Cl[c]{c:lbu}$, 
\begin{equation} \label{eq:macro_disconnect}
    \bar{\P}_N\big( A \hspace{-1ex}  \stackrel{\quad \I^u}{\mathrlap{\longleftrightarrow}{\mkern15mu{\text{\scriptsize{/}}}}}  \,  \infty \big)\leq \exp\{-cu \,  \textnormal{$\diam(A)$}/N\}.
\end{equation}  
\end{lemma}
\begin{proof} 
We start with a simple scheme to coarse-grain the blocking interface of $A$.
Let $\cL^{\text{SA}}$ be the set of all self-avoiding loops $\gamma\subset {\Z}^2$. For each $\gamma\subset {\Z}^2$, let $\gamma_N$ consist of all points $=(z_1,z_2)\in  N'\Z^2$, where $N'=1\vee  \lfloor N/200\rfloor $, such that $\gamma \cap [z_1,z_1+ N')\times [z_2,z_2+N')\neq \emptyset $ and let $\gamma_N'$ be a subset of $\gamma_N$ of maximal size such that $B(z,N/10)$ are pairwise disjoint.
For $j\geq1$, let $\cL_{j,N}^A=\{\gamma_N:\gamma \in \cL^{\text{SA}}, |\gamma_N|=j \text{ and $A$ is contained in the bounded component of ${\Z}^2\setminus\gamma$}\}$. It's easy to see that $|\cL_{j,N}^A|\leq cj\times e^{c' j}\leq e^{c'' j}$ and that for a fixed $\gamma_N\in \cL_{j,N}^A $, one can choose $\gamma_N'$ in at most $e^{cj}$ ways. Moreover,  $\cL_{j,N}^A$ is empty unless $j \geq c \, \text{diam}(A)/N =  j_*$. Hence,
\begin{align} \label{eq:disconnect_peierls}
\begin{split}
    \bar{\P}_N\big( A \hspace{-1ex}  \stackrel{\quad \I^u}{\mathrlap{\longleftrightarrow}{\mkern15mu{\text{\scriptsize{/}}}}}  \,  \infty \big)
    &\leq 
    \bar{\P}_N\big(\exists \gamma\in \cL^{\text{SA}} \text{ s.t. } \gamma\subset \V^u \text{ and $A$ lies in the bounded component of ${\Z}^2\setminus\gamma$}\big)\\
    &\leq \sum_{j\geq j_*} e^{cj}\times \sup_{\gamma_N\in \cL_{j,N}^A} \sup_{\gamma_N'} 
    \bar{\P}_N\big(
    \partial B(z,\tfrac{N}{100}) \stackrel{\V^u}{\longleftrightarrow} \partial B(z,\tfrac{N}{50}) \text{ for all }z\in \gamma_N' 
    \big)
    .
\end{split}
\end{align}
Now for any $z\in\Z^2$, let $\I^u_z$ be the union of all trajectories $w$ in $\I^u$ such that $w\cap B(z,N/100)\neq \emptyset$ and $w\cap B(z,N/50)^{c}= \emptyset$. 
Since $B(z,N/10), z\in \gamma_N'$ are pairwise disjoint, we have by the independence property of Poisson processes and \eqref{eq:Iu_closed_cont_macro} and using the bound $|\gamma_N'|\geq cj$ that (cf.~\eqref{eq:circuit_event} for notation), 
\[
\bar{\P}_N\big(
    \partial B(z,\tfrac{N}{100}) \stackrel{\V^u}{\longleftrightarrow} \partial B(z,\tfrac{N}{50}) \text{ for all }z\in \gamma_N' 
    \big)
    \leq 
    \sup_{z\in \gamma_N'}
    \bar{\P}_N\left( \Cir^\star_{z,\frac{N}{100}}(\I^u_{\tB(z,N/10)})^c
    \right)^{cj} 
    \leq e^{-cuj}.
\]
Combining above display, \eqref{eq:disconnect_peierls} and the fact that $u$ is larger than a suitable absolute constant yields that the left side of \eqref{eq:macro_disconnect} is bounded by $\sum_{j\geq j_*}
    e^{cj}\times e^{-c^\prime uj} \leq e^{-cuj_*}$, as desired.
\end{proof}
We are now ready to give the:
\begin{proof}[Proof of Proposition \ref{thm:crit_pt}] 

We start with the upper bound
and suppose by contradiction that
$\{\varphi\geq 0\}$ percolates. Then, by
\cite{gandolfi_uniqueness_1988}, all star-connected components of
$\{\varphi<0\}$ are almost surely finite. On the other hand, by symmetry, $\{\varphi\leq0\}$ percolates with the same probability as
$\{\varphi\geq 0\}$,
this yields an infinite star-connected component of $\{\varphi<0\}$,
a contradiction.

    For the lower bound, consider $a>0$ and define $u=\bar{a}^2/2$ (recall that $\bar{a}=\sqrt{g_N} a$) and split $\I^{u}$ into three independent interlacements $\I^{u}_1,\I^{u}_2,\I^{u}_3$, all with the same law as $\I^{u/3}$. It follows from    \eqref{eqcouplingintergff}, see also \eqref{eq:circuit_event} for the second event below, that 
    \begin{equation}\label{eq:decomp-1arm}\theta_\infty(-a)
        \geq \bar{\P}_N(0\stackrel{\I^{u}_1}{\longleftrightarrow}\partial B_{ N/2}) \times
        \bar{\P}_N({ \Cir_{0,N/4}(\I^{u}_2)}) \times
        \bar{\P}_N(B_{N /4}\stackrel{\I^{u}_3}{\longleftrightarrow}\infty).\end{equation}
    It remains to lower bound the three probabilities above.
    For the first one, note that $P_0(H_{\partial B_{2N}}<\infty)>c$ by \eqref{eq:capB} and \eqref{eq:last_exit}. Combining this with \eqref{eq:intensity} and \eqref{eq:capB}, we infer that the $\nu_N$-mass of the trajectories starting from $(0,0)$ whose forward path reaches $\partial B_{2N}$ is at least $c\, \capacity_N((0,0))\geq c\log(N\vee 2)^{-1}$. Hence the first term on the right side of \eqref{eq:decomp-1arm} is at least $1-e^{-cu/\log(N\vee 2)} \geq 1- e^{-c'a^2}$ by \eqref{eq:g_N} and choice of $u$. By \eqref{eq:Iu_closed_cont_macro}, the second probability is larger than $1-e^{-cu}$ and it follows for $u \geq \Cr{c:lbu}$ from \eqref{eq:macro_disconnect} that the third probability is greater than $1-e^{-cu}$.
    Hence for any fixed $N$, $\theta_\infty(-a)>0$ whenever $ \Cr{c:lbu} \leq u = \bar{a}^2/2 \leq c'a^2\log(N\vee 2)$, which concludes the proof of \eqref{eq:range_crit_pt}.

    The fact that there is no percolation at level $\ba_*$ follows a classical argument,which we only outline. For a given $N\geq 1$, suppose that $\theta_\infty(a_*)>0$. The proof then proceeds in three standard steps, see e.g.~\cite{beffara_percolation_2017,muirhead_sharp_2020} for similar arguments in the context of planar Gaussian fields.
\begin{enumerate}[label=(\roman*)]
\item \emph{Finite-size criterion.} Apply the RSW theorem of \cite{kohler-schindler_crossing_2023} to the bond configuration $\omega$ on $\Z^2$ in which $\{x,y\}$ is open if and only if $\varphi_x\geq \ba_*$ and $\varphi_y\geq \ba_*$ (its clusters are, up to isolated vertices, those of $\{\varphi\geq\ba_*\}$) and to its planar dual $\omega^*$. Both laws satisfy the required assumptions: they are invariant under the symmetries of $\Z^2$ and positively associated, as their edge variables are monotone functions of $\varphi$. Combined with the fact that, by \eqref{eq:field_decomp} and the exponential decay of $g_N$ at scale $N$, the restrictions of $\varphi$ to regions at mutual distance $\rho\geq N$ are independent up to harmonic averages of variance at most $Ce^{-c\rho/N}$, the percolation of $\{\varphi\geq\ba_*\}$ yields that $\P_N(\Con_{0,R}(\{\varphi\geq\ba_*\}))\to1$ as $R\to\infty$, see \eqref{eq:LU_event}.
\item \emph{Continuity.} As $\Con_{0,R}(\cdot)$ only involves $\varphi|_{B_{4R+1}}$ and $\varphi_x$ has a bounded density, the probability in (i) decreases by at most $CR^2\delta$ when $\ba_*$ is replaced by $\ba_*+\delta$.
\item \emph{Renormalisation.} By a renormalisation and Peierls argument akin to \cite{rodriguez_phase_2013}, gluing the events $\Con_{x,R}(\cdot)$, $x\in\Lambda(R)$, by means of \eqref{eq:conChain} and decoupling them as in (i), it follows from (i) and (ii) that $\{\varphi\geq\ba_*+\delta/2\}$ percolates for suitable $R$ and $\delta=\delta(R)>0$, which contradicts the definition \eqref{eq:a_*} of $a_*$. \qedhere
\end{enumerate}

\end{proof}

\begin{remark}\label{rem:cableBara}
\begin{enumerate}[label=(\arabic*)]
    \item The second assertion of Proposition~\ref{thm:crit_pt} is only used to identify the truncated quantity $\theta_{\cdot}(a_*)$ from \eqref{eq:connectivity_function} with the one-arm probability $\P_N(0\stackrel{\geq \ba_*}{\longleftrightarrow}\partial B_{\cdot})$, which is controlled by \eqref{eq:near_critical_full} and \cite{rz25a}. If one does not insist on the truncation at criticality, i.e.~if the critical cost $\theta_{\xi}(a_*)$ in Theorems~\ref{T:main}, \ref{thm:upper_bound} and~\ref{prop:lower_bound} is replaced by $\P_N(0\stackrel{\geq \ba_*}{\longleftrightarrow}\partial B_{\xi})$, then the proofs below do not rely on it, and in particular do not require any RSW-type input.
\item For $\tilde a_*$, the analogue cable system critical parameter, we can in fact show the stronger result that 
\[\tilde a_*\in(-c\log(1+N)^{-1/2},0).\] 
Indeed,  
combining \cite[Prop.~2.4.5]{lawler_random_2010} and the Markov property, one readily argues that for all $x,y \in \Z^2$ at distance at least $N \geq 1$, one has $g_N(x,y) \leq Ce^{-c|x-y|/N}$. In particular $g_N$ decays exponentially in $|x-y|$ for fixed $N \geq 1$, hence the upper bound is a direct consequence of \cite[Theorem 4.1]{prevost_percolation_2023}.
The lower bound in \eqref{eq:range_crit_pt} follows from the preceding argument verbatim using the metric graph version of the interlacement coupling. For the discrete model, one also expects that $a_*<0$ for every $N\geq 1$. We do not pursue this here.
\end{enumerate}
\end{remark}

%
%\subsection{Proof of Lemma \ref{lem:near_critical}}\label{sec:proofNearCrit}
We now turn to the proof of Lemma~\ref{lem:near_critical}. Estimates of this type were already established in our companion paper \cite{rz25a} for the cable graph field $\tilde\varphi$ in the case $a=0$. In principle, Lemma~\ref{lem:near_critical} could be deduced from these results through a careful comparison between $\varphi$ and $\widetilde\varphi$; we do not pursue this approach here.
Instead, we derive Lemma~\ref{lem:near_critical} from \cite[Proposition~4.3]{bi_2026_onear}, which is proved for the massless discrete free field with Dirichlet boundary conditions.
The main additional ingredient is then a comparison between the massive and massless fields, based on a decomposition of the latter into the former plus an independent Gaussian correction (cf.~\eqref{eq:chi_good_probability} below). Lastly, some additional care is required in the present near critical setting, where one typically cannot afford additive errors, see \eqref{eq:fullNearCrit}, and also \eqref{eq:apply_near_crit} for a similar consideration in the proof of Theorem \ref{thm:upper_bound}.

The proof of Lemma~\ref{lem:near_critical} hinges on the following estimate for harmonic averages, which is also useful for later purposes.
In the sequel we write $t_+=\max\{t, 0\}$ for $t\in \R$. Recall the notation from \eqref{eq:field_decomp}.

\begin{lemma}[Critical harmonic average]\label{lem:critHarm}
For all $\lambda \geq 1$ and $D\subset U\subset \Z^2$, let $d=\inf_{y\in D,z\in U^c}|y-z|$. If $\diam(U)\leq 100N$ and $\diam(D)\leq \lambda d$, then we have that for all $t>0$,
\begin{equation} \label{eq:critHarm}
    \P_N\big(\sup_{x\in D}\eta^{U^c}_x\geq t \big)\leq c\lambda^2  \exp\bigg\{-\frac{c}{\log(\tfrac Nd\vee 2)}\left(t -c'\sqrt{\log(\tfrac Nd\vee 2)}\right)_+^2\bigg\},
\end{equation}
Additionally, there exists $\Cl[c]{tail_harm}<\infty$ such that when $t^2/\log(\tfrac Nd \vee 2)\geq \Cr{tail_harm}$, we have that for all $k\geq 1$
\begin{equation}\label{eq:critHarmExp}
    \E_N\big[\big(\sup_{x\in D}\eta^{U^c}_x \big)^k \1\big\{\sup_{x\in D}\eta^{U^c}_x\geq t\big\}\big]\leq c(k, \lambda )t^ke^{-{ct^2}/{\log(\frac Nd \vee 2)}}.
\end{equation}
\end{lemma}
The proof of Lemma~\ref{lem:critHarm} appears in Section~\ref{sec:harm_avg}. We now explain how to show Lemma \ref{lem:near_critical}.
In the notation from \eqref{eq:field_decomp}, for $L\geq1$, we abbreviate $\psi^L=\psi^{B_L^c}$ and let $\varphi^L$ denote the massless discrete  free field with zero boundary condition outside $B_L$. The only difference between the law of $\psi^L$ and $\varphi^L$ is therefore the mass: both have zero boundary condition outside $B_L$, while $\psi^L$ has killing rate $N^{-2}$ and $\varphi^L$ has no killing.
For the rest of the section, 
we tacitly enlarge the underlying probability space so that $\varphi, \psi^L,\varphi^L$ are defined jointly, and continue to denote the resulting probability measure by $\P_N$.

The following Lemma is essentially \cite[Proposition~4.3]{bi_2026_onear} without the restriction $k\geq N(\log N)^{-1}$ therein. As we explain now, such restriction can be easily removed.
\begin{lemma}\label{lem:BGL}
For all $a\in\R$ and $N\geq r\geq1$,
\begin{equation}\label{eq:BGL}
\P_N(
0\stackrel{{\varphi^{2N}\geq-a}}{\longleftrightarrow}\partial B_r
)
\leq c f_r^{2N}(a).
\end{equation}
\end{lemma}

\begin{proof}
The proof of \eqref{eq:BGL} when $ N/\log(N)\leq r\leq N/2$ is already provided by \cite[Proposition~4.3]{bi_2026_onear} with $h\leq 0$ therein.
We now claim that the same bound remains valid for all  $N^{1-c_0}\leq r\leq N/2$, for some $c_0\in(0,1)$. Indeed, in 
view of (4.30) and the display below (4.32) in \cite{bi_2026_onear}, the sole purpose of taking $k\geq N/\log(N)$ therein is so that the interval $ (
 \frac{c_7}{\log N}, \frac{c_6}{\log N-\log k})$ remains nonempty. Now by taking $c_0<c_6/c_7$ small enough, we have that $\frac{c_6}{\log N-\log k}> \frac{c_7}{\log N}$ for all $k\geq N^{1-c_0}$, hence the interval is still nonempty with this relaxed condition. The remainder
of the proof of \cite[Proposition~4.3]{bi_2026_onear} applies
without change.
If $r\!<\!N^{1-c_0}$, then
$\sqrt{\tfrac{\log(N/r)}{\log N}}\!\geq\!\sqrt{c_0}$, and \eqref{eq:BGL} follows
from the trivial upper bound by $1$. 
\end{proof}

We proceed by writing $\varphi^L$ as a sum of $\psi^L$ and an independent field. For $x\in \Z^2$ and $L\geq1$, let $\tilde{P}^L_x$ denote the transition density of the simple  random walk on $\Z^2$ killed upon exiting $B_L$ (without exponential killing) and let $g^L(\cdot,\cdot)$ be the corresponding Green's function. Let $g_N^L(\cdot,\cdot)=g_N^{B_L^c}(\cdot,\cdot)$, and
\begin{equation}\label{eq:tg}
\tilde{g}^L(x,y)
 \defeq g^L(x,y)-g_N^L(x,y) \stackrel{{\eqref{eq:gNK}}}{=}\int_0^\infty
\Big(1-\lambda_x^{-1}e^{-t/N^2}\Big)
\tilde{P}_x(X_t=y)\,d t
\end{equation}
(recall that $\lambda_x=1+N^{-2}$ for all $x\in B_L$).
 One needs to check that the matrix $(\tilde{g}^L(x,y))_{x,y\in B_L}$, is positive semidefinite, which, in view of \eqref{eq:tg}, is the case if $(\tilde{P}^L_x(X_t=y))_{x,y\in B_L}$ is positive semidefinite for all $t\geq 0$. The latter is readily checked using the Markov property at time $t/2$ and time reversal. %, we have that for all $t\geq 0$ and $f:B_L\mapsto \R$,
%\[\sum_{x,y\in B_L}f(x)f(y)\tilde{P}^L_x(X_t=y)=\sum_{z\in B_L}\Big(\sum_{x\in B_L}f(x)\tilde{P}^L_x(X_{t/2}=z)\Big)^2\geq 0.\]
We may therefore define a centred Gaussian field $(\chi^L_x)_{x\in B_L}$ that is independent of $\psi^L$, whose covariance is given by $\E_N[\chi^L_x\chi^L_y]=\tilde{g}^L(x,y)$ and
\begin{equation}\label{eq:massDecomp}
\varphi^L\stackrel{d}{=}\psi^L+\chi^L.
\end{equation}

\begin{lemma}
There exists $\Cl[c]{fieldsOsci}<\infty$ such that for all $2\leq L\leq 100N$,
\begin{equation}\label{eq:chi_good_probability}
\P_N\Big(
\sup_{x\in B_{L/2}}\chi_x^L \leq \Cr{fieldsOsci}
\Big)
\geq\frac12.
\end{equation}
\end{lemma}
\begin{proof}
We start by controlling 
$\E[(\chi^L_x-\chi^L_y)^2]^{1/2}$, for $x,y\in B_{L/2}$,
the ``canonical metric" on $B_{L/2}$ (for the field). The definition of $\chi^L$ gives,
\[
\E[(\chi^L_x-\chi^L_y)^2]=
(\tilde g^L(x,x)-\tilde g^L(x,y))+(\tilde g^L(y,y)-\tilde g^L(x,y)).
\]
Since $\lambda_x=\lambda\coloneqq 1+N^{-2}$ on $B_L$, writing $1-\lambda^{-1}e^{-t/N^2}=(1-\lambda^{-1})+\lambda^{-1}N^{-2}\int_0^t e^{-s/N^2}\,ds$ in \eqref{eq:tg} and applying the Markov property at time $s$ yields the resolvent identity
\begin{equation}\label{eq:tgMarkov}
\tilde g^L(x,y)=(1-\lambda^{-1})\,g^L(x,y)+N^{-2}\sum_{z\in B_L}g_N^L(x,z)\,g^L(z,y),\qquad x,y\in B_L.
\end{equation}
As $1-\lambda^{-1}\leq N^{-2}$ and $g_N^L\leq g^L$, it follows that for all $x,y\in B_{L/2}$ and $e\in\{(0,\pm1),(\pm1,0)\}$,
\[
\tilde g^L(x,y)-\tilde g^L(x,y+e)\leq N^{-2}\Big(\big|g^L(x,y)-g^L(x,y+e)\big|+\sum_{z\in B_L}g^L(x,z)\,\big|g^L(z,y)-g^L(z,y+e)\big|\Big).
\]
For $y\in B_{L/2}$ and $1\leq j\leq 2L$, let $A_j(y)=\{z\in B_L:j+1> |y-z|\geq j\}$. Since $A_j(y)\leq cj$ for all $y\in B_{L/2}$, we have by \cite[Lemma 6.3.3 and Proposition 6.3.5]{lawler_random_2010} that $g^L(x,z)\leq c\log(L)$ for all $x,z\in B_L$ and hence the sum on the right-hand side of the last display is bounded by $c\log(L)\!+\!\sum_{j=1}^{2L}\sum_{z\in A_j(y)}g^L(x,z)\,\big|g^L(y,z)\!-\!g^L(y\!+\!e,z)\big|$, which in turn is at most
\begin{align*}
%\sum_{z\in B_L}\!g^L(x,z)\,\big|g^L(y,z)\!-\!g^L(y\!+\!e,z)\big|
%&\leq c\log(L)\!+\!\sum_{j=1}^{2L}\sum_{z\in A_j(y)}g^L(x,z)\,\big|g^L(y,z)\!-\!g^L(y\!+\!e,z)\big|\\
& c\log(L)\!+\!c'\sum_{j=1}^{2L}j\times \log\Big(\frac{L}{\big|j\!-\!|x-y|\big|\!\vee 1\!}\vee 2\Big)\!\times\! j^{-1}\leq cL,
\end{align*}
where we also used the reverse triangle inequality $|x-z|\geq ||y-z|-|x-y||$ in the last bound. Combining this with the last display yields $\widetilde g^L(x,y)-\widetilde g^L(x,y+e)\leq cL/N^2$ for all $x,y\in B_{L/2}$. 
Now, for all $x\neq y\in B_{L/2}$, one can find a nearest-neighbour path $(x_0,x_1,\dots,x_{c|x-y|})$ with $x_0=y,x_{c|x-y|}=x$. It then follows from this and the last two displays that for all $x,y\in B_{L/2}$
\[
\widetilde g^L(x,x)-\widetilde g^L(x,y)\leq \sum_{i=1}^{c|x-y|}\widetilde g^L(x,x_{i})-\widetilde g^L(x,x_{i-1})\leq c{L|x-y|}/{N^2}.
\]
This implies $\E[(\chi^L_x-\chi^L_y)^2]^{1/2}\leq C\sqrt{|x-y|/N}$ for all $x,y\in B_{L/2}$.
Hence $B_{L/2}$ has diameter at most $100C$ in the canonical metric of $(\chi^L_x)_{x\in B_{L}}$, and for $\epsilon\in(0,100C]$ it is covered by $c\epsilon^{-4}$ balls of radius $\epsilon$ in this metric (centred on a grid of mesh $\epsilon^2N/C^2$, recall $L\leq 100N$). By Dudley's bound \cite[Theorem 1.3.3]{noauthor_random_2007}, $\E_N[\sup_{x\in B_{L/2}}\chi_x^L]\leq c\int_0^{100C}\sqrt{\log(c'\epsilon^{-4})}\,d\epsilon\leq c$.
    It remains to control $\sup_{x\in B_{L/2}}\E_N[(\chi^L_x)^2]$. For any $x,y\in B_{L/2}$, we have that by a similar calculation based on \eqref{eq:tg}, splitting the time-integral over the intervals $[0,|x-y|^2]$, $[|x-y|^2,L^2]$ and $[L^2,\infty)$ and using \cite[(5.13)]{rz25a},
    \cite[(5.18) and (5.23)]{barlow_random_2017} along with the inequality
    $1-e^{-x}\leq x$, that $\widetilde g^L(x,y) \leq c$.
    We may now obtain \eqref{eq:chi_good_probability} by the above two displays and an application of the Borell-TIS inequality (see, e.g.~\cite[Theorem 2.1.1]{noauthor_random_2007}).
\end{proof}

\begin{proof}[Proof of Lemma \ref{lem:near_critical}]
We start with \eqref{eq:near_critical_killed}. By \eqref{eq:massDecomp}, \eqref{eq:chi_good_probability}, and the fact that $\chi^L$ and $\psi^L$ are independent, we have that for all $1\leq r\leq L/4$, $L\leq 100 N$,
\begin{align*}
\begin{split}
\P_N\big(
0\stackrel{{\psi^L\geq-a}}{\longleftrightarrow}\partial B_r
\big)
&\leq
2\P_N\big(
\inf_{x\in B_{L/2}}\chi_x^L\geq -\Cr{fieldsOsci}
\big)\times \P_N\big(
0\stackrel{{\psi^L\geq-a}}{\longleftrightarrow}\partial B_r
\big)\\
&=
2 \P_N\big(
0\stackrel{{\psi^L\geq-a}}{\longleftrightarrow}\partial B_r,\inf_{x\in B_{L/2}}\chi_x^L\geq - \Cr{fieldsOsci}
\big)\leq 2\P_N\big(
0\stackrel{{\varphi^L\geq-a-\Cr{fieldsOsci}}}{\longleftrightarrow}\partial B_r
\big). 
\end{split}
\end{align*}
Combining this with \eqref{eq:BGL} gives $\P_N(
0\stackrel{{\psi^L\geq-a}}{\longleftrightarrow}\partial B_r
)\leq c f^{L}_r(a+\Cr{fieldsOsci})\leq c'f^{2L}_r(a)$, which concludes the proof of \eqref{eq:near_critical_killed}. We now turn to \eqref{eq:near_critical_full}. Recall the notation around \eqref{eq:field_decomp} and let $\eta^{4N}_N:=\sup_{x\in B_N}\eta^{B_{4N}^c}_x$. It follows from \eqref{eq:field_decomp} and \eqref{eq:near_critical_killed} that,
\begin{align}\label{eq:fullNearCrit}
\P_N(
0\stackrel{{\varphi\geq-a}}{\longleftrightarrow}\partial B_r
)
&\!\leq\!
\E_N\Big[
\P_N\big(
0\!\stackrel{
{\psi^{4N}\geq-a-\eta^{4N}_N}
}{\longleftrightarrow}\!\!\partial B_r
\big)
\Big]
\leq
c\frac{\E_N\big[
\big(
a+\eta^{4N}_N
\big)_+
\vee
\sqrt{\log(8N/r)}
\big]}{\sqrt{\log (8N)}}
.
\end{align}
Let's first consider the case when $a\geq 0$.
It follows from \eqref{eq:critHarmExp}
and the inequality $u\vee v\leq u+v$ for all $u,v\geq 0$ that the expectation on the r.h.s above is at most 
$a+\E_N[\eta^{4N}_N]+\sqrt{\log(8N/r)}$, which is bounded by $c+a+\sqrt{\log(8N/r)} \leq c(a\vee\sqrt{\log(8N/r)})$. One can conclude with the same reasoning that the expectation on the r.h.s above is at most $c\sqrt{\log(8N/r)}$ when $a<0$.
\end{proof}

\section{Upper bounds} \label{sec:ub}
This section contains Theorem~\ref{thm:upper_bound}, which is a central element of this article. 
It gives an upper bound on $\theta_N(a)$, quantitative in $a$ and $N$, from which follows the upper bound yielding half of \eqref{eq:thm-asymp} (that is, the analogue of \eqref{eq:thm-asymp} with a $\limsup$ instead of $\lim$ as $N/\xi \to \infty$). The argument appears in full in the present section, but it is contingent on several intermediate results, namely, Proposition~\ref{prop:coarse_graining}, along with Lemmas~\ref{lem:bad_box}, \ref{lem:harm_avg} and~\ref{lem:disc}, whose proofs are postponed to later sections. These key intermediate results revolve around the construction of a coarse-grained object for the cluster of the origin, when its diameter is large but finite. This object needs to be able to witness the transition from fractional logarithmic to (slow) polynomial decay of the form $N^{- c (a-a_*)^2}$, with $c= \tfrac2\pi \vartheta $ and possibly $|a-a_*|\ll 1$, cf.~\eqref{eq:easy-asymp}, as the diameter $N$ of the cluster gets large, when measured in units of the length scale $\xi$. 

The coarse-grained object we construct, whose key features are summarised in Proposition~\ref{prop:coarse_graining}, gives rise to three different types of harmonic averages for the  field. They are the key probabilistic quantities to control in order to derive the upper bound. Lemma~\ref{lem:critHarm} takes care of the tail behaviour of so-called \emph{critical} harmonic averages (i.e.~associated to the critical cost $\theta_{\xi}(a_*)$ below), whereas Lemmas~\ref{lem:bad_box}, \ref{lem:harm_avg}  deal with (anchored and unrooted) \emph{off-critical} harmonic averages, that contribute to the polynomial decay in $N/\xi$ below. Lemma~\ref{lem:disc} is specific to the supercritical regime, and we return to it below.\\

We start with the main result of this section. Recall the constant $\Cr{gNc}$ from above Theorem~\ref{T:main}.
\begin{theorem} \label{thm:upper_bound}
    For all $a$ with $|a-a_*|\leq \Cr{gNc}$ and $N\geq 1$, we have that with $\xi$ as in \eqref{def:correlation_length},
    \begin{equation} \label{eq:thm_ub}
        \theta_N(a)\leq c \theta_{\xi}(a_*)\exp\left\{-c'\,\log\big((N/\xi)\vee 2\big)\right\}.
    \end{equation}
    Furthermore, for all $\eta\in(0,1)$, there exists $\Cl[c]{ub}(\eta)<\infty$ such that if $| \ba - \ba_*|^2\geq \Cr{ub}$, with $\vartheta$ as in \eqref{eq:tau},
    \begin{equation} \label{eq:thm_ub_asymp}
    \theta_N(a)\leq \theta_{\xi}(a_*)\exp\left\{-(1-\eta)\vartheta \log(N/\xi)\right\}.
    \end{equation}
\end{theorem}

Actually, the arguments leading to \eqref{eq:thm_ub_asymp} and the
matching lower bound \eqref{eq:prop_lb_asymp} extend beyond the
near-critical case, namely to the regime
$|a-a_*|>\Cr{gNc}$, and yield the following extension of \eqref{eq:thm_ub_asymp} and \eqref{eq:prop_lb_asymp}.
For all $\eta\in(0,1)$, there exists $C(\eta)<\infty$ and $A(\eta)<\infty$ satisfying $A(\eta)\geq 1$ and $\lim_{\eta\to0}A(\eta)=\infty$ such that for all $a$ with $|a-a_*|\leq A(\eta)$ and $| \ba - \ba_*|^2\geq C$,
\begin{equation}\label{eq:asymp_ext}
\theta_{\xi\vee 1}(a_*)\exp\left\{-(1+\eta)\vartheta \log(N/\xi)\right\}\leq \theta_N(a)\leq \theta_{\xi\vee 1}(a_*)\exp\left\{-(1-\eta)\vartheta \log(N/\xi)\right\}.
\end{equation}
The argument when $|a-a_*|>\Cr{gNc}$ is considerably simpler.
Rather than giving the full argument, we illustrate in Remark~\ref{rem:asymp_ext} how our method can be adapted in one representative case.

\medskip
We start by establishing \eqref{eq:thm_ub} in the near critical regime, that is, when $| \ba - \ba_*|^2\leq \lambda$ for some (fixed) $\lambda \geq 1 $, or equivalently $N \leq e^{\lambda} \xi$, cf.~\eqref{def:correlation_length}. The remainder of \eqref{eq:thm_ub} follows directly from  \eqref{eq:thm_ub_asymp}, applied with $\eta = 1/2$ (say), and the near-critical result with $\lambda= \Cr{ub}(1/2)$.

\begin{proof}[Proof of \eqref{eq:thm_ub}, near-critical regime]
Let $\lambda \geq 1$ and assume that $N \leq e^{\lambda} \xi$. Together with Proposition~\ref{thm:crit_pt}, this implies that $|\bar{a}| \leq \lambda + \Cr{C:a*} = c(\lambda) $. Disregarding the disconnection event inherent to $\theta_N(a)$ (see \eqref{eq:connectivity_function}) and applying {\eqref{eq:near_critical_full}} yields that $\theta_N(a) \leq c(\lambda)(\log(N\vee2))^{-1/2}$. Since $\log(N/\xi)\leq\lambda$, the lower bound in \eqref{eq:crit_cost_size} gives $\theta_\xi(a_*)\geq c(\log(N\vee 2))^{-1/2}\geq c(\lambda)\theta_N(a)$, which concludes the proof.
\end{proof}
We now prepare for the outstanding proof of \eqref{eq:thm_ub_asymp}, which will be handled differently depending on the value of $a$. The subcritical regime, $a> a_*$, is treated in Section~\ref{sec:ub_sub}, while the supercritical regime, $a<a_*$, is treated in Section~\ref{sec:ub_sup}.

We start with an estimate and a calculation that provides a blueprint for how the critical cost $\theta_{\xi}(a_*)$ in Theorem \ref{thm:upper_bound} is produced. We will use this repeatedly in the sequel. Assume $ \Cr{gNc}\geq|a-a_*| > 0$ (so in particular $\xi\geq1$). We introduce a new length scale $L$ { along with a separation scale $M$} and assume that $L\geq10\xi,M\geq 100$ and $N\geq 100ML$. Our assumptions leading to $\Cr{ub}$ will typically make these constraints feasible. We stress that $L$ is really a ``new'' scale, i.e.~in practice it will neither be proportional to $\xi$, nor to $N$. Recalling the notation from \eqref{eq:field_decomp}, let $$\bar{\eta}^L_\xi\defeq \sup \big\{ \eta^{K}(x) : \, x\in \tB_\xi, \, K= \partial B_{L/2}\big\}$$ 
(note that $\bar{\eta}^L_\xi$ depends implicitly on $a$ and $N$ via $\xi$). 
Observe that by  \eqref{eq:critHarmExp} and the elementary inequality $|x+y|^k\leq c(k)(|x|^k+|y|^k)$,  we have that whenever $\ba^2/\log(N/L)\geq \Cr{tail_harm}$,
\[\E_N\big[ 
    (\bar{\eta}^L_\xi-\ba)^{k}\1\{\bar{\eta}^L_\xi>\ba\}\big] \leq c(k)\big(|\ba|^{k} +\E_N\big[ 
    (\bar{\eta}^L_\xi)^{k}\1\{\bar{\eta}^L_\xi\geq|\ba|\}\big]\big)\stackrel{\eqref{eq:critHarmExp}}{\leq} c'(k)|\ba|^{k}.\]
Hence by \eqref{eq:f}, under the assumptions leading to \eqref{eq:critHarmExp}, we have that
\begin{equation}\label{eq:apply_near_crit}
\E_N\big[ 
    f^L_{\xi}(\bar{\eta}^L_\xi-\ba)^k\big]^{\frac{1}{k}}
\leq 2 \frac{\E_N\big[ 
    (\bar{\eta}^L_\xi-\ba)^{k}\1\{\bar{\eta}^L_\xi>\ba\}\big]^{\frac{1}{k}}\vee \sqrt{\log(L/\xi)}}{\sqrt{\log(L)}}
%\sqrt{\frac{\log(L/\xi)}{\log(L)}} \vee\frac{\E_N\big[ (\bar{\eta}^L_\xi-\ba)^{2k}\1\{\bar{\eta}^L_\xi>\ba\}\big]^{\frac{1}{k}}}{\sqrt{\log(L)\log(L/\xi)}}  
\leq c(k)f^L_{\xi}(|\ba|).
\end{equation}
Set $\mathcal F\coloneqq \sigma(\varphi_x:\,x\in\partial B_{L/2})$.
Let $\mathsf E$ and $\mathsf F$ be events such that $\mathsf E$ is $\F$-measurable and $\mathsf F$ is independent of $\F$. We return to $\mathsf E$ and $\mathsf F$ shortly.
Applying the Markov property of the field with respect to $\F$, see below \eqref{eq:field_decomp}, together with a union bound, the Cauchy-Schwarz inequality and Lemma~\ref{lem:near_critical} with $r=\xi$, we obtain that, for all $L\geq 10\xi$, $N\geq 2L$, and $a$ such that $\ba^2/\log(N/L)\geq \Cr{tail_harm}$ and $ | \ba - \ba_*|^2\geq 1$,
\begin{equation} \label{eq:ub_blueprint}
\begin{split}
    \P_N\big(\{ 0\stackrel{\geq \ba}{\longleftrightarrow}\partial B_\xi \}\cap (\mathsf E\cup \mathsf F) \big)
%     = &\E_N\Big[ \P_N\big(0\stackrel{\geq \ba}{\longleftrightarrow}\partial B_\xi, \, (\mathsf E\cup \mathsf F)
%    \,\big|\,\mathcal{F} \big)\Big]\\
    &\stackrel{\text{Markov}}{\leq} \E_N\Big[ 
    \P_N^{{L/2}}\big( 0\stackrel{\geq \ba-\bar{\eta}_\xi^L}{\longleftrightarrow}\partial B_\xi \big)
    \P_N\big(
    \mathsf E\cup \mathsf F\,\big|\,\mathcal{F} \big)\Big]\\
    &\ \stackrel{\eqref{eq:near_critical_killed}}{\leq} c\,\left(\E_N\big[ 
    f^L_{\xi}(\bar{\eta}^L_\xi-\ba)^2\big]^{\frac{1}{2}}\P_N(\mathsf E)^{\frac{1}{2}}+\E_N\big[ 
    f^L_{\xi}(\bar{\eta}^L_\xi-\ba)\big]\P_N\big(\mathsf F \big)\right) \\
     &\ \stackrel{\eqref{eq:apply_near_crit}}{\leq} cf^L_\xi( |\ba-\ba_*|)\left(\P_N(\mathsf E)^{\frac{1}{2}}+\P_N\big(\mathsf F \big)\right),
\end{split}
\end{equation}
where in the last step we have also used Proposition~\ref{thm:crit_pt}{, \eqref{eq:g_N}} and $ | \ba - \ba_*|^2\geq 1$ to deduce that $|\tfrac\ba{\ba-\ba_*}| \leq 1+|\tfrac{\ba_*}{\ba-\ba_*}| \leq 1+{ \Cr{C:a*}\Cr{gNc}^{-1}}$, which implies  that $f^L_\xi(|\ba|) \leq cf^L_\xi(|\ba-\ba_*|)$ in view of \eqref{eq:f}.  \\

We come to the events $\mathsf{E}$ and $\mathsf{F}$, which will retain essential connectivity properties of the cluster when its diameter exceeds $N$ but remains finite. To deal with them, we apply a coarse-graining scheme whose key features for the proof are gathered in the next proposition. The goal of the scheme is control the entropy from all possible crossing paths in the annulus $B_N \setminus B_L$ formed by the cluster. Instead of keeping track of the whole path $\gamma$ in the annulus, we only retain partial information of $\gamma$ in a collection of well-separated balls of radius $L$.  

We first specify the setup for the balls that are used in the coarse-graining. We set for any $x\in \Z^2$,
\begin{align} \label{eq:various_boxes}
\begin{split}
    \bC_x^L=B(x,L), 
    \quad \tC_x^L=B(x,4L), 
    \quad \text{and}\quad
    \haC^{L,M}_x=B(x,ML),
\end{split}
\end{align}
where $B(x,r)$ refers to the set of points at  Euclidean distance at most $r$ from $x$. 
We will occasionally work with two pairs of scales $L,M$ and $L',M'$. When the choice is clear from the context, we abbreviate $\bC_x^{L}$ and $\bC_x^{L'}$ by $\bC_x$ and $\bC'_x$, and use the same convention for $\tC_x,\tC'_x$ and $\haC_x,\haC'_x$.
%We write $\tilde C_x^L,\tilde D_x^L$ and $\tilde U_x^{L,M}$ to denote the corresponding balls in the cable system and adopt similar shorthand notations. 
 Recall the renormalised lattice $\Lambda(L)=  (\lfloor 2^{-1/2} L \rfloor\vee 1) \Z^2 $ already in use in Section~\ref{sec:facts}. 
Let
\begin{equation}\label{def:C}
\begin{aligned}
     &  \text{$\mathcal{C} \subset \Lambda(L)$ and $\mathcal{C}' \subset \Lambda(L')$, $\mathcal{C}\cup \mathcal{C}' \neq \emptyset$, be collections of vertices}\\
     & \text{with mutual distances at least $16ML$ and $16M'L'$, respectively.}
\end{aligned}
\end{equation}
(There could be de facto only one collection). For a line $\ell\subset \R^2$, denote by $\pi_{\ell}:\R^2\to\ell$ the orthogonal projection onto $\ell$. For $x,y\in \ell$, $d_{\ell}(x,y)$ is the length of the unique line segment in $\ell$ joining $x$ and~$y$.
%The reader could choose to omit $B_{2\xi}$ from \eqref{def:C} for the purpose of the next result. In practice Proposition~\ref{prop:coarse_graining} will apply when $N \gg \xi$.

\begin{proposition}[Coarse-graining] \label{prop:coarse_graining}
   For all $L\geq 1$, $M\geq 100$, $N\geq 100ML$ 
    and $x\in \mathbb{Z}^2$, there exists a positive integer $n_*$ and a family $\mathcal{A}=\mathcal{A}^{L,M}_{x,N}$ 
    of collections $\cC=\bigcup_{i=1}^{n_*}x_i\subset B_N$ (with $n_*=|\cC|$) as in (\ref{def:C}) such that with $n= N/ML$:
    \begin{align} 
    &\label{A:entropy}
        \begin{array}{l}   |\mathcal{A}| \leq \exp\left\{\Cl[c]{entropy_A} n\log(N/L)\right\}, \end{array}\\[0.3em] 
    &\label{A:cardinality}
        \begin{array}{l}  
        n_*\in [cn,c'n], 
        \end{array}  \\[0.3em] 
     & \label{A:connectivity}
        \begin{array}{l} 
        \text{for any %connected  %continuous 
        path $\gamma\subset\tilde\Z^2$ from $\bC_x^L$ to $B(x,N)^c$, there exists}\\
        \text{$\cC\in \A$ such that $\gamma$ crosses $D_{x_i}^{L}\setminus C_{x_i}^L$ for all $i\in\{1,\dots,n_*\}$}
        \end{array} \\[0.3em] 
    &\label{eq:A_structure}
        \begin{array}{l}
         \text{there exists a line $\ell_*\subset \R^2$ such that $x \in\ell_*$ and for all $ 1\leq i\leq n_*$,}\\
         \text{$d_{\ell_*}\big(\pi_{\ell_*}(x_i),\pi_{\ell_*}(B(x,L))\big)\geq 16MLi$.}
        \end{array}
    \end{align} 
    Furthermore, for all $\eta\in(0,1)$, there exist  $\Cl[c]{cg_rho}(\eta)>0$ and $\Cl[c]{NML}(\eta) \in [1,\infty)$ such that the following hold. Let $N\geq \Cr{NML}ML$ and $\rho \in(0,\Cr{cg_rho}]$, and take any $\cC\in\mathcal{A}$ and $\tilde\cC\subset \cC$ with $|\tilde\cC |\geq |\cC|(1-\rho)$. For each $x\in \tcC$, let $\U_x\subset \Z^2$ be such that $\pi_{\ell_*}(\U_x)\subseteq \pi_{\ell_*}(C_x^L)$ and $\capacity_N(O_x)\geq c\log(N/L)^{-1}$. Then
    \begin{equation}  \label{A:capacity}
        \capacity_N\Big(\bigcup_{x\in \tcC}\U_x\Big) \geq \left(\frac{c}{n}\log(N/L)+ \frac{1+\eta}{2\vartheta}\right)^{-1}, \qquad \text{(see \eqref{eq:tau} regarding $\vartheta$)}.
    \end{equation}
\end{proposition}
The proof of Proposition~\ref{prop:coarse_graining} would take us too far on a tangent, and is postponed to Section~\ref{sec:tube}. It also hinges on certain estimates having to do with \eqref{A:capacity} that are also presented in Section~\ref{sec:tube}.

We now collect two lemmas that control the supremum of certain collections of harmonic averages, the proofs of which will be supplied in Section~\ref{sec:harm_avg}. Actually, two different types of harmonic averages will play a role:  
Lemma~\ref{lem:bad_box} will be used to control events playing the role of $\mathsf{E}$ in \eqref{eq:ub_blueprint}, which have to do with the effect of conditioning on $\mathcal{F}$, i.e., on ``near-critical'' information. This leads to the study of \emph{anchored} harmonic averages (in $B_{L/2}$). Lemma \ref{lem:harm_avg} on the other hand will be used to control $\mathsf{F}$ in \eqref{eq:ub_blueprint}, which is related to the decoupling of (unrooted) $L$-boxes. The very fact that decoupling is feasible can actually be viewed as determining the onset of the sought after ``off-critical'' regime, that ends up exhibiting polynomial decay (in $N/\xi$). 

We start with anchored quantities.
 In view of \eqref{eq:field_decomp}, we abbreviate $\eta^L(\cdot)= \eta^{K}(\cdot)$ and $\psi^L(\cdot)= \psi^{K}(\cdot)$ for $K= B_{L/2}$.
For $S\subseteq
%\tilde
\Z^2$, $\epsilon\in(0,1)$ and $a \in \R$,
consider
\begin{equation}\label{eq:shift_bad}
    \{\text{$\eta^L$ is $(S,\epsilon)$-bad} \}\defeq
    \Big\{\sup_{y\in S} \eta^L({y}) \geq \epsilon|\ba| \Big\},
\end{equation}
and denote the complement event by $\{\text{$\eta^L$ is $(S,\epsilon)$-good} \}.$
\begin{lemma}[Anchored harmonic average] \label{lem:bad_box}
Let $N,L,R\geq 1$ be such that $N, R\geq L$ and consider 
\[
g(N,R)\defeq
\begin{cases}
\log(N/R)^2, & \text{if } N\geq 100R,\\
\big(R^2 \wedge \exp\{R/N\}\big)^{-2}, & \text{if } N< 100R.
\end{cases}
\]
Let $A\subset \Z^2$ and $(S_x)_{x\in A}\subset
%\tilde
\Z^2$ be such that $d(x,B_{L/2})\geq 2R$ and $S_x\subseteq %\t
B(x,R)$ for all $x\in A$. Then
\begin{equation}\label{eq:bad_box}
\P_N
\Big(
\bigcap_{x\in A}\big\{\eta^L\text{ is $(S_x,\epsilon)$-bad}\big\}
\Big) \leq 
c
\exp\bigg\{-\frac{c'\log(N/L)}{g(N,R)}\bigg(\epsilon|\ba| - c'' \sqrt{|A| \frac{g(N,R)}{\log(N/L)}}\bigg)_+^2\bigg\}.
\end{equation}
\end{lemma}
We turn to unrooted harmonic averages.
In view of \eqref{eq:various_boxes} and \eqref{eq:field_decomp}, for $x\in \Z^2$, we use the shorthand
\[
\eta_x(\cdot)=\eta^{(U_x)^c}(\cdot),\quad
\eta_x'(\cdot)=\eta^{(U_x')^c}(\cdot),\quad
\psi_x(\cdot)=\psi^{(U_x)^c}(\cdot),\quad
\psi_x'(\cdot)=\psi^{(U_x')^c}(\cdot).
\]
Note that these fields depend implicitly on the parameters $L,L',M,M'$.
And for $\cC,\cC'$ as in \eqref{def:C}, set
\begin{align}\label{eq:EF_events}
\begin{split} 
&\{\eta_x \text{ is }\epsilon\text{-bad} \}\defeq
    \Big\{\sup_{y\in  \tC_x}
    %\big(\text{sgn}(\ba) 
    \eta_x(y)%\big) 
    \geq (1-\epsilon)|\ba| \Big\}, %\quad \text{ (see below \eqref{eq:various_boxes} for $\tilde\tC_x$)}
    \\
   &\mathsf{F}(\cC,\cC')%=\mathsf{F}_{L,L',M,M'}^{\epsilon,N}(\cC,\cC')
   \defeq\bigcap_{x\in \cC} 
\big\{
    \eta_x \text{ is }\epsilon\text{-bad}
\big\}\cap \bigcap_{x\in \cC'} 
\big\{
    \eta_x' \text{ is }\epsilon\text{-bad}
\big\}
\end{split}
\end{align}
and denote the complement of the former event by $\{\eta_x \text{ is }\epsilon\text{-good} \}$. Further,
define the union of boxes associated to $\cC,\cC'$ in \eqref{def:C} by
\begin{equation} \label{def:C_boxes}
\Sigma(\mathcal{C},\mathcal{C}') \defeq \bigcup_{x\in\cC} \bC_x \cup \bigcup_{x\in\cC'} \bC'_x.
\end{equation}
and abbreviate $\Sigma(\mathcal{C}) = \Sigma(\mathcal{C},\emptyset) =\bigcup_{x\in\cC} \bC_x$, and similarly $\mathsf{F}(\cC)=\mathsf{F}(\cC,\emptyset)$.

\begin{lemma}[Unrooted harmonic average] \label{lem:harm_avg}
For all $\eta\in(0,1)$, $k \geq 1$, there exist $\Cl[c]{nml_eta}(\eta),\Cl[c]{ml_eta}(\eta, k)\in(0,1)$ such that for all $L,L',$ $M,M',N\geq 1$ with $ L\geq L'$, $M\geq M'\geq \Cr{ml_eta}^{-1}$, $N\geq 100ML$ and $\cC,\cC'$ as in \eqref{def:C} satisfying:
\begin{enumerate}[label={(\roman*)}]
    \item $d(\cC,\cC')\geq 15ML$, $(\cC\cup \cC')\cap B_{2ML}=\emptyset$ %(with $d(\cC,\emptyset)=\infty$)
    , $B(\cC \cup \cC^\prime, 4L) \subset B_{2kN}$, 
    \item $\log(\tfrac{N}{ML})/\log(\tfrac{N}{L})\leq \Cr{nml_eta}$ and  $\log(\tfrac{N}{M'L'})/\log(\tfrac{N}{L'})\1_{\{\cC'\neq\emptyset\}}\leq \Cr{nml_eta}$ 
\end{enumerate}
we have, with $\Sigma(\cdot, \cdot)$ as in \eqref{def:C_boxes},
\begin{equation}\label{eq:harm_avg}
 \P_N^{L/2}\left(\mathsf{F}(\cC,\cC')\right)  
\leq   c\exp\left\{-
\frac{\capacity_N(\Sigma(\cC,\cC'))}{2(1+\eta)} \left((1-\epsilon)|\ba|-\frac{c'L}{L'M'}\sqrt{\frac{|\cC\cup\cC'|}{\capacity_N(\Sigma(\cC,\cC'))}}\right)_+^2
\right\}.
\end{equation}
\end{lemma}
The proof of Lemma~\ref{lem:harm_avg} is given in Section~\ref{sec:harm_avg} for better readability, but it is important. The key point is the sharp bound $(1+\eta)\capacity_N(\Sigma(\cC,\cC'))^{-1}$ on the variance of the relevant averaged field, see Lemma~\ref{lem:supvar_harm_avg}, which is ultimately responsible, together with Proposition~\ref{prop:coarse_graining}, for the exact exponent $\vartheta$ in \eqref{eq:thm-asymp}. The separation conditions in (ii) above are needed precisely to achieve this sharp constant.

\subsection{Subcritical phase}\label{sec:ub_sub}

In this section we prove \eqref{eq:thm_ub_asymp} in the simpler case $a>a_*$, using the generic bound \eqref{eq:ub_blueprint} for certain events $\mathsf{E}$ and $\mathsf{F}$, and applying the above lemmas to control the resulting probabilities $\P_N(\mathsf E)$ and $\P_N(\mathsf F)$. This relies in particular on a careful choice of parameters $L$ and $M$ as functions of both $a$ and $N$, see~\eqref{eq:ML_choice_ub} below, to yield a sharp result \eqref{eq:thm_ub_asymp} that is quantitative in the \emph{pair} $(a,N)$.

Recall that $\P_N^{L} =  \P_N^{\partial B_L}$.
\begin{proof}[Proof of \eqref{eq:thm_ub_asymp} for $a> a_*$]
By requiring $\Cr{ub}(\eta)$ in Theorem~\ref{thm:upper_bound} to satisfy  $\Cr{ub}\geq 2 \Cr{C:a*}\cdot \Cr{gNc}^{-1}$, we may assume that $\ba \geq |\ba-\ba_*|-|\ba_*| \geq c$, using Proposition~\ref{thm:crit_pt} to obtain the last inequality. 
Let $M\geq 100,L\geq 10\xi$ and $N\geq 100ML$. Recall the collection of vertices $\A = \A^{L,M}_{0,N}$ supplied by Proposition~\ref{prop:coarse_graining}, and the notation $\eta^L$, $\psi^L$ around \eqref{eq:shift_bad}. For any $\cC\in \A$, $x\in \cC$ and $\rho,\epsilon\in(0,1/2)$, let
\begin{align}\label{eq:EF_events-bis}
\begin{split}
\mathsf{E}(\cC) = \mathsf{E}_{L,M,N}^{\rho,\epsilon}(\cC)&\defeq
\big\{ 
    \exists \tilde{\mathcal{C}} \!\subset\!  \mathcal{C} \text{ with } |\tilde{\mathcal{C}}| \! \geq\!\rho |\mathcal{C}|/2 
    \text{ s.t. $\eta^L$ is $(%\tilde 
    U_x,\tfrac{\epsilon}{2})$-bad} 
   \text{ for all } x\in \tilde\cC
\big\},\\
\mathsf{F}(\cC) = \mathsf{F}_{L,M,N}^{\rho,\epsilon}(\cC)&\defeq
\big\{
    \exists \tilde{\mathcal{C}} \!\subset\! \mathcal{C} \text{ with } |\tilde{\mathcal{C}}|\! \geq\! (1-\tfrac{\rho}{2}) |\mathcal{C}|  
    \text{ s.t. } \bC_x \!\!{\stackrel{\psi^{L} \geq (1-\epsilon/2)\ba}{\longleftrightarrow} } \!{\tC_x^{c}}\text{ for all } x\in \tilde\cC
\big\},
\end{split}
\end{align}
where, for $A,B\subset %\tilde
\Z^2$ and any %continuous 
field $f$ on $
%\tilde
{\Z}^2$, we use $\{A{\stackrel{f\geq a}{\longleftrightarrow}}B\}$ to denote the event that $\{\exists \mathscr{C}\in \{y\in %\tilde
{\Z}^2:f(y)\geq a\}: \mathscr{C}\cap A\neq \emptyset,\mathscr{C}\cap B\neq \emptyset,\mathscr{C}\text{ connected}\}.$
By Proposition \ref{prop:coarse_graining}, taking advantage of \eqref{A:connectivity} and \eqref{def:C} only for the time being, one obtains that
\begin{equation}\label{eq:ub_start}
   \theta_N(a) 
   \leq |\A|\sup_{\cC\in \A}\P_N\big(0\stackrel{\geq \ba}{\longleftrightarrow}\partial B_\xi, \, \bC_x {\stackrel{\geq \ba}{\longleftrightarrow} }{\tC_x^{c}} \text{ for all }x\in \cC\big).
\end{equation}
Using the Markov property of the field, it is straightforward to see by \eqref{eq:shift_bad}, \eqref{eq:EF_events}, \eqref{eq:EF_events-bis} and \eqref{eq:field_decomp},
\begin{align}\label{eq:subcrit_ub_decomp}
\begin{split}
 &\mathsf{E}(\cC)^c\cap \{\bC_x {\stackrel{\geq \ba}{\longleftrightarrow} }{\tC_x^{c}} \text{ for all }x\in \cC\}\subseteq \mathsf{F}(\cC), \text{ and moreover,}\\
&\{ \bC_x {\stackrel{\geq (1-\epsilon/2)\ba}{\longleftrightarrow} }{\tC_x^{c}} \}\subseteq \{\eta_x\text{ is $\epsilon$-bad}\} \cup \{\bC_x\stackrel{ \psi_x  \geq \epsilon\ba/2}{\longleftrightarrow} \tC_x^c\} \text{ under $\P_N^{L/2}$.}
\end{split}
\end{align}
It thus follows from \eqref{eq:ub_start}, \eqref{eq:subcrit_ub_decomp} and applying \eqref{eq:ub_blueprint} with $\mathsf{E}=\mathsf{E}(\cC)$ and $\mathsf{F}=\mathsf{F}(\cC)$ that,
\begin{equation}\label{eq:ub_three_events-bis}
    \theta_N(a)\leq cf^L_\xi(|\ba - \ba_* |)|\A|\sup_{\cC\in \A}\Big(\P_N(\mathsf{E}(\cC))^{\frac{1}{2}}+\sum_{i=1}^2\P_N^{L/2}(\mathsf{F}^i(\cC))\Big),
\end{equation}
where we have also applied a union bound to the event $\mathsf{F}(\cC)$, leading to
\begin{align}\label{eq:EF_events-bisbis}
\begin{split}
\mathsf{F}^1(\cC) = \mathsf{F}_{L,M,N}^{1,\rho,\epsilon}(\cC)&\defeq
\big\{
    \exists \tilde{\mathcal{C}} \!\subset\! \mathcal{C} \text{ with } |\tilde{\mathcal{C}}| \! \geq\!(1-\rho) |\mathcal{C}|  
    \text{ s.t. } \eta_x\text{ is }\epsilon\text{-bad} \text{ for all } x\in \tilde\cC
\big\},\\
\mathsf{F}^2(\cC) = \mathsf{F}_{L,M,N}^{2,\rho,\epsilon}(\cC)&\defeq
\big\{
    \exists \tilde{\mathcal{C}} \!\subset\!  \mathcal{C} \text{ with } |\tilde{\mathcal{C}}| \! \geq\!\rho |\mathcal{C}|/2 
    \text{ s.t. } \bC_x\stackrel{ \psi_x \geq \epsilon\ba/2}{\longleftrightarrow} \tC_x^c \text{ for all } x\in \tilde\cC
\big\},
\end{split}
\end{align}
with $\psi_x(\cdot),$ $\eta_x(\cdot)$ as introduced around \eqref{eq:EF_events}.
It's clear that $\P_N^{{L/2}}(\mathsf{F}^1(\cC))$ can be controlled by \eqref{eq:harm_avg} with $\cC'=\emptyset$ and $L'=L,M'=M$ upon choosing the scales $L,M$ appropriately, with an extra binomial coefficient $\binom{n}{ (1-\rho) n}$ from considering all possible choices of $\tcC$. 
By \eqref{eq:A_structure}, we may apply Lemma \ref{lem:bad_box} with $A=\tcC$, $S_x=%\tilde
 U_x$ and $R=ML$ to get,
\begin{equation}\label{eq:sub_badbox}
    \P_N\big(\mathsf{E}(\cC)\big)
    \leq
    c \binom{n}{\rho n/2}
\exp\bigg\{-\frac{c'\log(N/L)}{\log(N/(ML))^{ 2}}\bigg(\epsilon|\ba| - c'' \sqrt{n \frac{\log(N/(ML))^{ 2}}{\log(N/L)}}\bigg)_+^2\bigg\},
\end{equation}
where we also used \eqref{A:cardinality} to control $|\tcC|$. Here and throughout the rest of this article, for $a\geq b \geq 1$ we let $\binom{a}{b}= \binom{\lceil a\rceil}{\lceil b\rceil } \vee \binom{\lceil a\rceil}{\lfloor b\rfloor }$, the latter referring to the usual binomial coefficients.
The outstanding bound on $\P_N^{{L/2}}(\mathsf{F}^2(\cC))$ is an easier matter. By \eqref{def:C} and the sentence following \eqref{eq:field_decomp}, $(\psi_x)_{x\in\cC}$ are mutually independent Gaussian fields, and by translation invariance it suffices to bound the probability that $B_{2L}\setminus B_L$ is crossed in $\{\varphi\geq \epsilon\ba/2\}$ under $\P_N^{ML}$. By \eqref{eqcouplingintergff} with $u=(\epsilon\ba/2)^2/2$, such a crossing lies in $\V^u$ and thus precludes $\Cir_{0,L}(\I^u)$, so by \eqref{eq:Iu_closed_cont} and \eqref{eq:cap-killed} this probability is at most $\exp\{-c\,u\log(M)^{-1}\}$. Hence
\begin{equation}\label{eq:ub_local_field}
  \P_N^{{L/2}}(\mathsf{F}^2(\cC))
 \leq c\binom{ n }{  \rho n/2} \left(\P_N^{ML}\big(\partial B_L\stackrel{ \geq \epsilon\ba/2}{\longleftrightarrow}\partial B_{2L} \big)\right)^{ \rho n/2}
  \leq c\binom{ n }{  \rho n/2} \exp\left\{-c\frac{\rho n\epsilon^2 \ba^2}{\log(M)}\right\}.
\end{equation}
We now fix the scales $L$ and $M$ and assemble the estimates on \eqref{eq:ub_three_events-bis}. Recall that $n=N/ML$ and that $\ba/a \sim  (\frac2{\pi} \log N)^{1/2} $ as $N \to \infty$. We pick 
\begin{equation}\label{eq:ML_choice_ub}
    L=\frac{N}{\log({|\ba-\ba_*|})} \quad\text{and}\quad M=\frac{N}{L\log(N/L)^2} \text{ (hence $n=(\log\log|\ba-\ba_*|)^2$)}.
\end{equation}
By \eqref{eq:range_crit_pt}, the condition $L\geq 10\xi$, the conditions above \eqref{A:capacity} with $O_x=C_x^L$, those of  Lemma~\ref{lem:critHarm} and of Lemma~\ref{lem:harm_avg} are all met by above choices, since the quantities $M$, $L/\xi={e^{|\ba-\ba_*|^2}}/{\log({|\ba-\ba_*|})}$, $n$ and ${\log(N/L)}/{\log(N/ML)}$ can all be made arbitrarily large by making ${|\ba-\ba_*|}$ large.
By the elementary inequality $\binom{n}{k}\leq \exp\{ck\log\big((n/k)\vee 2\big)\}$ and \eqref{A:entropy}, the binomial coefficients appearing in \eqref{eq:sub_badbox}, \eqref{eq:ub_local_field} and in the bound for $\P^{L/2}_N(\mathsf{F}^1(\cC))$ are all at most
\begin{equation} \label{eq:final_entropy}
    \exp\big\{cn\log\big(\rho^{-1}(1-\rho)^{-1}\big)\big\}
    \leq \exp\big\{\Cr{entropy_A} \log\log({|\ba-\ba_*|})^3\big\},
\end{equation}
given that $\log(N/L)\geq c\log\big(\rho^{-1}(1-\rho)^{-1}\big)$, which can be achieved by taking ${|\ba-\ba_*|}>c(\rho)$. We now control the factor $f^L_{\xi}(|\ba-\ba_*|)$ (cf.~\eqref{eq:f} regarding $f_r^L(\cdot)$) that appears in \eqref{eq:ub_three_events-bis} with the choice of $L$ in \eqref{eq:ML_choice_ub}. 
Using \eqref{eq:f} and the fact that $\log(L/\xi)/\log(L)\asymp |a-a_*|^2$ (which follows from \eqref{eq:ML_choice_ub}), we see that for all $a\in(-1,1)$ such that $\ba>c$,
\begin{equation}\label{eq:inner_shift_moments}
\begin{split}
 f^L_{\xi}(|\ba -\ba_* |)
 &\leq c{|a-a_*|}+{|a-a_*|^2}\big(\log(L)\log(L/\xi)\big)^{-1/2}\\
 &\leq c{|a-a_*|} +c'{|a-a_*|^2}\big(\log(N){|\ba-\ba_*|^2}\big)^{-1/2}\leq c''{|a-a_*|}.    
\end{split}   
\end{equation}
By \eqref{eq:ML_choice_ub}, we have that $n^{-1}\log(N/L)\leq c\log\log({|\ba-\ba_*|})^{-1}$. Hence by combining this inequality,
\eqref{A:capacity}, which is in force, \eqref{def:C_boxes}, \eqref{eq:EF_events-bisbis} and picking $\rho<\Cr{cg_rho}(\eta)$, the capacity functional that appears when applying \eqref{eq:harm_avg} to $\P_N^{L/2}(\mathsf{F}^1(\cC))$ is bounded from below whenever ${ |\ba-\ba_*|^2} \geq c(\eta)$ by
\begin{equation} \label{eq:final_cap}
    \left(\frac{c}{n} \log\left(N/L\right)  + \frac{1+\eta}{2\vartheta}\right)^{-1}\geq 2\vartheta(1+2\eta)^{-1}.
\end{equation}
To see what leads to the exponent $\log(N/\xi)$ in \eqref{eq:thm_ub_asymp}, note that by \eqref{eq:range_crit_pt} and taking ${ |\ba-\ba_*|^2}>c(\eta)$ (which implies $\ba>c'(\eta)$),
\begin{equation}\label{eq:xi_exp}
    \log(N/\xi)=(a-a_*)^2g_N= \ba^2(1-\tfrac{a_*}{a})^2= \ba^2(1+O(\eta))^2.
\end{equation}
To conclude, by picking $\epsilon<c(\eta),\rho<\Cr{cg_rho}(\eta)$ small enough,  taking $|\ba-\ba_*|^2\geq \Cr{ub}(\eta)$ large enough,
%taking $\ba^2\geq \Cr{ub}\red /2$ with $\Cr{ub}=\Cr{ub}(\eta)$ sufficiently large \red (this can always be achieved when $|\ba-\ba_*| \geq \Cr{ub}$ if $ \Cr{ub} \geq 2\Cr{C:a*}$ using \eqref{eq:range_crit_pt}),
we may apply Lemma \ref{lem:harm_avg} to $\P_N^{{L/2}}(\mathsf{F}^1(\cC))$ with $\cC'=\emptyset$ and $L'=L,M'=M$ and
combine \eqref{eq:ub_three_events-bis}, \eqref{eq:ub_local_field}, \eqref{eq:sub_badbox}, \eqref{eq:ML_choice_ub}, \eqref{eq:final_entropy}, \eqref{eq:inner_shift_moments}, \eqref{eq:final_cap} and \eqref{eq:xi_exp}, to get -- after a change of variables in $\eta$ -- that
\begin{align*}
\begin{split}
\theta_N(a)&\leq  c\,{|a-a_*|} \exp\big\{2\Cr{entropy_A} \log\log({|\ba-\ba_*|})^3\big\}\\
&\quad \times\Bigg(\exp\!\left\{\!-\frac{\ba^2\vartheta}{1+\eta}\!\right\}    
+ \exp\!\left\{\!-c\rho\log\log(|\ba-\ba_*|)\epsilon^2\ba^2\!\right\}+\exp\!\left\{\!-\frac{c\epsilon^2\ba^2\log\log(|\ba-\ba_*|)}{\log\log\log(|\ba-\ba_*|)^2}\!\right\}
\Bigg)\\
&\leq \theta_{\xi}(a_*)\exp\left\{-(1-\eta)\vartheta\log(N/\xi)\right\},
\end{split}
\end{align*}
where, to obtain the $\theta_{\xi}(a_*)$ in final inequality, we have also used $\theta_{\xi}(a_*)\geq c|a-a_*|$, see below \eqref{eq:crit_cost_size}. Note that the constant $c$ appearing as a pre-factor in the previous display is effectively absorbed by $\exp\{\eta\vartheta\log(N/\xi)\}=\exp\{|\ba-\ba_*|^2 \eta\vartheta\}$ whenever $|\ba - \ba_*| \geq \Cr{ub}(\eta)$, with $\Cr{ub}(\eta)$ sufficiently large. 
\end{proof}

\begin{remark}\label{rem:asymp_ext}
We now explain how to obtain the upper bound in \eqref{eq:asymp_ext} when $\Cr{gNc}< a-a_*\leq A(\eta)$. The proofs for the upper bound when $a<a_*$ and the lower bound follows from similar adaptations of our argument in Section \ref{sec:ub_sup} and Section \ref{sec:lb}. In this simpler regime, the critical cost $\theta_{\xi\vee1}(a_*)$ is a constant and can be absorbed by $\exp\{\eta\vartheta\log(N/\xi)\}$ whenever $|\ba - \ba_*| \geq c(\eta)$. Indeed, $\xi<1$ and in view of Proposition~\ref{thm:crit_pt} and \cite[Theorem 1.1]{rz25a} with $h_N=1$ therein, we have that $ 1\geq \theta_{\xi\vee1}(a_*)\geq \theta_{1}(0)\geq c.$
Hence there is no need to work with the field $\psi^L$ and invoke Lemma \ref{lem:bad_box}, Lemma \ref{lem:critHarm}, \eqref{eq:ub_blueprint} and the event $\mathsf{E}$ in \eqref{eq:EF_events-bis} to capture the critical cost. \eqref{eq:ub_start} and \eqref{eq:ub_three_events-bis} then simplify to:
\begin{equation}\label{eq:subcrit_simple}
    \theta_N(a) 
   \leq |\A|
   \sup_{\cC\in \A}\sum_{i=1}^2\P_N(\mathsf{F}^i(\cC)).
\end{equation} 
It suffices to check that Proposition \ref{prop:coarse_graining}, Lemma \ref{lem:harm_avg} and \eqref{eq:ub_local_field} are still in force with the same choice of $L$ and $M$ as in \eqref{eq:ML_choice_ub}. Since $a-a_*\leq A(\eta)^2$, asking $|\ba-\ba_*|^2\geq C(\eta)$ with $C$ sufficiently larger than $A(\eta)$ forces $N$ to be large. Hence with the same reasoning as below \eqref{eq:ML_choice_ub},
$L, M, n$ and $\log(N/L)/\log(N/ML)$ can all be made arbitrarily large by making $|\ba-\ba_*|$ large. 
Consequently, after increasing
$C(\eta)$ depending on $A(\eta)$ if necessary, all the condition required to invoke
Proposition \ref{prop:coarse_graining}, Lemma \ref{lem:harm_avg}, and
\eqref{eq:ub_local_field} are satisfied. Applying these estimates to control \eqref{eq:subcrit_simple} yields \eqref{eq:asymp_ext}.
\end{remark}

\subsection{Supercritical phase}\label{sec:ub_sup}

We now prove \eqref{eq:thm_ub_asymp} in the more difficult case $a<a_*$, which in particular implies $a<0$ on account of \eqref{eq:range_crit_pt}. This leads to the  study of a disconnection problem, see Lemma~\ref{lem:disc} below. This lemma isolates the key input, which is proved separately in the next section.

\begin{proof}[Proof of \eqref{eq:thm_ub_asymp} for $a<a_*$]
We still use the generic framework around \eqref{eq:ub_blueprint}, but will consider separately the intersection of the connection-disconnection event in \eqref{eq:connectivity_function} with three events $E^i$ introduced below. Each $E^i$ will in turn lead to its own events $\mathsf{E}^i$ and $\mathsf{F}^i$, to which \eqref{eq:ub_blueprint} applies. The cases $i=1,2$ will be dealt with directly, the case $i=3$, which is the most intricate, comprises the disconnection problem alluded to above. Lemma~\ref{lem:disc} will be used to treat that case.

Consider scales $L\geq 10\xi,M\geq 100$ and $N\geq 100ML$ to be chosen later. Recall $\eta_x$ and the notion of $\varepsilon$-good/bad from around \eqref{eq:EF_events}. Take $\cC\in \A=\A^{L,M}_{0,N}$ (cf. Proposition~\ref{prop:coarse_graining}), and consider the following events for $\rho,\epsilon\in(0,1/2)$ and $k\geq 100$, with $\mathscr{C}^{\geq a}$ denoting the cluster of $0$ in $\{ \varphi \geq a\}$: 
\begin{align}\label{eq:supcrit_Events}
\begin{split}
    E^{1}(\cC) &\defeq
        \big\{
        \exists \tilde{\mathcal{C}} \subset \mathcal{C} \text{ with } |\tilde{\mathcal{C}}| \geq \rho|\mathcal{C}| /6 
        \text{ s.t. } C_x\stackrel{\geq \ba}{\longleftrightarrow}\partial B_{kN}
        \text{ and }C_x\hspace{-1ex}  \stackrel{\quad \geq {\ba}}{\mathrlap{\longleftrightarrow}{\mkern15mu{\text{\scriptsize{/}}}}}  \,\infty \text{ for all } x\in \tilde\cC\big\},\\
    E^{2}(\cC) &\defeq
        \left\{
        \begin{aligned}
        & 0\hspace{-1ex}  \stackrel{\quad \geq {\ba}}{\mathrlap{\longleftrightarrow}{\mkern15mu{\text{\scriptsize{/}}}}}  \,\infty \text{ and $\exists \tilde{\mathcal{C}} \subset \mathcal{C} \text{ with } |\tilde{\mathcal{C}}| \geq \rho |\mathcal{C}| /6
        \text{ s.t. $\eta_x$ is}$}\\
        & \text{$\epsilon/2$-good ,$\bC_x\stackrel{ \geq \ba}{\longleftrightarrow} \infty$ and $\bC_x\stackrel{\mathscr{C}^{\geqslant \ba}}{\longleftrightarrow}\tC_x^{c}$ for all } x\in \tilde\cC
        \end{aligned}        
        \right\},\\
    E^{3}(\cC) &\defeq
        \big\{
        \exists \tilde{\mathcal{C}} \subset \mathcal{C} \text{ with } |\tilde{\mathcal{C}}| \geq (1-\rho/3) |\mathcal{C}|  
        \text{ s.t. } \bC_x\hspace{-1ex}  \stackrel{\quad \geq {\ba}}{\mathrlap{\longleftrightarrow}{\mkern15mu{\text{\scriptsize{/}}}}}  \,\partial B_{kN}
        \text{ or $\eta_x$ is $\epsilon/2$-bad for all } x\in \tilde\cC\big\}.
\end{split}
\end{align}
Similar to \eqref{eq:ub_start}, it follows from \eqref{A:connectivity}, \eqref{def:C} and \eqref{eq:connectivity_function} that
\begin{equation}\label{eq:sup_ub_start}
   \theta_N(a) 
   \leq |\A|\sup_{\cC\in A}\P_N\big(0\stackrel{\geq \ba}{\longleftrightarrow}\partial B_\xi, \, \bC_x {\stackrel{\mathscr{C}^{\geq \ba}}{\longleftrightarrow} }{\tC_x^{c}} \text{ for all }x\in \cC, \, 0\hspace{-1ex}  \stackrel{\quad \geq {\ba}}{\mathrlap{\longleftrightarrow}{\mkern15mu{\text{\scriptsize{/}}}}}  \,\infty\big).
\end{equation}
It's easy to see from the definitions in \eqref{eq:supcrit_Events} that 
$\{\bC_x {\stackrel{\mathscr{C}^{\geq \ba}}{\longleftrightarrow} }{\tC_x^{c}} \text{ for all }x\in \cC\}\cap \{0\hspace{-1ex}  \stackrel{\quad \geq {\ba}}{\mathrlap{\longleftrightarrow}{\mkern15mu{\text{\scriptsize{/}}}}}  \,\infty\}\subseteq \bigcup_{i=1}^3 E^i(\cC)$ for any $\cC \in \A$ (to see this, consider for instance the complement of $E^3$, which readily leads to $E^1$ and $E^2$ on the event in \eqref{eq:sup_ub_start}). It thus follows from a union bound that,
\begin{align} \label{eq:supcrit_decomp}
\begin{split}
    \theta_N(a)\leq |\A|\times \sup_{\cC\in \A}
    \left(\sum_{i=1}^3\P_N\big(0\stackrel{\geq \ba}{\longleftrightarrow}\partial B_\xi, \, E^i(\cC)\big)\right).
\end{split}
\end{align}
We start with the terms involving $E^1(\cC)$ and $E^2(\cC)$. For $i\in\{1,2\}$, the first step is to decompose $E^i(\cC)$ into events $\mathsf{E}^{i}$ and $\mathsf{F}^{i}$ which will play the roles of $\mathsf{E}$ and $\mathsf{F}$ in \eqref{eq:ub_blueprint}.

Let $\gamma\in \cL^{\text{SA}}$ (see above~\eqref{eq:disconnect_peierls} for notation, idem for $\gamma_N$, $\gamma_N'$ appearing below) and let $\text{int}_{\gamma}$ denote the bounded component of $\Z^2\setminus \gamma$. For each $\gamma \in \cL^{\text{SA}}$ and $x\in \cC$, consider the set
\[\cL^{x}_{j}=\{\gamma_N:\gamma\in \cL^{\text{SA}}, \gamma \subseteq \{\varphi<\ba \},
\text{int}_\gamma\cap C_x\neq\emptyset,\text{int}_\gamma\cap \partial B_{kN}\neq\emptyset \text{ and $|\gamma_N|=j$}\}.\]
It's easy to see that $|\cL^{x}_{j}|\leq cj\times e^{c'j}\leq e^{c''j}$ and that $\cL^{x}_{j}=\emptyset$ for $j<k-2$ since $\cC\subset B_N$. With these choices, for any $x\in \cC$, the event
%\begin{equation}
%\big\{C_x\stackrel{\geq \ba}{\longleftrightarrow}\partial B_{kN}\text{ and }C_x\hspace{-1ex}  \stackrel{\quad \geq {\ba}}{\mathrlap{\longleftrightarrow}{\mkern15mu{\text{\scriptsize{/}}}}}  \,\infty\big\}\subseteq \left\{
%\begin{aligned} &\exists \gamma\in \cL^{\text{SA}}  \text{ s.t. } \gamma \subseteq \{z\in\tilde\Z^2:\varphi_z<\ba \},\\&\text{int}_\gamma\cap C_x\neq\emptyset\text{ and }\text{int}_\gamma\cap \partial B_{k_N}\neq\emptyset\end{aligned}\right\}.
%\end{equation}
$
\big\{C_x \, {}{\leftrightarrow} \, \partial B_{kN}, \, C_x \,  {}{\mathrlap{\leftrightarrow}{\mkern7mu{\text{\scriptsize{/}}}}}  \ \ \infty\big\}$ (in $\{\varphi \geq \ba\}$)
appearing as part of $E^1(\cC)$ implies the existence of a $j\geq k-2$ and $\gamma_N\in \cL^{x}_{j}$ such that
$\Cir_{z,N/100}^\star(\{\varphi \geq \ba \})^c$ occurs 
(see~\eqref{eq:circuit_event} for notation) for all $z\in \gamma_N'$. In light of this, and using the terminology from \eqref{eq:shift_bad}, consider the events \begin{align*}
\begin{split}
\mathsf{E}^{1}(\gamma_N)&\defeq 
\big\{\text{$\exists \tilde{\gamma}_N\!\subset \!\gamma_N'$ with $|\tilde{\gamma}_N|\geq |\gamma_N'|/2$ s.t. $\eta^L$ is $(%\t
B(z,\tfrac{N}{10}),\epsilon)$-bad for all $z\in \tilde{\gamma}_N$}\big\},\\
\mathsf{F}^{1}(\gamma_N)&\defeq
\big\{\text{$\exists \tilde{\gamma}_N\!\subset\! \gamma_N'$ with $|\tilde{\gamma}_N|\geq |\gamma_N'|/2$ s.t. 
$\Cir_{z,N/100}^\star(\{\psi^{L}\geq (1\!  - \!\epsilon)\ba\})^c$ occurs 
for all $z\in \tilde{\gamma}_N$} \big\},
\end{split}
\end{align*}
 for any fixed $j\geq k-2$ and $\gamma_N\in \cL^{x}_{j}$.
For any $A\subseteq %\tilde
\Z^2$, on the event $\{\inf_{ A}\eta^{L}\geq - \epsilon|\ba|\}$, we have that $(\{\psi^{L}\geq (1\! -  \!\epsilon)\ba\} \cap A)\subseteq (\{\varphi \geq \ba \}\cap A)$ for $\ba < 0$ by the Markov property of the field. Putting things together, we can thus apply \eqref{eq:ub_blueprint} with $\mathsf{E}=\mathsf{E}^{1}(\gamma_N)$ and $\mathsf{F}=\mathsf{F}^{1}(\gamma_N)$ to obtain
\begin{align}\label{eq:inf_start1}
\begin{split}
    \P_N\big(0&\stackrel{\geq \ba}{\longleftrightarrow}\partial B_\xi, \, E^1(\cC)\big) \\
    \leq & \, cn\sup_{x\in \cC} \bigg( \sum_{j\geq k-2} e^{cj}\sup_{\gamma_N\in \cL^{x}_{j}} \P_N\Big(\{0\stackrel{\geq \ba}{\longleftrightarrow}\partial B_\xi\} \, %\bigcap_{z\in \gamma_N'}
    \text{$\cap \bigcap_{z\in\gamma'_N}\Cir_{z,N/100}^\star(\{\varphi \geq \ba \})^c$ }\Big) \bigg)\\
    \leq & \, cn f^L_\xi(|\ba - \ba_*|) \bigg( \sum_{j\geq k-2} e^{cj} \sup_{x\in \cC}\sup_{\gamma_N\in \cL^{x}_{j}} \Big(\P_N\big(\mathsf{E}^{1}(\gamma_N)\big)^{\frac{1}{2}}+\P_N\big(\mathsf{F}^{1}(\gamma_N)\big)\Big) \bigg).
\end{split}
\end{align}
For $E^2(\cC)$, first observe that by a similar reasoning as above \eqref{eq:inf_start1}, we have that for any $x\in \cC$,
\begin{equation}\label{eq:inf_pre2}
\left\{
\begin{array}{c}
\text{ $0\hspace{-1ex}  \stackrel{\quad \geq {\ba}}{\mathrlap{\longleftrightarrow}{\mkern15mu{\text{\scriptsize{/}}}}}  \,\infty$,  $ \bC_x\stackrel{ \geq \ba}{\longleftrightarrow} \infty$, $\bC_x\stackrel{\mathscr{C}^{\geqslant \ba}}{\longleftrightarrow}\tC_x^{c}$,}  \\
\text{$ \eta^L$ is $(%\tilde 
U_x,\tfrac{\epsilon}{4})$-good, and $ \eta_x$ is $\tfrac{\epsilon}{2}$-good } 
\end{array}
\right\}\subseteq
\text{ $\Cir_{x,L}(\{\psi_x\geq \epsilon\ba/4\})^c$} 
\end{equation}
(again, see \eqref{eq:shift_bad} and \eqref{eq:EF_events} for the relevant notions of goodness on the left-hand side, and \eqref{eq:circuit_event} for the right-hand side). Now consider the events,
\begin{align*}
\mathsf{E}^{2}(\tcC)&\defeq 
\big\{\text{$\exists \bar{\cC}\subset \tcC$ with $|\bar{\cC}|\geq |\tcC|/2$ s.t. $\eta^L$ is $(%\tilde
 U_x,\tfrac{\epsilon}{4})$-bad for all $x\in \bar{\cC}$}\big\},\\
\mathsf{F}^{2}(\tcC)&\defeq 
\big\{\text{$\exists \bar{\cC}\subset \tcC$ with $|\bar{\cC}|\geq |\tcC|/2$ s.t.  $\Cir_{x,L}(\{\psi_x\geq \epsilon\ba/4\})^c$ occurs for all $x\in \bar{\cC}$}\big\}.
\end{align*}
It then follows from \eqref{eq:inf_pre2} that we can apply \eqref{eq:ub_blueprint} with $\mathsf{E}=\mathsf{E}^{2}(\tcC)$ and $\mathsf{F}=\mathsf{F}^{2}(\tcC)$ and get
\begin{equation}\label{eq:inf_start2}
\P_N\big(0\stackrel{\geq \ba}{\longleftrightarrow}\partial B_\xi, \, \!E^2(\cC)\big)
\leq cf^L_\xi(|\ba-\ba_*|)\binom{n}{\rho n/6}\sup_{\tcC}
\left( \P_N\big(\mathsf{E}^{2}(\tcC)\big)^{\frac{1}{2}}+\P_N\big(\mathsf{F}^{2}(\tcC)\big) \right),
\end{equation}
where the supremum is over all $\tcC \subset \cC$ such that $|\tcC|\geq \rho|\cC|/6$.

We now bound the probability of the events $\mathsf{E}^{i}$ and $\mathsf{F}^{i}$ for $i\in\{1,2\}$ in \eqref{eq:inf_start1} and \eqref{eq:inf_start2}.
We first look at $\mathsf{F}^{1}$ and $\mathsf{F}^{2}$. 
For $\mathsf{F}^{1}$ we can use the isomorphism \eqref{eqcouplingintergff} applied to $\varphi$ under $\P_N^L$, which has the same law as $\psi^L$ under $\P_N$, and monotonicity, to replace $\Cir_{z,N/100}^\star(\{\psi^{L}\geq (1\!  - \!\epsilon)\ba\})$ by $\Cir_{z,N/100}^\star(\I^u_{B_z})$
with $B_z= %\t
B(z,\tfrac{N}{10})$ and $u=(1-\epsilon)^2\ba^2/2$ for a given $z\in \tilde{\gamma}_N$. 
Since moreover, for $\gamma_N\in \cL^{x}_{j}$, $|\tilde{\gamma}_N|\ge cj$ and the balls $B(z,\tfrac{N}{10})$, $z\in\tilde{\gamma}_N$, are disjoint, we get by the a-priori bound \eqref{eq:Iu_closed_cont_macro} and independence that
\begin{align}\label{eq:inf_loc_cost1}
\begin{split}
    \P_N\big(\mathsf{F}^{1}(\gamma_N)\big)%&\leq c\binom{j}{j/2}\sup_{\tilde\gamma_N}\P_N^{L/2}\Big( \bigcap_{z\in \tilde\gamma_N}\{\text{$\K^{(1-\epsilon)\ba}$ is $(z,\tfrac{N}{100})$-bad}\}\Big)\\
    &\leq \binom{j}{j/2}\sup_{\tilde\gamma_N}
  \bar\P_N^{L/2}\Big( 
   \bigcap_{z\in\tilde{\gamma}_N} \Cir_{z,N/100}(\I^u_{B_z})^c
  \Big)\leq e^{-cj\ba^2},
\end{split}
\end{align}
where we also used the fact that $|\gamma_N'|\geq cj$ and $|\ba|>c$ for suitably large $c$ (ensured by choosing $\Cr{ub}$ large enough) in the last inequality. The bound for $\mathsf{F}^2(\tcC)$ is obtained exactly as \eqref{eq:ub_local_field}, now using the symmetric version $\I^u\subseteq\{\varphi\geq-\sqrt{2u}\}$ of \eqref{eqcouplingintergff} with $u=(\epsilon\ba/4)^2/2$ (recall that $\ba<0$), which gives $\P_N(\mathsf{F}^2(\tcC))\leq c\binom{n}{\rho n/12}\exp\{-c'\rho n(\epsilon\ba)^2\log(M)^{-1}\}$.

For $\mathsf{E}^{i},i\in\{1,2\}$, which both involve anchored quantities $\eta^L(\cdot)$, we use Lemma \ref{lem:bad_box}.
Note that for any $\tilde{\gamma}_N$ in the definition of $\mathsf{E}^1$, we have that $j\geq |\tilde{\gamma}_N|\geq |\gamma_N'|/2\geq cj$ and that the set $A=\tilde{\gamma}_N\cap B_{2\bar c\sqrt{j}N}^c$ satisfies $|A| \geq c'j$ for $\bar c\in(0,1)$ sufficiently small. Hence by applying  \eqref{eq:bad_box} to this $A$ and with $R=(\bar c\sqrt{j}-1)N$, $S_x=%\t
B(x,\tfrac{N}{10})$, we get that for all $j\geq c$ and $\ba^2>c(\epsilon)$,
\begin{multline}\label{eq:inf_crit_cost1}
    \P_N(\mathsf{E}^{1}(\gamma_N))
    \leq c\binom{j}{j/2}\sup_{\tilde{\gamma}_N}
    \P_N\big(
    \text{$\eta^L$ is $(%\t
    B(z,\tfrac{N}{10}),\epsilon)$-bad for all $z\in \tilde{\gamma}_N\setminus B_{2\bar c\sqrt{j}N}$}
    \big)    \\
    \leq e^{cj} \exp\bigg\{-\frac{c'\log(N/L)}{\big((\sqrt{j}N)^2\wedge e^{\sqrt{j}}\big)^{-2}}\bigg(\epsilon|\ba|-c''\sqrt{\frac{j\big((\sqrt{j}N)^2\wedge e^{\sqrt{j}}\big)^{-2}}{\log(N/L)}}\bigg)_+^2\bigg\}
 %   \leq e^{cj} e^{-c'j(\epsilon\ba)^2}
 \leq e^{-cj(\epsilon\ba)^2}.
\end{multline}
For $\mathsf{E}^{2}(\tcC)$, Lemma~\ref{lem:bad_box} applies exactly as in \eqref{eq:sub_badbox} (with $A=\bar{\cC}$, $S_x=U_x$ and $R=ML$) and yields the bound \eqref{eq:sub_badbox} with $\rho n/12$ in place of $\rho n/2$, which completes the treatment of terms involving $E^1(\cC)$ and $E^2(\cC)$ in \eqref{eq:supcrit_decomp}. From now on, $L$ and $M$ are chosen as in \eqref{eq:ML_choice_ub}; the estimate on the term involving $E^3(\cC)$ hinges on this choice.
The following lemma takes care of $E^{3}(\cC)$ in \eqref{eq:supcrit_decomp}. 
\begin{lemma}\label{lem:disc}
 For all $k\geq 2$, $\eta\in(0,1)$, there exist $\Cl[c]{disc}(\eta,k)<\infty$ and $\epsilon_0(\eta)>0$ such that, with $L,M,N$ as in \eqref{eq:ML_choice_ub}, the following holds: for all $\rho \in(0,\Cr{cg_rho}(\eta)]$ (defined~at \eqref{eq:A_structure}), all $\tcC\subset \cC\in \A^{L,M}_{0,N}$ with cardinality at least $|\cC|(1-\rho/3)$, all 
  $a$ with $-\Cr{gNc}\leq a-a_*<0$, and ${|\ba-\ba_*|^2\geq \Cr{disc}}$ and $\epsilon\leq\epsilon_0$,
    \begin{equation}\label{eq:lem_disc}
    \P_N\bigg(
    0\stackrel{\geq \ba}{\longleftrightarrow}\partial B_\xi, \,
    \bigcap_{x\in \tcC}
    \bigg\{
    \begin{array}{c}
    C_x\hspace{-1ex}  \stackrel{\quad \geq {\ba}}{\mathrlap{\longleftrightarrow}{\mkern15mu{\text{\scriptsize{/}}}}}  \,\partial B_{kN} \text{ or }\\
    \eta_x\text{ is $\epsilon/2$-bad}
    \end{array}
    \bigg\}
    \bigg)\! \leq \! f^L_\xi(|\ba-\ba_*|) \exp\big\{-(1-\eta)\vartheta{\log(N/\xi)}\big\}.
    \end{equation}
\end{lemma}
We will devote Section~\ref{sec:disc} to the proof of Lemma \ref{lem:disc}.
For the time being, we will explain how to complete the proof of \eqref{eq:thm_ub_asymp} by putting together all the above ingredients, including Lemma \ref{lem:disc}.
By the same reasoning as in \eqref{eq:final_entropy}, the factor of $n$ appearing in \eqref{eq:inf_start1} and the binomial coefficients appearing in \eqref{eq:inf_start2} and in the bounds on $\mathsf{E}^2(\tcC)$ and $\mathsf{F}^2(\tcC)$ are all at most $\exp\big\{\Cr{entropy_A} \log\log({|\ba-\ba_*|})^3\big\}$
given that $\log(N/L)\geq c\log\big(\rho^{-1}(1-\rho)^{-1}\big)$, which can be achieved by taking ${|\ba-\ba_*|^2} >c(\rho)$. Hence upon picking $\epsilon= \epsilon_0$ and $\rho\leq \Cr{cg_rho}$ small enough and taking $k\geq c(\eta)$, ${|\ba-\ba_*|^2}\geq \Cr{disc}(\eta,k)$ and large enough, we get by combining \eqref{eq:inner_shift_moments}, \eqref{eq:supcrit_decomp}, \eqref{eq:inf_start1}, \eqref{eq:inf_start2}, \eqref{eq:inf_loc_cost1}, \eqref{eq:inf_crit_cost1}, the bounds on $\mathsf{E}^2(\tcC)$ and $\mathsf{F}^2(\tcC)$ above, \eqref{eq:ML_choice_ub} and Lemma \ref{lem:disc} that
\begin{align*}
\theta_N(a)\leq c&{|a-a_*|} e^{2\Cr{entropy_A}\log\log({|\ba-\ba_*|})^3}\\
\times 
&\left(
\begin{aligned}
&\sum_{j=k-2}^\infty e^{cj}e^{-c'j(\epsilon\ba)^2}
+\exp\Big\{\!-\!\frac{\log\log(|\ba-\ba_*|)(\epsilon\ba)^2}{\log\log\log(|\ba-\ba_*|)^2}\Big\}+\\
& \exp\{\!-\!c{\rho \log\log(|\ba-\ba_*|)\epsilon^2\ba^2 }\}
+\exp\{-(1- \eta/ 3)\vartheta \log(N/\xi)\}
\end{aligned}
\right),%\\
%& \leq c|a| \exp\{2\Cr{entropy_A}\log\log(\ba)^3\}
%\times 
%\left(e^{-c(\epsilon\ba)^2k} + \exp\Big\{\!-\!\tfrac{1-2\eta}{2}\ba^22\vartheta\Big\}\right)\\
%& \leq c|a| \exp\{2\Cr{entropy_A}\log\log(\ba)^3\}
%\times \exp\Big\{\!-\!\tfrac{1-3\eta}{2}\ba^22\vartheta\Big\}\\
%   &\leq \theta(0,\xi) \exp\bigg\{-\frac{1-\eta}{3}2\vartheta\log(N/\xi)\bigg\},
\end{align*}
with $\vartheta>0$ as in \eqref{eq:tau}. The first term in parenthesis is bounded by $e^{-c(\epsilon\ba)^2k}$ and the second and third terms by $e^{-(\epsilon\ba)^2 h(|\ba-\ba_*|)}$ for a function $h$ with $h(s)\to\infty$ as $s\to\infty$, so that for $k \geq c(\eta)$ and $|\ba-\ba_*|\geq c(\eta)$ they can be absorbed by replacing $\eta/ 3$ by $2\eta/3$ in the last term of the previous display. A final fraction $\eta/3$ is sacrificed to take care of the entropy term involving $\Cr{entropy_A}$ and produce the factor $\theta_{\xi}(a_*)$ via $f^L_\xi(|\ba-\ba_*|)$, which relies on \eqref{eq:crit_cost_size}, following the same reasoning as in the subcritical case at the end of Section~\ref{sec:ub_sub}.
\end{proof}

\section{Disconnection estimate}\label{sec:disc}

This section deals with the proof of Lemma~\ref{lem:disc}, which was instrumental in the proof of \eqref{eq:thm_ub_asymp} in the supercritical case $a<a_*$.
\begin{figure}[t] 
    \centering
    \includegraphics[width=\textwidth]{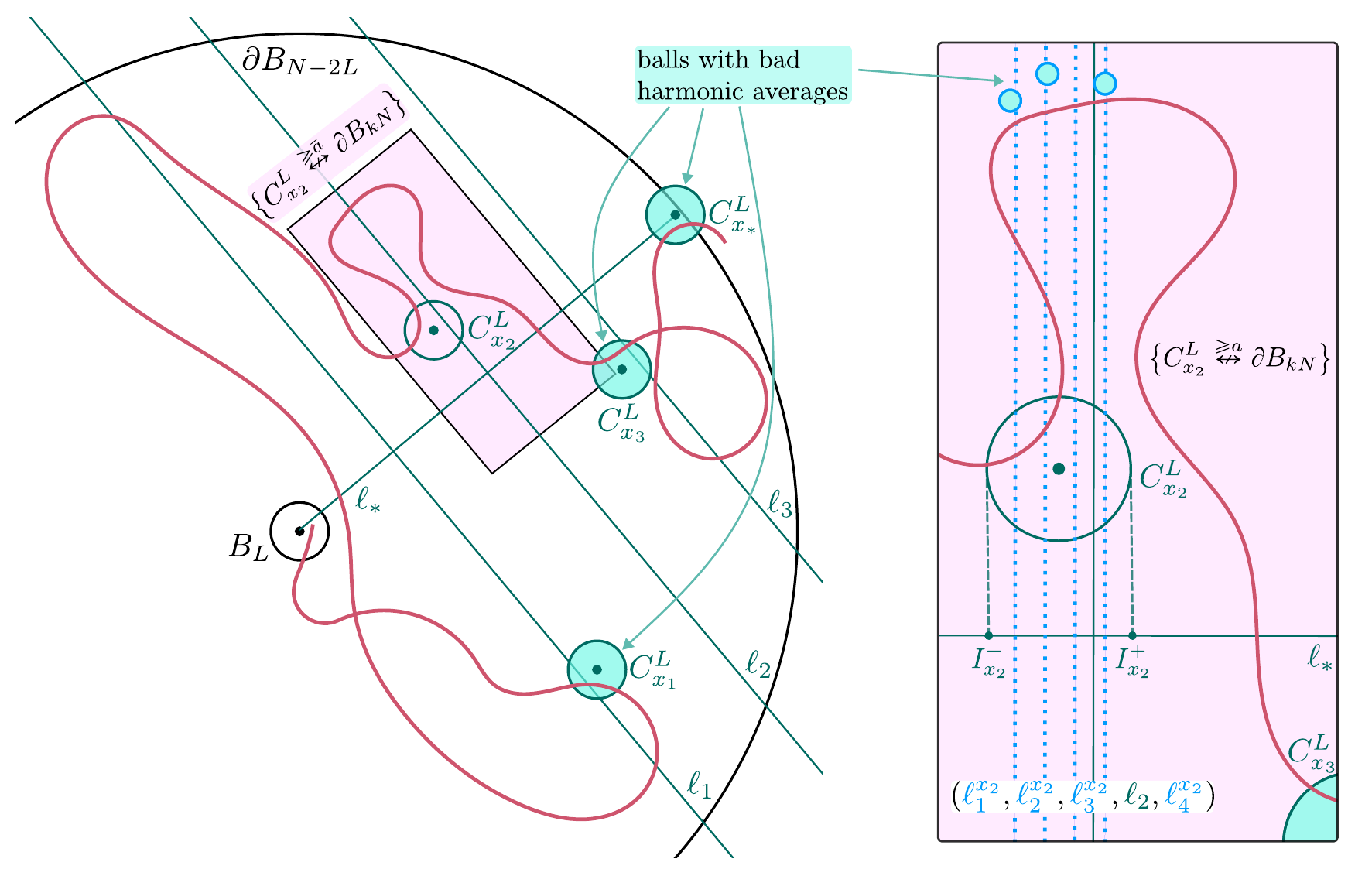}
    \caption{The left figure depicts the coarse-graining scheme from Proposition \ref{prop:coarse_graining} and constructed at the end of Section \ref{sec:tube} with $n_*=4$ and $\cC=\{x_1,x_2,x_3,x_*\}$. The highlighted balls $C_{x_1}^L,C_{x_3}^L,C_{x_*}^L$ are those with bad harmonic average in \eqref{eq:lem_disc}. 
    The right figure zooms in on the pink region on the left and shows how the disconnected ball $C_{x_2}^L$ is treated using the construction in \eqref{eq:lprop_mutdist}--\eqref{eq:Adisc_entropy}. The small highlighted balls there correspond to the balls centred at $\bigcup_{z\in \bar{\cC}\setminus\bar{\cC}^{\eta} }\mathcal{D}'_z$ in \eqref{eq:event_two_harm}. The capacity of the union of all highlighted balls across both figures is computed in \eqref{eq:cap_disc}, yielding the constant $\vartheta$ appearing in \eqref{eq:lem_disc}.} 
    \label{fig:cg_disc}
\end{figure}
To handle the family of disconnection events $\{\bC_x \hspace{-1ex}  \stackrel{}{\mathrlap{\leftrightarrow}{\mkern7mu{\text{\scriptsize{/}}}}}   \partial B_{kN}\}$ appearing in \eqref{eq:lem_disc}, we introduce a scheme that ultimately collects a collection of smaller balls of radius $L' \ll L$ (recall that $C_x=B(x,L)$) with poor connectivity properties and suitably large capacity (cf.~\eqref{eq:disc_implies_bad} and Lemma \ref{lem:block_cap} below), which obstructs the connection to $B_{kN}^c$. The collection of smaller balls will form the set $O_x$ in \eqref{A:capacity} whenever $\bC_x$ is disconnected from $  \partial B_{kN}$ ($O_x$ is simply taken to be $\bC_x$ otherwise). We refer to Figure \ref{fig:cg_disc} for an illustration.

In the sequel we work with arbitrary scales $M,L\geq 1$ satisfying $N\geq 100ML$ and only impose the choice \eqref{eq:ML_choice_ub} at the end of this section.
Consider scales $M',L'\geq 1$ such that $L\geq 100M'L'$ to be chosen later. Recall $\A=\A^{L,M}_{0,N}$ from Proposition~\ref{prop:coarse_graining}. The following construction applies to any collection $\cC\in \A$, fixed in the sequel. Recall the line $\ell_*\subset \R^2$ from \eqref{eq:A_structure} and orthogonal projection $\pi_{\ell_*}$ onto it. 
For each $x\in \cC$, denote by $I_x^{-},I_x^{+}\in \pi_{\ell_*}(C_x)$ the two extremal points of $\pi_{\ell_*}(C_x^L)$ (with respect to the natural order along $\ell_*$). Let $\cS_x=\{ \ell^x_1,\dots,\ell^x_{n_x}\}\subset \R^2$ be a collection of lines that are perpendicular to $\ell_*$ such that, denoting by $\pi_{\ell_*}(\cS_x)=\{\pi_{\ell_*}(\ell^x_1),\dots,\pi_{\ell_*}(\ell^x_{n_x})\}$:
\begin{align}
\label{eq:lprop_position}
    &\text{all points in $\pi_{\ell_*}(\cS_x)$ lie in-between $I_x^{-},I_x^{+}$ along $\ell_*$},\\
    \label{eq:lprop_mutdist}
    &\text{all points in $\pi_{\ell_*}(\cS_x)$ have mutual $d_{\ell_*}$-distance at least $16M'L'$},\\
\label{eq:lprop_endist}
    &\text{$d_{\ell_*}(I_x^{\cdot},\pi_{\ell_*}(\ell^x_i))\geq 16M'L'$ for all $1\leq i\leq n_x$ and $\cdot\in\{-,+\}$ },\\
\label{eq:lprop_card}
    &\text{$n_x\in\big[\Cl[c]{col_num_s}L/(M'L'),L/(M'L')\big]$ }.
\end{align}
For each $x\in \cC$, we fix a collection $\cS_x=\{ \ell^x_1,\dots,\ell^x_{n_x}\}$  of lines satisfying \eqref{eq:lprop_position}-\eqref{eq:lprop_card} from here onwards. We then introduce $\hat{\ell}_i^x$, a discrete approximation of the line ${\ell}_i^x$ in $B_{kN}$, and the associated set of possible ``disconnection points'' $\mathcal{D}_x$, whose $L'$-neighbourhood will block a connection from $C_x$ to $\partial B_{kN}$ along each line. Recall that $C_y^\prime$, $D_y^\prime$ are defined as $C_x=C_x^L$ etc.~ in \eqref{eq:various_boxes} with $L'$ in place of $L$. Let
\begin{align}\label{def:disc_col}
\begin{split}
&%\hat{\cS}_x\defeq\{\hat{\ell}^x_1,\dots,\hat{\ell}^x_{n_x}\}\quad \text{and} \quad 
\mathcal{D}_x\defeq\{(z_1,\dots,z_{n_x}):z_i\in \hat{\ell}^x_i\text{ for all $1\leq i\leq n_x$}\}, \text{ where }  \\
&
\hat{\ell}^x_i\defeq\{y\in \Lambda(L')\cap B_{(k+1)N}:\bC'_y\cap \ell^x_i\neq\emptyset \}
\text{ for all $1\leq i \leq n_x$.}
\end{split}
\end{align}
Let $\A^{\text{disc}}=\A^{\text{disc}} (\cC)$ be the collection of all families $\mathcal{D}=(\mathcal{D}_x)_{x\in \cC}$ obtained in this way. Note that $\A^{\text{disc}}$ depends implicitly on all of ${N,L,L',M,M'}$. By construction, for any $x\in \cC$ and $1\le i\le n_x$, the cardinality of the discretised line segment  $\hat{\ell}^x_i$ is bounded by $ckN/L'$, and by \eqref{A:cardinality} and \eqref{eq:lprop_card}, there are at most $cnn_x \leq c'nL/(L'M')$ such segments involved. This yields (for each $\cC$),
\begin{equation}\label{eq:Adisc_entropy}
    |\A^{\text{disc}}|\leq \exp\left\{\frac{cnL}{L'M'}\log\Big(\frac{kN}{L'}\Big)\right\}.
\end{equation}

With a view towards \eqref{A:capacity}, the following lemma will be helpful. 
\begin{lemma}\label{lem:block_cap}
For any $x\in \cC$, $\mathcal{D}\in \A^{\text{disc}}$ and any $\mathcal{D}_x^\prime\subseteq \mathcal{D}_x$ with $|\mathcal{D}_x^\prime|\geq \frac{1}{2}|\mathcal{D}_x|$, we have that
\begin{equation} \label{eq:block_cap}
    \capacity_N\Big(\bigcup_{y\in \mathcal{D}_x^\prime}\bC'_y \Big)\geq c\left(\big(L'M'\log(N/L')/L\big) \vee \log(N/L)\right)^{-1}.
\end{equation}
\end{lemma}
\begin{proof}
We abbreviate $\bigcup_{y\in \mathcal{D}_x^\prime}\bC'_y=\Sigma$. 
By the variational principle for capacity, see \cite[(1.61)]{sznitman_topics_2012}, $\capacity_N(\Sigma)$ equals $(\inf_{\mu}\sum_{x,y\in \Z^2}g_N(x,y)\mu(x)\mu(y))^{-1}$, where the infimum is over all probability measures $\mu$ supported on $\Sigma$.
By picking $\mu$ to be the uniform measure, we obtain an immediate lower bound,
    \begin{equation} \label{eq:cap_var_form}
      \capacity_N(\Sigma) \geq
      \frac{|\Sigma|^2}{\sum_{u,v \in\Sigma} g_N(u,v)}\geq 
      {\inf_{y\in \mathcal{D}_x^\prime,u\in \bC'_y}}
      \frac{|\Sigma|}{
       \sum_{v \in\bC'_y} g_N(u,v) 
       +
       \sum_{z\in \mathcal{D}_x^\prime\setminus\{ y\}} \sum_{w \in \bC_{z}^{L'}}
       g_N(u,w)
       }
      .
    \end{equation}
  By \eqref{eq:bessel_ub} and \cite[(3.27)]{rz25a} with $h_N=1$ therein, we have that for any $y\in \mathcal{D}_x^\prime$ and $u\in \bC'_y$,
  \[\sum_{v \in\bC'_y} g_N(u,v) \leq c (L')^2 K_0({ L'}/{N}) \leq c (L')^2 \log(N/L'),\]
    where we also used the fact that $N\geq 100L'$ in the last inequality.
    Furthermore, it follows from \eqref{eq:lprop_mutdist} that we can find a relabelling of $\mathcal{D}_x^\prime\setminus\{ y\}$ into $(y_1,\dots y_{n_x-1})$ such that for all $1\leq i\leq n_x-1$, $d(C'_{y_{i}},C'_y)\geq \lceil i/2 \rceil M'L'$. Hence by \eqref{eq:g2-asymp-unif},
    \eqref{eq:bessel_ub} and \cite[(3.13)]{rz25a} we have that
    \[
    \sum_{z\in \mathcal{D}_x^\prime\setminus\{ y\}} \sum_{w \in \bC'_{z}}
    g_N(u,w) \leq 2(L')^2\sum_{i=1}^{n_x}K_0\Big(\frac{M'L'i}{N}\Big)
    \leq cn_x(L')^2\Big( K_0\Big(\frac{M'L'n_x}{N}\Big)\vee 1\Big)
    \leq cn_x(L')^2\log\Big(\frac NL\Big),
    \]
    where the last bound follows since $M'L'n_x\leq L\leq N/100$ (see \eqref{eq:lprop_card}).
    Combining the last two displays and the fact that $|\Sigma|\geq  c(L')^2  n_x$, we find using  \eqref{eq:lprop_card} that the denominator on the r.h.s of \eqref{eq:cap_var_form} is at most
    \begin{align*}
    \begin{split}
        c(L')^2n_x \left(\frac{L'M'}{\Cr{col_num_s} L}\log(N/L') +\log(N/L)\right) \leq
        c'|\Sigma| \left(\frac{L'M'}{L}\log(N/L') +\log(N/L)\right) ,
    \end{split}
    \end{align*}
    which concludes the proof of \eqref{eq:block_cap}.
\end{proof}
\begin{proof}[Proof of Lemma \ref{lem:disc}]
Recall the events defined in \eqref{eq:circuit_event}--\eqref{eq:LU_event}. It follows from \eqref{eq:conChain} and \eqref{def:disc_col} that
\begin{equation} \label{eq:disc_implies_bad}
 \bigcap_{x\in \tcC} \{C_x\hspace{-1ex}  \stackrel{\quad \geq {\ba}}{\mathrlap{\longleftrightarrow}{\mkern15mu{\text{\scriptsize{/}}}}}  \,\partial B_{kN}\}
\subseteq 
\Big\{
    \text{$ \exists \mathcal{D}\in \A^{\text{disc}}( \tcC)$ s.t.
     $\Con_{y,L'}(\{\varphi \geq \ba\})^c$ 
    for all $y\in { \bigcup_{x\in\tcC}}\mathcal{D}_x$}    
\Big\}.
\end{equation} 
Recall $\psi_x'$, $\eta_x' $ and $\varepsilon$-badness from above \eqref{eq:EF_events}, and consider the events for $\epsilon\in(0,1/2)$ and $x\in \Z^2$,
\begin{equation}\label{eq:psi_con}
\{\text{$x$ is $\epsilon$-disc}\}\defeq
\Con_{x,L'}(\{ \psi'_x \geq \tfrac{\epsilon}{2} \ba\})^c
\cup \{\eta'_x\text{ is $\epsilon$-bad}\},
\end{equation}
where we defined $\{\text{$x$ is $\epsilon$-con}\}$ as the complement of $\{\text{$x$ is $\epsilon$-disc}\}$.
With this and for fixed $\cC$, $\tcC$ as appearing in the statement of Lemma~\ref{lem:disc}, i.e.,~$\cC \in \A$ and $\tcC\subset \cC$ having cardinality at least $|\cC|(1-\rho/3)$,  and any $\mathcal{D}\in  \A^{\text{disc}}(\tcC) \equiv \A^{\text{disc}}$, introduce (see above \eqref{eq:shift_bad} regarding $\eta^L$)
\begin{align}\label{eq:disc_bad}
\begin{split}
  \mathsf{E}^{1}(\mathcal{D}) &\defeq
    \left\{
    \text{$\exists \bar{\cC}\subset \tcC$ with $|\bar{\cC}|\geq \rho|\cC|/3$ s.t. $\eta^L$ is $(%\tilde 
    U_x,\tfrac{\epsilon}{2})$-bad for all $x\in \bar{\cC}$} 
    \right\},\\
    \mathsf{E}^{2}(\mathcal{D}) &\defeq
    \left\{
    \begin{aligned}
    &\text{$\exists \bar{\cC}\subset \tcC$ with $|\bar{\cC}|\geq \rho|\cC|/3$ s.t. for all $x\in \bar{\cC}$,} \\
    &\text{$\eta^L$ is $(%\tilde 
    U_y',\tfrac{\epsilon}{2})$-bad for some vertex $y\in \mathcal{D}_x$} 
    \end{aligned} 
    \right\},\\
    \mathsf{F}(\mathcal{D}) &\defeq
    \left\{
    \begin{aligned}
     &  \text{$\exists \bar{\cC}\subset \tcC$ with $|\bar{\cC}|\geq (1-\rho)|\cC|$ s.t. for all $x\in \bar{\cC}$,} \\
    &\text{either $\eta_x$ is ${\epsilon}$-bad or 
    $y$ is $\epsilon$-disc for all $y\in \mathcal{D}_x$}
    \end{aligned}
    \right\}.
\end{split}
\end{align}
As we now detail, it follows with these definitions that
the l.h.s of \eqref{eq:lem_disc} is at most
\begin{align} \label{eq:disc1}
\begin{split}
&|\A^{\text{disc}}| \cdot \sup_{\mathcal{D}}\P_N\Big(  0\stackrel{\geq \ba}{\longleftrightarrow}\partial B_\xi, \,
    \bigcap_{x\in \tcC}\!\Big(\{\eta_x\text{ is ${\tfrac{\epsilon}{2}}$-bad}\}\!\cup\!\!\bigcap_{y\in \mathcal{D}_x}\! \Con_{y,L'}(\{\varphi \geq \ba\})^c 
    \!  \Big)\Big)\\
&\leq  c\, |\A^{\text{disc}}|\cdot f^L_\xi(|\ba-\ba_*|)\sup_{\mathcal{D}}
\Big(
\sum_{i=1,2}\P_N\big(\mathsf{E}^i(\mathcal{D})\big)^{\frac{1}{2}}+\P_N^{L/2}\big(\mathsf{F}(\mathcal{D})\big)
\Big),
\end{split}
\end{align}
with the supremum ranging over $\A^{\text{disc}}= \A^{\text{disc}}(\tcC)$. Let us now explain \eqref{eq:disc1}. The first line is immediate using \eqref{eq:disc_implies_bad}. For the second line, consider the auxiliary event $\tilde{\mathsf{F}}$ defined similarly as $\mathsf{F}$,
\[
\tilde{\mathsf{F}}(\mathcal{D})=
    \left\{
    \begin{gathered}
       \text{$\exists \bar{\cC}\subset \tcC$ with $|\bar{\cC}|\geq (1-\rho)|\cC|$ s.t. for all $x\in \bar{\cC}$, either $\Con_{y,L'}(\{\psi^{L}\geq (1-\tfrac{\epsilon}{2})\ba\})^c$ }  \\
     \text{occurs for all $y\in \mathcal{D}_x$, or $\sup\nolimits_{z\in %\tilde
     {D}_x}E_z^{L/2}[\psi^{L}_{H_{U_x^c}},H_{U_x^c}<\infty]\geq (1-\epsilon)|\ba|$}
    \end{gathered}
    \right\}.
\]
To obtain the bound in the second line of \eqref{eq:disc1} with $\tilde{\mathsf{F}}$ instead of $\mathsf{F}$, one sacrifices on the complement of $\mathsf{E}^1$ a fraction $|\cC|/3$ of $\tilde{\cC}$ to transform $\{\eta_x\text{ is ${\tfrac{\epsilon}{2}}$-bad}\}$ into the event involving the supremum in the definition of $\tilde{\mathsf{F}}$. One applies a similar reasoning to the (dis)connectedness events {$\Con_{y,L'}(\cdot)^c$}, leveraging  $(\mathsf{E}^2)^c$ instead.
This leads to $\tilde{\mathsf{F}}$ as defined above, which is independent of $\sigma(\varphi_x : x\in \partial{B}_{L/2})$.
One then simply applies \eqref{eq:ub_blueprint} with $\mathsf{E}=\mathsf{E}^{1}(\mathcal{D})\cup \mathsf{E}^{2}(\mathcal{D})$ and $\tilde{\mathsf{F}}$ in place of $\mathsf{F}$, noting to this effect that $\mathsf{E}^i$ and $\tilde{\mathsf{F}}$ indeed have the required measurability properties. The replacement of $\P_N(\tilde{\mathsf F}(\mathcal{D}))$ by $ \P_N^{L/2}(\mathsf{F}(\mathcal{D}))$ works as follows. First, the event involving the sup in the definition of $\tilde{\mathsf{F}}$ exactly corresponds to the event $\eta_x$ is $\epsilon$-bad in the definition of $\mathsf{F}$. 
Now by \eqref{eq:LU_event}, for any $y\in \mathcal{D}_x$, the events $\Con_{y,L'}(\{\psi^{L}\geq (1-\tfrac{\epsilon}{2})\ba\})^c$ and $\Con_{y,L'}(\{\psi'_y\geq \tfrac{\epsilon}{2}\ba\})^c$ 
only depends on the fields $\psi^L$ and $\psi'_y$ inside $%\tilde
{D}_y'$. Furthermore, 
on the event $\{\eta'_y\text{ is $\epsilon$-bad}\}^c$ (c.f~ \eqref{eq:EF_events}), we have that $\{\psi^L\geq (1-\frac{\epsilon}2)\ba\}\cap %\tilde
{D}'_y\supseteq \{\psi'_y\geq \frac{\epsilon}2\ba\}\cap %\tilde
{D}'_y$. This gives $\P_N(\tilde{\mathsf F}(\mathcal{D}))\leq\P_N^{L/2}(\mathsf{F}(\mathcal{D})).$ Overall, \eqref{eq:disc1} follows.

%\red The replacement of $\P_N(\tilde{\mathsf F}(\mathcal{D}))$ by $ \P_N^{L/2}(\mathsf{F}(\mathcal{D}))$ would be an equality, weren't it for the fact that the event $\{\text{$y$ is $\epsilon$-disc}\}$ also allows for the possibility that $\eta'_y$ is $\epsilon$-bad, see \eqref{eq:psi_con}. This will be used now.
We now focus on $\mathsf{F}(\mathcal{D})$, which warrants further care, and consider the event,
\begin{equation}\label{eq:disc_b3}
\mathsf{E}^{3}(\mathcal{D})\defeq
\big\{
\exists x\in \tcC 
\text{ s.t. 
$\Con_{y,L'}(\{ \psi'_y \geq \tfrac{\epsilon}{2} \ba\})^c$
for more than $n_x/2$ sites $y\in \mathcal{D}_x$}  
\big\}. 
\end{equation}
Given any $\mathcal{D}= (\mathcal{D}_x)_{x \in \tcC}\in \A^\text{disc}$, let $\mathcal{D}^\prime =(\mathcal{D}_x^\prime)_{x\in \tcC}$ be such that $\mathcal{D}_x^\prime\subseteq \mathcal{D}_x$ and $|\mathcal{D}_x^\prime|\geq \frac{1}{2}|\mathcal{D}_x|$ for all $x \in \tcC$, and consider 
\begin{equation}\label{eq:event_two_harm}
    \mathsf{F}'( \mathcal{D}^{\prime})\defeq \bigcup_{\substack{\bar{\cC}\subseteq\tcC\\|\bar{\cC}|\geq (1-\rho )|\cC|}}\bigcup_{\bar{\cC}^\eta\subseteq \bar{\cC}}
    \bigg(
    \bigcap_{x\in \bar{\cC}^{\eta}}
    \{\eta_x\text{ is ${\epsilon}$-bad}\}\cap \!\!\bigcap_{z\in \bar{\cC}\setminus\bar{\cC}^\eta,y\in \mathcal{D}_z'}
    \{\eta'_y\text{ is ${\epsilon}$-bad}\}\bigg).
\end{equation}
It thus follows from \eqref{eq:disc_b3} and \eqref{eq:event_two_harm} that,
\begin{equation} \label{eq:F3}
\P_N^{L/2}\big(\mathsf{F}(\mathcal{D})\big)
\leq \P_N^{L/2}\big(\mathsf{E}^{3}(\mathcal{D})\big)
+e^{c\frac{nL}{L'M'}}
\sup_{\mathcal{D}^\prime} \P_N^{L/2}\big(\mathsf{F}^\prime(\mathcal{D}^{\prime})\big),
\end{equation}
where we used \eqref{eq:lprop_card} to bound by $  \binom{L/(M'L')}{L/(2M'L')}^{cn}\leq e^{cnL/(M'L')}$ the number of choices for $\mathcal{D}^\prime$ in \eqref{eq:F3}.

We now proceed to control the probability of $\mathsf{E}^{i}(\mathcal{D})$ for $i\in\{1,2,3\}$. By \eqref{eq:A_structure} and \eqref{eq:lprop_position}, we have that $d(\cC,B_{L/2})\wedge d(\bigcup_{x\in\cC}\mathcal{D}_x,B_{L/2})\geq 2ML$. Hence Lemma~\ref{lem:bad_box} with $R=ML$ applies to $\mathsf{E}^{i}(\mathcal{D})$, $i\in\{1,2\}$, exactly as in \eqref{eq:sub_badbox}, up to an additional factor $(\frac{L}{M'L'})^n$ accounting for the choice of $y\in\mathcal{D}_x$ (only needed for $i=2$, cf.~\eqref{eq:lprop_card}), and with $\log(n)^2$ in place of $\log(N/(ML))^2$. Since $\log(N/L)=\log\log|\ba-\ba_*|$ and $n=(\log\log|\ba-\ba_*|)^2$ by \eqref{eq:ML_choice_ub}, this gives, for $|\ba-\ba_*|\geq c(\epsilon)$,
\begin{equation}\label{eq:B12}
    \P_N\big(\mathsf{E}^{i}(\mathcal{D})\big)\leq c\binom{ n }{  \rho n/3} \Big(\frac{L}{M'L'}\Big)^n
\exp\Big\{-\frac{c'\log\log(|\ba-\ba_*|)\,(\epsilon\ba)^2}{\log\log\log(|\ba-\ba_*|)^2}\Big\}.
\end{equation}
For $\mathsf{E}^3(\mathcal{D})$, let $\I^u(B_{L'})$ be the trace of all interlacement trajectories at level $u$ hitting $B_{L'}$ under $\bar{\P}_N^{L'M'}$. In view of \eqref{eq:lprop_mutdist}, $(\psi'_y)_{y\in \mathcal{D}_x}$ are independent fields. Hence, using \eqref{eq:lprop_card} to deduce that the number of relevant $y\in \mathcal{D}_x$ in \eqref{eq:disc_b3} is at least $\tfrac{\Cr{col_num_s}L}{2M'L'}$, and combining \eqref{eq:Iu_closed_cont} with the bound $\text{cap}_N^{M' L'}(B_{L'}) \geq c\log (M')^{-1}$ implied by \eqref{eq:cap-killed}, yields that for $u=(\epsilon \ba/2)^2/2$ (the factor $2^{L/(M'L')}$ below accounts for the choice of the relevant sites $y\in\mathcal{D}_x$ in \eqref{eq:disc_b3}, and is absorbed in the last bound for $|\ba-\ba_*|\geq c$),
\begin{equation}\label{eq:B3}
    \P_N^{L/2}\big(\mathsf{E}^{3}(\mathcal{D})\big)
    \leq cn\,2^{\tfrac{L}{M'L'}}\, \bar{\P}_N^{L'M'}\big(
    \Con_{0,L'}\big(\I^u(B_{L'})\big)^c\big)^{\frac{\Cr{col_num_s}L}{2M'L'}}
    \leq  c n\exp\left\{-\frac{cL}{M'L'}\frac{(\epsilon\ba)^2}{\log(M')}\right\} . 
\end{equation}
%\red (to handle the event involving the diameter, one also notices that $\{\I^u(B_{L'})=\emptyset\}$, which is a possibility, has probability bounded by $\exp\{-u \text{cap}_N^{M' L'}(B_{L'})\}$, and the capacity is bounded by $\log (M')^{-1}$ using \cite[(5.4)]{rz25a}; cf.~also the proof of \eqref{eq:Iu_closed_cont}). 
We now explicate our choice of scales. Recall the choice of $L,M$ from \eqref{eq:ML_choice_ub} and pick 
\begin{equation}\label{eq:disc_scales}
        L'=\frac{N\log({|\ba-\ba_*|})}{{|\ba-\ba_*|}} \quad \text{and}\quad M'=\frac{N}{L'\log(N/L')^3} \text{ (hence $\frac{L}{L'M'}\asymp \log({|\ba-\ba_*|})^2$)}.
\end{equation}
With \eqref{eq:disc_scales}, we first explain how \eqref{A:capacity} and Lemma \ref{lem:block_cap} lead to the precise constant $\vartheta$ appearing in \eqref{eq:lem_disc}, which is related to the event $\mathsf{F}'(\mathcal{D}')$ in \eqref{eq:event_two_harm}.
For any possible choice of $\mathcal{D}^\prime =(\mathcal{D}_x^\prime)_{x\in \tcC}$, it follows from \eqref{eq:block_cap}, \eqref{eq:ML_choice_ub} and \eqref{eq:disc_scales} that for all $x\in \tcC$,
\[
\capacity_N\big(\bigcup_{y\in \mathcal{D}_x^\prime}\bC'_y \big)
\geq c\big( \log({|\ba-\ba_*|})^{-1}\vee \log\log({|\ba-\ba_*|})\big)^{-1}
\geq c \log(N/L)^{-1}.
\]
Moreover, by \eqref{eq:lprop_position} and \eqref{eq:lprop_endist}, we have that $\pi_{\ell_*}(\bigcup_{y\in \mathcal{D}_x^\prime}\bC'_y)\subseteq \pi_{\ell_*}(C_x^L)$ for all $x\in \cC$. Hence for any $\rho <\Cr{cg_rho}$ and ${|\ba-\ba_*|} \geq c(\eta)$, \eqref{A:capacity} applies, with $O_x$ either equal to $C_x$ or to $\bigcup_{y\in \mathcal{D}_x^\prime}\bC'_y$ in what follows. Thus, for every $\bar{\cC}$, $\bar{\cC}^{\eta}$ appearing in the definition of $\mathsf{F}'(\mathcal{D}^{\prime})$ (cf.~\eqref{eq:event_two_harm}), we get that
\begin{equation}\label{eq:cap_disc}
    \capacity_N\Big(\bigcup_{x\in \bar{\cC}^{\eta}}\bC_x\cup \bigcup_{z\in \bar{\cC}\setminus\bar{\cC}^{\eta} ,y\in \mathcal{D}'_z}\bC'_y\Big)\geq \left(\frac{c}{n}\log(N/L)+\frac{1+\eta}{2\vartheta}\right)^{-1}\geq \frac{2\vartheta}{1+2\eta},
\end{equation}
assuming that ${|\ba-\ba_*|}\geq c(\eta)$ and
recalling that $n=N/ML$ and \eqref{eq:ML_choice_ub} for the last bound. The control on the event $ \mathsf{F}'( \mathcal{D}^{\prime})$ is thus a direct consequence of Lemma \ref{lem:harm_avg}, as we now explain. By \eqref{eq:ML_choice_ub} and \eqref{eq:disc_scales}, we know that $\log(N/(ML))/\log(N/L)$, $\log(N/(M'L'))/\log(N/L')$ and $(M')^{-1}$ can all be made arbitrarily small (in particular,~smaller than $\Cr{nml_eta} \wedge \Cr{ml_eta}$, which depend implicitly on $k \geq 1$ and $\eta \in (0,1)$) by making ${|\ba-\ba_*|}$ large. Hence by \eqref{eq:A_structure}, \eqref{eq:lprop_mutdist} and \eqref{eq:lprop_position}, we may apply Lemma \ref{lem:harm_avg} with $\cC=\bar{\cC}^{\eta}$ and $\cC'=\bigcup_{z\in \bar{\cC}\setminus \bar{\cC}^{\eta}}\mathcal{D}'_z$ and get for all ${|\ba-\ba_*|} \geq c(\eta, k)$,
\begin{equation} \label{eq:tF3}
\P_N^{L/2}\big(\mathsf{F}'( \mathcal{D}^{\prime}) \big)\leq e^{cn}
\exp\left\{-\frac{\vartheta}{1+\eta}\bigg((1-\epsilon)|\ba|-c'\sqrt{\frac{\log\log({|\ba-\ba_*|})^2\log({|\ba-\ba_*|})^6}{2\vartheta}}\bigg)_+^2\right\},
\end{equation}
where we also used bound
$|\bar{\cC}^{\eta}\cup\bigcup_{z\in \bar{\cC}\setminus\bar{\cC}^{\eta}}\mathcal{D}'_z|\leq \frac{cnL}{L'M'}\leq c\log\log({|\ba-\ba_*|})^2\log({|\ba-\ba_*|})^2$ and the fact that the number of choices for $\bar{\cC}$ and $\bar{\cC}^\eta$ is at most $c\binom{n}{(1-\rho)n}2^n\leq e^{cn}$. To assemble all the preceding estimates, first note that, for ${|\ba-\ba_*|}^2>c(\rho)$ and $k \geq 1$,
\[e^{n\frac{L}{M'L'}}\vee \binom{n}{\rho n/3}\vee \Big(\frac{L}{M'L'}\Big)^n\vee n \vee e^{n} \leq \exp\Big\{\frac{cnL}{L'M'}\log\Big(\frac{kN}{L'}\Big)\Big\},\]
hence in view of \eqref{eq:Adisc_entropy}, we may absorb all the corresponding counting factors appearing in \eqref{eq:F3}, \eqref{eq:B12}, \eqref{eq:B3} and \eqref{eq:tF3} into the prefactor $\A^{\text{disc}}$ appearing in \eqref{eq:disc1}. To conclude, by picking $\epsilon<c(\eta), \rho <\Cr{cg_rho}(\eta)$ small enough, taking ${|\ba-\ba_*|^2}\geq \Cr{disc}(\eta,k)$ with $\Cr{disc}$ sufficiently large, combining \eqref{eq:disc1}, \eqref{eq:F3}, \eqref{eq:B12}, \eqref{eq:B3} and \eqref{eq:tF3}, we obtain that the left-hand side of \eqref{eq:lem_disc} is bounded by
\begin{multline*}
%&\P_N\Big(
%    0\stackrel{\geq \ba}{\longleftrightarrow}\partial B_\xi, \,
%    \bigcap_{x\in \tcC}
%    \Big\{
%    C_x\hspace{-1ex}  \stackrel{\quad \geq {\ba}}{\mathrlap{\longleftrightarrow}{\mkern15mu{\text{\scriptsize{/}}}}}  \,\partial B_{kN} \text{ or }
%    \eta'_x\text{ is $\epsilon/2$-bad}
%    \Big\}
%    \Big)\\
%\leq &
f^L_\xi(|\ba-\ba_*|) \exp\{c\log\log({|\ba-\ba_*|})^2\log({|\ba-\ba_*|})^2\log(k{|\ba-\ba_*|})\}
\times \\
\bigg(\exp\Big\{\!-\!\frac{c'\log\log({|\ba-\ba_*|})\epsilon^2\ba^2}{\log\log\log({|\ba-\ba_*|})^2}\Big\}
 +\exp\{-\log({|\ba-\ba_*|})\epsilon^2\ba^2\}+\exp\{-\ba^2(1+ \eta)^{-1} \vartheta \}\bigg),%\\
%\leq &|a|\exp\left\{\!-\frac{1-\eta}{2} \ba^22\vartheta\!\right\}.
\end{multline*}
from which \eqref{eq:lem_disc} readily follows.
\end{proof}

\section{Porous sets in two dimensions and coarse-graining} \label{sec:tube}
In this section, we seek estimates on the capacity of ``porous sets'' with explicit constant. These sets correspond in our previous arguments to the union of balls (of different size) that carry ``bad'' harmonic averages. We refer to Figure~\ref{fig:cg_disc} for an illustration, and to the sets involved in \eqref{eq:cap_disc} for a concrete example. These sets ultimately need to track the correct exponent $\vartheta$
(see for instance the statement of Lemma~\ref{lem:disc}).
We give in Proposition~\ref{prop:tube} below generic upper and lower bounds on such sets. These are estimates that concern the random walk $X$  (see Section~\ref{sec:facts}), which are of independent interest.
We then present the proof of Proposition~\ref{prop:coarse_graining}, which appears at the end of this section, and which Proposition~\ref{prop:tube}  feeds into.

 For the time being, ``porous sets''  are simply sets of the forms $\bigcup_{i\in A} S_i$, where $A$ is some indexing set and $S_i$ are sufficiently visible for $X$. A useful picture to keep in mind is that given $N\geq P\geq 1$, we place $P$ blocks of scale $N/P$ along a path of length $N$.  In particular, the usual line segment of length $N$ can be thought of as a porous set with $P=N$.
 As explained in Section~\ref{sec:intro_proof}, Proposition~\ref{prop:tube} sharpens the planar bounds of \cite{rz25a} (cf.~\cite{goswami_radius_2022,prevost_first_2025} for transient setups) by keeping track of the proximity between the sets $S_i$, parametrised by $\delta$, which is key to obtain the sharp constant in \eqref{eq:thm-asymp}. We write $|\cdot|$ for the Euclidean distance and $d(A,B)=\inf_{a\in A, b\in B}|a-b|$.

\begin{proposition} \label{prop:tube}
For all $\epsilon,\delta\in(0,1)$, there exists $\Cl[c]{tube_nP}\in(0,1),c_\star>0$ and $\Cl[c]{tube_N}<\infty$, all depending on $\epsilon,\delta$, such that for all $N \geq \Cr{tube_N}$, integers $P\in[\Cr{tube_N},N],A\subset \{1,\dots,P\}$ with $n=|A|$ satisfying $\Cr{tube_nP}P\leq n\leq P$, and sets $(S_i)_{i\in A} \subset  B_{2N}$, the following hold:
    \begin{enumerate}[label={(\roman*)}]
    \item if $d(S_i,S_j)\geq (|i-j|-\delta)N/P$ for all $i\neq j\in A$ and $\capacity_N(S_i)\geq \kappa$ for each $i\in A$ and some $\kappa>0$,
        \begin{align}
            \capacity_N\bigg(\bigcup_{i\in A} S_i \bigg)
            &\geq 
            \left[ \big({c_\star\kappa P}\big)^{-1}+(1+\varepsilon)(2\vartheta)^{-1} \right]^{-1};
          \label{eq:tube_lb_macro}
        \end{align}
    \item if $ |x_1-x_2|\leq (|i-j|+1/\delta)N/P$
     for all $i,j\in A,x_1\in S_i$, $x_2\in S_j$, 
        \begin{align}
            \capacity_N\bigg(\bigcup_{i\in A} S_i \bigg)
            &\leq (1+\varepsilon)\, 2\vartheta.
        \label{eq:tube_ub_macro}
            \end{align}
    \end{enumerate}
\end{proposition}

 To prepare for the proof of Proposition~\ref{prop:tube}, we first approximate the variational problem in \eqref{eq:tau} by probability measures having a continuous density with respect to the Lebesgue measure. Recall $\mathscr{E}(\mu)$ from \eqref{eq:energy-form}. One can show that $\frac2\pi K_0(2|\cdot|)$ with $K_0$ given by \eqref{eq:K_0} is the Green function of planar Brownian motion with variance $1/2$ at time $1$, killed at an independent exponential time with mean one, and so $\mathscr{E}(\mu)$ is twice the corresponding energy of $\mu$. In particular, using the semigroup property and reversibility one readily finds that $\mathscr{E}$ is positive definite. Finite Borel measures $\mu$ on $\R$ are tacitly identified with $\R\times\{0\}\subset\R^2$ below.

%$\mathscr{E}(\mu,\nu)=\iint \frac2\pi K_0(2|x-y|)\,\mu(dx)\,\nu(dy)\in[0,\infty]$, so that $\mathscr{E}(\mu,\mu)=%

\begin{lemma}\label{lem:smooth_density}
For all $\varepsilon\in(0,1)$ there exists a continuous density $f_\varepsilon:[0,1]\to\R_+$ ($\int f_\varepsilon=1$) such~that 
\begin{equation} \label{eq:bessel_tilted_energy}
% \frac2{\pi}   \int_0^1 \int_0^1  K_0\big({2}|x-y|\big) f_\varepsilon(x)f_\varepsilon(y) \,dx\,dy
     \mathscr{E}(f_\varepsilon (x)dx)  \leq{(1+\epsilon)} \vartheta^{-1}. 
\end{equation}
\end{lemma}
 
\begin{proof} Since $\vartheta>0$ (see below \eqref{eq:tau}), by \eqref{eq:tau} there exists $\mu\in\mathcal{P}([0,1])$ with $\mathscr{E}(\mu)\leq \sqrt{1+\varepsilon}\vartheta^{-1}$. (In fact this variational problem has a unique minimiser $\mu^*$, which is explicit, see \cite[Theorem 4.9, p.76 and (4.38)]{sznitman_brownian_1998}, but we will not need this.)
. As $\mathscr{E}(\mu)<\infty$ and $K_0(0)=\infty$, the measure $\mu$ has no atoms, hence $\mu([h,1-h])\to\mu([0,1])=1$ as $h\downarrow 0$.
For $h\in(0,\frac12)$, let $b_h\geq 0$ be continuous, supported in $[-h,h]$ (a `bump' function), with $\int b_h=1$, and let $g_h(x)=\int b_h(x-a)\,\mu(da)$, $x\in\R$. Since $b_h$ is bounded and uniformly continuous, $g_h$ is continuous and bounded. Moreover, since $b_h(\cdot-a)$ is supported in $[0,1]$ whenever $a\in[h,1-h]$,
we have $$m_h \coloneqq\int_0^1 g_h(x)\,dx \geq\int_{[h,1-h]}\int_0^1b_h(x-a)\,dx\,\mu(da)=\mu([h,1-h]),$$ so we may fix $h=h(\varepsilon)$ such that $m_h>0$ and $m_h^{-2}\leq \sqrt{1+\varepsilon}$. Let $f_\varepsilon \coloneqq m_h^{-1}g_h|_{[0,1]}$, which is continuous and nonnegative on $[0,1]$ with $\int_0^1 f_\varepsilon=1$. Writing $\mu_a=\mu(\cdot-a)$ for $a\in\R$ and using Fubini's theorem, one has the rewrite
\begin{align*}
\mathscr{E}(f_\varepsilon(x)dx) &= m_h^{-2}\iint_{\R^2} b_{h}(a) b_{h}(b)\,\mathscr{E}\big(\mu_a|_{[0,1]},\mu_b|_{[0,1]}\big)\,da\,db. 
\end{align*}
Applying the Cauchy-Schwarz inequality in the form $\mathscr{E}(\mu,\nu)\leq \mathscr{E}(\mu)^{1/2}\mathscr{E}(\nu)^{1/2}$, for
 $\mathscr{E}(\mu),\mathscr{E}(\nu)<\infty$ (see \cite[Ch.~2]{fukushima_dirichlet_2011} for a proof), and the translation invariance $\mathscr{E}(\mu_a)=\mathscr{E}(\mu)$ to the previous display give %Since %$K_0\geq 0$, restricting $g_h(x)dx$ to $[0,1]$ does not increase its energy, and so, by Step~2,
\begin{align*}
\mathscr{E}(f_\varepsilon(x)dx)% &= m_h^{-2}\iint_{\R^2} b_{h}(a) b_{h}(b)\,\mathscr{E}\big(\mu_a|_{[0,1]},\mu_b|_{[0,1]}\big)\,da\,db \\
%&\leq m_h^{-2}\iint_{\R^2} b_{h}(a) b_{h}(b)\,\mathscr{E}\big(\mu_a|_{[0,1]}\big)^{1/2}\mathscr{E}\big(\mu_b|_{[0,1]}\big)^{1/2} \,da\,db \\
&{\leq m_h^{-2}\iint_{\R^2} b_{h}(a) b_{h}(b)\,\mathscr{E}(\mu_a)^{1/2}\mathscr{E}(\mu_b)^{1/2} \,da\,db}
 = m_h^{-2}\,\mathscr{E}(\mu) \leq(1+\varepsilon)\vartheta^{-1},
\end{align*}
where the last bound follows by choice of $\mu$ and $h$.
\end{proof}

\begin{proof}[Proof of Proposition~\ref{prop:tube}]
In what follows, for $\varepsilon>0$ as appearing in the statement of Proposition~\ref{prop:tube}, let $f_\varepsilon$ be given by Lemma~\ref{lem:smooth_density}.
For $n \geq 1$ and  $i\in\{1,\dots,n\}$, let $Z_n= \frac{1}{n}\sum_{i=1}^{n}  f_\varepsilon\left(\frac{i}{n}\right)$. By a Riemann sum argument, for $n \geq C(\varepsilon)$ we have $Z_n^{-1}\leq 1+\epsilon$. Let
\begin{equation} \label{eq:discretised_bessel_em}
    \mu_n\defeq  \frac{1}{Z_n n} \sum_{i=1}^n  f_\varepsilon\left(\frac{i}{n}\right)\delta_{\frac in}.
\end{equation}
It is easy to see that $\mu_n\in\mathcal{P}([0,1])$ and $\mu_n$ converges weakly to the measure with density $f_\varepsilon$ as $n\to\infty$. For $\mu$ a probability measure with finite support in $\Z^2$, let $\langle \mu, G_N \mu \rangle= \sum_{x,y} g_N(x,y) \mu(x)\mu(y)$. By the variational principle \cite[(1.61)]{sznitman_topics_2012}, $\capacity_N(S)^{-1}=\inf_{\mu\in\mathcal{P}(S)} \langle \mu, G_N \mu \rangle$, where $S =\bigcup_{i\in A} S_i$. Hence it suffices to show that, under the assumptions of~(i) and~(ii) of Proposition~\ref{prop:tube}, respectively,
\begin{align}
\inf_{\mu\in\mathcal{P}(S)} \langle \mu, G_N \mu \rangle
&\leq \frac{1}{c_\star\kappa P}+(1+\epsilon)(2\vartheta)^{-1}, \label{eq:tube2_lb_macro}\\
\inf_{\mu\in\mathcal{P}(S)} \langle \mu, G_N \mu \rangle
&\geq (1-\epsilon)(2\vartheta)^{-1}. \label{eq:tube2_ub_macro}
\end{align}
 We first show \eqref{eq:tube2_lb_macro}. To obtain an upper bound, we consider measures $\mu\in \cP(S)$ of the form
\begin{equation} \label{eq:tube_meas}
    \mu(x)= \frac{1}{n}\sum_{i=1}^n f\left(\frac{i}{n}\right)\bar{e}_{N,S_i}(x),
\end{equation} 
where $f$ is a free to choose, nonnegative, bounded and continuous function on $[0,1]$ such that $\sum_{x\in S}\mu(x)=1$, and $\bar{e}_{N,S_i}(\cdot)= \capacity_N(S_i)^{-1}e_{N,S_i}(\cdot)$ denotes normalised equilibrium measure, a probability measure on $S_i$.
It follows upon splitting the sum over $i,j\in \{1,\dots,n\}$ below over diagonal and off-diagonal contributions, applying \eqref{eq:last_exit} to the diagonal part to perform the sum over $x_2$ (say), and using the standing assumption $\capacity_N(S_i)\geq \kappa$ for all $i\in A$ and $n\geq \Cr{tube_nP}P$, that for $\mu$ of the form \eqref{eq:tube_meas},
\begin{align} \label{eq:tube_lb_recipe}
\begin{split}
   \langle \mu, G_N \mu \rangle \leq  \frac{\Vert f \Vert_{\infty}^2}{\Cr{tube_nP}P\kappa} + \frac{1}{n^2}\sum_{\substack{i,j=1 \\ i\neq j }}^n \sup_{x_1\in S_i,x_2\in S_j} g_N\left(x_1,x_2\right)f\left(\frac{i}{n}\right)f\left(\frac{j}{n}\right).
\end{split}
\end{align}
We will soon see that the calculations to estimate the first off-diagonal terms in \eqref{eq:tube_lb_recipe} are similar to those in the proofs of \cite[Proposition 3.1]{rz25a} (cf.~in particular the proof of (3.8)). One however can't simply recycle the estimates in \cite{rz25a} due to the extra dependence on $\delta$ in the current setting.

We now choose $f(\cdot)=f_\varepsilon(\cdot)/Z_n$ with $f_\epsilon$ so that \eqref{eq:bessel_tilted_energy} holds. Since we may assume that $n \geq c(\varepsilon)$ without loss of generality, it follows that $Z_n \geq 1/2$ for such $n$ hence $\Vert f_\varepsilon/Z_n\Vert_{\infty}\leq c'(\varepsilon)$ since $f_\varepsilon$ is bounded uniformly on $[0,1]$. Overall this yields in view of \eqref{eq:tube_lb_recipe} the first term on the right-hand side of \eqref{eq:tube2_lb_macro}. With our choice of $f(\cdot)$, and in view of $\mu_n$ in \eqref{eq:discretised_bessel_em}, the second term on the right-hand side of \eqref{eq:tube_lb_recipe} can be recast as
$$
\sum_{\substack{i,j=1: i \neq j }}^n \sup_{x_1\in S_i,x_2\in S_j} g_N\left(x_1,x_2\right)\mu_n\left(\frac{i}{n}\right)\mu_n\left(\frac{j}{n}\right).
$$
We will consider near-diagonal and off-diagonal contributions separately. By assumption, for $x_1\in S_i,x_2\in S_j$ and $i \neq j$, with $\Cr{tube_nP}$ to be determined momentarily, we have that 
 \begin{equation}\label{eq:dist-lb2pt}
 |x_1-x_2| \geq d(S_i,S_j)\geq (|i-j|-\delta)\frac NP \geq \bigg(  \frac{\Cr{tube_nP}|i-j|}{n} -\frac{\delta}{P}\bigg) N.
 \end{equation}
 Let $t\in (0,1)$, eventually to be chosen as a function of $\delta, \varepsilon$ only.  Note in particular that $|x_1-x_2| > 0$ and hence $|x_1-x_2|\geq 1$ whenever $i \neq j$. We first consider $i,j$ such that $|i-j| > tn$. 
For such $i,j$ and all $P\geq \frac{\delta}{\Cr{tube_nP}\Cr{k0}(\epsilon)t}$, using Lemma~\ref{lem:bessel_cont} and the fact that $K_0(\cdot)$ is decreasing implies
 \begin{align*}
K_0\left({2}\,\frac{|x_1-x_2| \vee 1}{N}\right)& = K_0\left({2}\,\frac{|x_1-x_2| }{N}\right)  \stackrel{\eqref{eq:dist-lb2pt}}{\leq} K_0\left({2}\, \bigg(\frac{\Cr{tube_nP}|i-j|}{n} -\frac{\delta}{P} \bigg) \right) \\
&\leq (1+\varepsilon) K_0\left(\frac{2|i-j|}{n}-\frac{2(1-\Cr{tube_nP})|i-j|}{n}\right)
\leq (1+\varepsilon)^2 K_0\left(\frac{2|i-j|}{n}\right),
 \end{align*}
where the last bound follows again from Lemma~\ref{lem:bessel_cont} by taking $\Cr{tube_nP}\geq 1-\Cr{k0}(\epsilon)$. 
Now in view of, \eqref{eq:discretised_bessel_em}, and using the facts that $(x,y)
    \in [0,1]^2\mapsto  K_0\left({2}|x-y|\right) Z_n^{-2}{f_\varepsilon(x) f_\varepsilon(y)}\1_{\{|x-y|> t\}}$ is piecewise continuous and bounded uniformly in $n$ (recall that $t=t(\varepsilon)$ is fixed), and combining the previous estimate for $K_0$ with the asymptotics \eqref{eq:g2-asymp-unif}, one obtains, for all $N\geq c'(t,\varepsilon,\delta)$, $P \geq c(t,\varepsilon,\delta)$ and $n \geq \Cr{tube_nP}(\epsilon)P$ that
    \begin{align} \label{eq:tube2_macro_lb_main}
    \begin{split}
        &\sum_{\substack{i,j=1 \\ |i-j|>tn }}^n \sup_{x_1\in S_i,x_2\in S_j} g_N\left(x_1,x_2\right)\mu_n(i/n)\mu_n(j/n)\\
        &\quad\leq\frac{2(1+\epsilon)}{\pi }\sum_{i,j=1}^n K_0(2|i-j|/n)\1_{\{|i-j|>t n\}}\mu_n(i/n)\mu_n(j/n) \\
         & \quad \leq \frac{2(1+\epsilon)^2}{\pi } \int_0^1 \int_0^1  K_0\big({2}|x-y|\big) f_\varepsilon(x)f_\varepsilon(y) \,dx\,dy
        \stackrel{\eqref{eq:bessel_tilted_energy}}{\leq} 
        (1+\epsilon)^3(2\vartheta)^{-1}.
    \end{split}
    \end{align}
    To deal with the near-diagonal contributions, i.e. $i,j$ such that $ 1\leq |i-j|< tn$, first note that by observing the proof of Lemma~\ref{lem:bessel_cont}, one has that $\tfrac{K_0(t-\zeta)}{K_0(t)} \leq c$ uniformly in $t<1$ and $\zeta\leq t/2$. Using this and the upper bound on $g_N$ stemming from \eqref{eq:g2-asymp-unif} applied with $\varepsilon=1$, one obtains  that for all $x_1 \in S_i$, $x_2 \in S_j$, which satisfy \eqref{eq:dist-lb2pt}, one has $g_N(x_1,x_2) \leq c' K_0(\tfrac{|i-j|}{P})$. 
    Applying this and using the fact that $    \sum_{k=1}^{tn} K_0\left({k}/{P}\right) \leq 2 tn K_0(\Cr{tube_nP}t)$ for $t \le c$ and $n \geq \Cr{tube_nP}(\epsilon)P$, which follows from \cite[Lemma 3.4]{rz25a}, and since $f_\varepsilon$ is bounded on $[0,1]$ by $C(\varepsilon)$ and $Z_n^{-1}\leq 2$ for large enough $n$, it follows that for $N, P, n \geq C$ and $t \leq c$,
  \begin{equation}\label{eq:tube2_macro_lb_error}
        \sum_{\substack{i,j=1 \\ 1\leq |i-j|\leq tn }}^n \sup_{x_1\in S_i,x_2\in S_j} g_N\left(x_1,x_2\right)\mu_n\left(\frac{i}{n}\right)\mu_n\left(\frac{j}{n}\right) \leq C''(\varepsilon)\frac1{n} \sum_{k=1}^{tn} K_0\left(\frac{k}{ P}\right)
        \leq 2C''(\varepsilon)t K_0(\Cr{tube_nP}t).
    \end{equation}
    By \eqref{eq:bessel_ub} and using that $\lim_{t\downarrow 0} t \log (1/t) =0$, one can choose $t=t(\varepsilon)$ small enough so that the right-hand side of \eqref{eq:tube2_macro_lb_error} is at most $\varepsilon(2\vartheta)^{-1}$. Combining with \eqref{eq:tube2_macro_lb_main}, and feeding the resulting bound on $\sum_{\substack{i,j=1: i \neq j }}^n \sup_{x_1\in S_i,x_2\in S_j} g_N\left(x_1,x_2\right)\mu_n\big(\tfrac{i}{n}\big)\mu_n\big(\tfrac{j}{n}\big)$ into \eqref{eq:tube_lb_recipe}, \eqref{eq:tube2_lb_macro} follows.

    For \eqref{eq:tube2_ub_macro}, it suffices to consider the case $n=P$ by the monotonicity of capacity, cf.~\eqref{eq:tube_ub_macro}. Note that we may w.l.o.g.~assume $(S_i)_{1\leq i\leq P}$ is pairwise disjoint. Indeed, one can let $\tilde S_1=S_1,\tilde S_i=S_i\setminus \bigcup_{j=1}^{i-1}\tilde S_j$ for $2\leq i\leq P$ and note that our assumption $|x_1-x_2|\leq (|i-j|+1/\delta)N/P$ for all $i,j\in A,x_1\in  S_i$ and $x_2\in S_j$ continues to hold with $S_i$ replaced by $\tilde S_i$ since $\tilde S_i\subset S_i$ for all $1\leq i\leq P$. Let $\mu$ be a probability measure on $S$, and define the measure $\bar{\mu}$ on $[0,1]$ as
    \begin{equation}
    \label{eq:nu-def}
    \bar\mu(A) = \sum_{i=1}^{P}\bigg(\sum_{x\in S_i}\mu(x)\bigg)\delta_{\frac iP}(A), 
    \end{equation} 
    for $A\subset [0,1]$ a Borel set. Since the sets $S_i$ are disjoint by assumption, $\bar{\mu}$ defines a probability measure. {For $t\in(0,1)$, consider the capped kernel $k_t(r)\coloneqq \frac{2}{\pi}K_0\big(2(r\vee t)\big)$, $r\geq 0$, which is continuous on $[0,1]$, and increases pointwise to $\frac{2}{\pi}K_0(2r)$ as $t\downarrow 0$ since $K_0$ is decreasing. Write $\mathscr{E}_t(\nu)= \int_0^1 \int_0^1  k_t\big(|x-y|\big) \,d\nu(x)\,d\nu(y)$ for the corresponding quadratic form. We first claim that for every $\varepsilon>0$ there exists $t=t(\varepsilon)\in(0,1)$ such that
    \begin{equation} \label{eq:lb-cont}
    E_t \defeq \inf_{\nu} \mathscr{E}_t(\nu)  \geq (1-\varepsilon)^{\frac13} (2\vartheta)^{-1},
    \end{equation}
    where the infimum is over all probability measures $\nu$ on $[0,1]$.
Indeed, since $k_t$ is continuous and bounded and the set of probability measures on $[0,1]$ is weakly compact, the infimum is attained at some $\nu_t$, and $t\mapsto E_t$ is non-increasing and bounded from above by $(2\vartheta)^{-1}$ in view of \eqref{eq:energy-form}--\eqref{eq:tau}, since $k_t\leq\frac2\pi K_0(2|\cdot|)$ and $\iint\frac2\pi K_0(2|x-y|)\,d\nu\,d\nu=\frac12\mathscr{E}(\nu)$. Pick $t_m\downarrow 0$ such that $\nu_{t_m}\to\nu$ weakly. For fixed $s\in(0,1)$ and all $m$ with $t_m\leq s$, we have $k_{t_m}\geq k_s$, whence $E_{t_m}\geq \iint k_s\,d\nu_{t_m}d\nu_{t_m}\to \iint k_s\,d\nu\,d\nu$ as $m\to\infty$. Letting $s\downarrow 0$ and using monotone convergence, $\lim_{t\downarrow0}E_t\geq \frac12\mathscr{E}(\nu)\geq (2\vartheta)^{-1}$, which yields \eqref{eq:lb-cont}.

The parameter $t$ will now be leveraged to compare $k_t$ with the discrete kernel $g_N$ via \eqref{eq:g2-asymp-unif}, \emph{including} near the diagonal. Let $x_1 \in S_i$, $x_2 \in S_j$, $i,j\in\{1,\dots,P\}$ (possibly $i=j$), and set $r=\frac{|i-j|}{P}\vee t$. By assumption and since $P\leq N$, we have $\frac{|x_1-x_2|\vee 1}{N}\leq \frac{|i-j|+1/\delta}{P}\leq r+\frac{1}{\delta P}$. Hence, as $K_0$ is decreasing, applying Lemma~\ref{lem:bessel_cont} with $2r\in(0,2]$ in place of $t$ and $\zeta=\frac{2}{\delta P}$, which satisfies $\zeta\leq \Cr{k0}\cdot 2r$ as soon as $P\geq (\Cr{k0}\delta t)^{-1}$ (with $\Cr{k0}$ chosen for the precision $1-(1-\varepsilon)^{\frac13}$), we obtain
\begin{equation}
\label{eq:lb-cont-dis}
k_t\Big(\frac{|i-j|}{P}\Big)=\frac2\pi K_0(2r) \leq (1-\varepsilon)^{-\frac13}\, \frac2\pi K_0 \Big( {2}\frac{|x_1-x_2|\vee 1}{N} \Big)\leq (1-\varepsilon)^{-\frac23}\, g_N(x_1,x_2),
\end{equation}
where the last bound follows from \eqref{eq:g2-asymp-unif} (applied with $k=4$, recalling that $S\subset B_{2N}$) whenever $N \geq c(\varepsilon)$. Substituting $\bar{\mu}$ from \eqref{eq:nu-def} into \eqref{eq:lb-cont} with $\nu=\bar\mu$ and using \eqref{eq:lb-cont-dis}, we get
\begin{multline*}
(1-\varepsilon)^{\frac13} (2\vartheta)^{-1}\leq \sum_{i,j=1}^{P}\bigg(\sum_{x_1\in S_i}\mu(x_1)\bigg) \bigg(\sum_{x_2\in S_j}\mu(x_2)\bigg)  k_t \Big( \frac{|i-j|}{P} \Big) \\
\leq (1-\varepsilon)^{-\frac23} \sum_{i,j=1}^{P}\sum_{x_1\in S_i} \sum_{x_2\in S_j}  g_N(x_1,x_2) \mu(x_1)\mu(x_2) = (1-\varepsilon)^{-\frac23} \langle \mu, G_N \mu \rangle,
\end{multline*}
for all $P\geq c(\varepsilon,\delta)$ and $N\geq c(\varepsilon)$. Since $\mu\in\mathcal{P}(S)$ was arbitrary, \eqref{eq:tube2_ub_macro} follows.}
\end{proof}

Using Proposition \ref{prop:tube}, we now prove Proposition~\ref{prop:coarse_graining}, the coarse-graining scheme for paths, which played a prominent role in Section~\ref{sec:ub} and was admitted there. We refer to Figure \ref{fig:cg_disc} for an illustration.

\begin{proof}[Proof of Proposition \ref{prop:coarse_graining}]
By translation invariance we may assume $x=0$.
Let $L\geq 1,M\geq 100,N\geq 100ML$ and $\gamma\subset\tilde\Z^2$ be connected path from $B_L$ to $B_{N-2L}^c$, stopped at its first exit point of $B_{N-2L}$. For any set $A\subset \R^2$, define its $\Lambda=\Lambda(L)$-paving by
$    \text{Pav}(A)= \{x\in \Lambda(L):B^{2}(x,L)\cap  A\neq \emptyset\}$,
where $B^{2}(x,L)$ is the Euclidean ball in $\R^2$.
We now delineate how the set $\cC$ is constructed from $\gamma$.
One finds a point $x_*\in \text{Pav}(\partial B_{N-2L})$ such that the first exit point of $\gamma$ from $\tB_{N-2L}$ lies in $\tilde C_{x_*}^L$. In particular,
$\gamma \text{ crosses }D_{x_*}^L\setminus C_{x_*}^L$.   
Now consider the line $\ell_*\subset \R^2$ joining $0$ and $x_*$. Let $n_*=\lfloor\frac{N}{20ML}\rfloor$ and let $\ell_0,\dots,\ell_{n_*-1}$ be the lines perpendicular to $\ell_*$ such that
\begin{align}
%\label{eq:cg_l0}
%    &\text{},\\
\label{eq:cg_mutdist}
    &\text{$\ell_0$ passes through $0$, $d(\ell_i,\ell_{i+1})=\tfrac{N}{n_*}$ for all $0\leq i\leq n_*-2$},\\
\label{eq:cg_position}
    &\text{all points in $\{\pi_{\ell_*}(\ell_1),\dots,\pi_{\ell_*}(\ell_{n_*-1})\}$ lie in-between $0$ and $x_*$ along $\ell_*$}.
\end{align}
For each $i\in\{1,\dots,n_*-1\}$, since $\gamma$ starts in $B_L$ and exits $ B_{N-2L}$ by a point in $\tilde{C}_{x_*}^L$, $\gamma$ must intersect $\ell_i$, hence one can find a point $x_i\in \text{Pav}({\ell}_i)$ such that $\gamma$ crosses $D_{x_i}^L\setminus C_{x_i}^L$. If such $x_i$ is not unique, break ties by taking $x_i$ with smallest distance to $\ell_*$.
We then define $\A=\A^{L,M}_{0,N}$ as the family consisting of all collections $\cC=(x_1,\dots,x_{n_*-1},x_*)$ that can be obtained in this way.

We proceed to verify the properties listed in Proposition \ref{prop:coarse_graining}. 
Note that the fact that \eqref{A:cardinality} and \eqref{A:connectivity} hold and that  $\cC\subset B_N$ is immediate from the construction. 
For \eqref{A:entropy}, there are at most $|\text{Pav}({\partial B_{N-2L}})|\leq c(N/L)$ choices for $x_*$. For any fixed $x_*$, the corresponding family of lines $\ell_*,\ell_0,\dots,\ell_{n_*-1}$ is deterministic. And for each $i\in\{1,\dots,n_*-1\}$ there are at most $|\text{Pav}({\ell}_i)\cap B_{N-2L}|\leq c(N/L)$ choices for $x_i$, hence $|\A|\leq c\frac{N}{L}\times (\frac{N}{L})^{n_*-1}\leq \exp\{cn\log(N/L)\}$.

It only remains to check \eqref{eq:A_structure}, \eqref{A:capacity} and that $\cC$ satisfies \eqref{def:C}. Note that $\tfrac{N}{n_*}\geq 20ML$ and $L(\tfrac{N}{n_*})^{-1}\leq \frac{1}{20}$.
For any $1\leq i\leq n_*$, let $O_{x_i}$ be such that $\pi_{\ell_*}(O_{x_i})\subseteq \pi_{\ell_*}(C_{x_i}^L)$. It then follows from \eqref{eq:cg_mutdist} that for all $i\neq j\in \{0,1,\dots,n_*-1\},$
\begin{equation}\label{eq:tube_dist1}
d(O_{x_i},O_{x_j})\geq \tfrac{N}{n_*}|i-j|-4L
\geq \tfrac{N}{n_*}(|i-j|-4L(\tfrac{N}{n_*})^{-1})
\geq \tfrac{N}{n_*}(|i-j|-\tfrac{1}{5}).
\end{equation}
Additionally, for any $i\in \{0,1,\dots,n_*-1\}$
\begin{equation}\label{eq:tube_dist2}
d(O_{x_i},O_{x_{n_*}})\geq N-4L-i\tfrac{N}{n_*}-2L=\tfrac{N}{n_*}(n_*-i)-6L\geq \tfrac{N}{n_*}(|n_*-i|-\tfrac{3}{10}).
\end{equation}
The fact that $\cC$ satisfies \eqref{def:C} follows from \eqref{eq:tube_dist1} and \eqref{eq:tube_dist2} since they imply that for all $i\neq j\in\{1,\dots,n_*\}$, $|x_i-x_j|\geq 20ML-4L\geq 16ML$ (by the same reasoning applied to the centres $x_i$, whose projections onto $\ell_*$ lie within distance $L$ of $\pi_{\ell_*}(\ell_i)$, resp.~beyond $N-3L$ for $x_{n_*}=x_*$). Similarly, for all $i\in\{1,\dots,n_*\}$, \eqref{eq:tube_dist1} and \eqref{eq:tube_dist2} yield that $d(B_L,x_i)\geq \tfrac{N}{n_*}i-4L\geq 20MLi-4L\geq 16MLi.$ For \eqref{A:capacity},
let $\tcC\subset \cC$ be such that $|\tcC|\geq (1-\rho) |\cC|$. 
For any $\eta\in(0,1)$, by taking $N\geq \Cr{NML}(\eta)ML$ and $\rho\leq \Cr{cg_rho}(\eta)$, we can ensure that $N\geq n_*\geq \Cr{tube_N}(\eta)$ and $n_*\geq |\tcC|\geq \Cr{tube_nP}n_*$. \eqref{A:capacity} thus follows from applying Proposition \ref{prop:tube} to $\capacity_N\big(\bigcup_{x\in \tcC}O_x\big)$ with $\kappa=c\log(N/L)^{-1}$, $P=n_*$ and $\delta=3/10$.
\end{proof}

\section{Lower bounds}\label{sec:lb}
In this section we derive companion lower bounds on $\theta_N(a)$ to those of Theorem~\ref{thm:upper_bound}, see Theorem~\ref{prop:lower_bound} below. The strategy is outlined in the paragraph \emph{Lower bounds} of Section~\ref{sec:intro_proof}; we now describe how it is implemented. The near-critical regime is handled by Lemma~\ref{L:nclb}, which relies on the RSW-type estimate \eqref{eq:loop_closed_cont}. Away from criticality, the key input is Proposition~\ref{prop:lb_general}, a connection estimate inside a tube for the field killed outside a set $\sK$, which also produces the critical cost $\theta_{\xi}(a_*)$ by conditioning on the cluster of the origin in $\tB_{\xi_\epsilon}$, see \eqref{eq:density_comp}. When $a>a_*$, Proposition~\ref{prop:lb_general} with $\sK=\emptyset$ directly yields the lower bound, the exponent $\vartheta$ arising from the capacity estimate of Lemma~\ref{lem:cap_ll}. When $a<a_*$, the leading cost instead comes from a change of measure enforcing the disconnection of a tube from its surroundings, see \eqref{eq:supercrit-first-CM} and Lemma~\ref{lem:supcritCap1}, while Proposition~\ref{prop:lb_general} is applied to the field killed outside the tube and only contributes at lower order. These arguments are given in Section~\ref{sec:lbpf}, and Proposition~\ref{prop:lb_general} is proved in Section~\ref{sec:lbProp}. Throughout, the harmonic averages produced by the decoupling are controlled by Lemma~\ref{lem:critHarm}, see \eqref{eq:lbNearcritDisc} and \eqref{eq:harmAvgTube}. We now state the main result of this section.
\begin{theorem}\label{prop:lower_bound}
For all $a$ with $|a-a_*|\leq \Cr{gNc}$ and $N\geq 1$, we have that with $\xi$ as in \eqref{def:correlation_length},
\begin{equation} \label{eq:prop_lb}
\theta_N(a)\geq c \theta_{\xi}(a_*) \exp\left\{-c'\,\log\big((N/\xi)\vee 2\big)\right\}.
\end{equation}
Furthermore, for all $\eta\in(0,1)$, there exists $\Cl[c]{lb}(\eta)<\infty$ such that if $| \ba - \ba_*|^2\geq \Cr{lb}$, with $\vartheta$ as in \eqref{eq:tau}, \begin{equation} \label{eq:prop_lb_asymp}
\theta_N(a)\geq \theta_{\xi}(a_*) \exp\left\{-({1+\eta}) \vartheta \log(N/\xi)\right\}.
\end{equation}
\end{theorem}
As in Section~\ref{sec:ub}, we begin the proof of Theorem~\ref{prop:lower_bound} by establishing \eqref{eq:prop_lb} in the near-critical regime, that is, when $| \ba - \ba_*|^2\leq \lambda$ for some (fixed) $\lambda \geq 1 $. The remainder of \eqref{eq:prop_lb} follows from \eqref{eq:prop_lb_asymp}, applied with $\eta = 1/2$, and the near-critical result with $\lambda= \Cr{lb}(1/2)$.
Since $| \ba - \ba_*|^2\leq \lambda$ implies $\log(N/\xi)\leq \lambda$, \eqref{eq:prop_lb} in the near-critical regime directly follows from Lemma~\ref{L:nclb} below, which itself relies on an RSW estimate for loops, see \eqref{eq:loop_closed_cont}.
%We start with and estimate which complements Lemma~\ref{lem:near_critical}.
%Recall from \eqref{eq:truncated-event} that the subscript ``tr'' below refers to the relevant connection being achieved by a cluster having finite capacity.
\begin{lemma}[Near-critical lower bound] \label{L:nclb}
For all $N, \lambda \geq 1$ and $a \in \R $ such that $|\ba -\ba_*| \leq \lambda$,
\begin{equation} \label{eq:nclb}
\theta_N(a)\geq c(\lambda)\,\theta_{\xi}(a_*).
%\P_N\big( 0\, \substack{\stackrel{ \geq \ba}{\longleftrightarrow} \\ \textnormal{tr}} \, \partial B_{N} ,\capacity_N(\mathscr{C}^{\geq \ba})<\infty \big ) \geq c(\lambda) \P_N\big( 0 \stackrel{\geqslant \ba_*}{\longleftrightarrow}\partial B_{\xi} ), \quad (\xi=\xi(a, N)).
\end{equation}
\end{lemma}
\begin{proof}
For $N\geq 1$, we first show that for all $a \in \R $ such that $|\ba -\ba_*| \leq \lambda$ and $K\in\{\emptyset, B_{4N}^c\}$,
\begin{equation}\label{eq:lbNearcritCable}
\P_N^K(0\stackrel{\tilde{\varphi }\geq \ba}{\longleftrightarrow}\partial B_N,\capacity_N^K( \tilde{\mathscr{C}}^{\geqslant \ba})<\infty)\geq c(\lambda)\,\theta_{\xi}(a_*).
\end{equation}
Using the same comparison argument as \cite[(4.6)]{rz25a}, we have that $\{ \capacity_N^K(B_N) \leq \text{cap}_N^K( \tilde{\mathscr{C}}^{\geqslant \ba}) < \infty \}\subseteq \{0\stackrel{\tilde\varphi \geq \ba}{\longleftrightarrow}\partial B_N\}$.
In a similar manner as \cite[(4.5)]{rz25a},
it then follows from \eqref{eq:capB}, \eqref{eq:cap-killed} and \cite[Theorem 3.7]{drewitz_cluster_2022} that this event has probability at least \begin{align*}
\frac{1}{2\pi\sqrt{g_N}}\int_{c}^{2c} \frac{e^{-ca^2\log(N)}}{t\sqrt{t-(1/g_N)}}\,dt
&\geq \frac{c^\prime(\lambda)}{\sqrt{g_N}}\int_{c}^{2c} \frac{1}{t\sqrt{t-(1/g_N)}}\,dt \geq \frac{c(\lambda)}{\sqrt{g_N}},
\end{align*}
 where we also used \eqref{eq:g2-asymp-unif} and \eqref{eq:killed_green_smallh} to yield that $g_N(0)\geq g_N^K(0)\geq cg_N(0)$. 
Since $\xi/N=e^{-|a-a_*|^2g_N}\geq c(\lambda)$, the upper bound in \eqref{eq:crit_cost_size} and \eqref{eq:g_N} show that this is at least $c(\lambda)\theta_\xi(a_*)$.

It then suffices to show that for $N\geq 1$, $\theta_N(a)\geq c(\lambda)\,\P_N^K(0\stackrel{\tilde{\varphi }\geq \ba}{\longleftrightarrow}\partial B_N,\capacity_N( \tilde{\mathscr{C}}^{\geqslant \ba})<\infty)$ for $K\in\{\emptyset,B_{4N}^c\}$.
When $a\geq a_*$, we know by Proposition \ref{thm:crit_pt} that
$\theta_N(a)=\P_N(0\stackrel{\geq \ba}{\longleftrightarrow}\partial B_N)\geq \P_N(0\stackrel{\tilde\varphi \geq \ba}{\longleftrightarrow}\partial B_N)$, thus completing the proof for $a\geq a_*$.
We move on to the case $a<a_*$. Since the event $\{\partial B_{4N} \stackrel{\geq \ba}{\mathrlap{\longleftrightarrow}{\mkern15mu{\text{\scriptsize{/}}}}}\,\,\,\,\infty\}$ is measurable with respect to $\sigma(\varphi_x:x\in B_{4N}^c)$, we have that for all $s>0$,
\begin{align}\label{eq:lbNearcritDisc}
\begin{split}
\theta_N(a)
&\geq \P_N(\partial B_{4N} \stackrel{\geq \ba}{\mathrlap{\longleftrightarrow}{\mkern15mu{\text{\scriptsize{/}}}}}\,\,\,\,\infty,0\stackrel{\geq \ba}{\longleftrightarrow}\partial B_N)\\
&\geq \P_N(\partial B_{4N} \stackrel{\geq \ba}{\mathrlap{\longleftrightarrow}{\mkern15mu{\text{\scriptsize{/}}}}}\,\,\,\,\infty,\inf_{x\in B_N}\eta^{B_{4N}^c}_x\geq -s)\times \P_N^{B_{4N}^c}(0\stackrel{\geq \ba+s}{\longleftrightarrow}\partial B_N).
\end{split}
\end{align}
Recall the event $\Cir_{0,4N}$ from \eqref{eq:circuit_event}. By changing the measure to $\tilde{\mathbb{P}}_N$, which is the law of $(\varphi(x) +\ba P_x(H_{B_{8N}}<\infty))_{x\in \Z^2}$ under $\mathbb{P}_N$, we have that by the Cameron-Martin formula
\begin{align}\label{eq:lbNearcritCM}
\begin{split}
\P_N(\partial B_{4N} \stackrel{\geq \ba}{\mathrlap{\longleftrightarrow}{\mkern15mu{\text{\scriptsize{/}}}}}\,\,\,\,\infty)&\geq \P_N\big(\Cir_{0,4N}(\{\varphi<\ba\})\big)\\
&\geq \P_N\big(\Cir_{0,4N}(\{\varphi<0\})\big)\exp\bigg\{\!-\!\frac{\ba^2\capacity_N(B_{8N})+e^{-1}}{2\P_N\big(\Cir_{0,4N}(\{\varphi<0\})\big)}\bigg\}.
\end{split}
\end{align}
In view of \eqref{eq:lupu_loop_coupling} and \eqref{eq:loop_closed_cont}, one has that $ \P_N\big(\Cir_{0,4N}(\{\varphi<0\})\big)\geq\frac12\Q_N\big(\Con_{0,2N}(\mathscr{L})\big)\geq c.$
Combining this with \eqref{eq:lbNearcritCM} and \eqref{eq:capB} yields that $\P_N(\partial B_{4N} \stackrel{\geq \ba}{\mathrlap{\longleftrightarrow}{\mkern15mu{\text{\scriptsize{/}}}}}\,\,\,\,\infty)\geq c(\lambda):=\rho$.
Now by picking $s=s(\rho)$ large enough, it follows from \eqref{eq:critHarm} that $ \P_N(\inf_{x\in B_N}\eta^{B_{4N}^c}_x< -s)\leq e^{-cs^2}\leq \rho/2$, which implies $\P_N(\partial B_{4N} \stackrel{\geq \ba}{\mathrlap{\longleftrightarrow}{\mkern15mu{\text{\scriptsize{/}}}}}\,\,\,\,\infty,\inf_{x\in B_N}\eta^{B_{4N}^c}_x\geq -s)\geq \rho/2$.
We may now conclude in view of \eqref{eq:lbNearcritDisc} and \eqref{eq:lbNearcritCable} since $|\ba+s-\ba_*| \leq c(\lambda)$ and $\P_N^{B_{4N}^c}(0\stackrel{\geq \ba+s}{\longleftrightarrow}\partial B_N)\geq \P_N^{B_{4N}^c}(0\stackrel{\tilde\varphi \geq \ba+s}{\longleftrightarrow}\partial B_N).$
\end{proof}
\begin{comment}
First assume that \eqref{eq:prop_lb_asymp} holds. Applying it with $\eta=1/2$ say, it remains to show \eqref{eq:prop_lb} in the regime $N \geq 1$ and $|\ba -\ba_*| \leq \Cr{lb}(1/2) \equiv \lambda$, which amounts to arguing that
\begin{equation}\label{eq:nclb-true}
\theta(\ba,N)\geq c \,\theta( \ba_* ,\xi)
\end{equation}
in this regime of parameters. To this end we invoke Lemma~\ref{L:nclb} with the above choice of $\lambda$, in the following way. First observe that $\P_N\big( 0 \stackrel{\geqslant \ba_*}{\longleftrightarrow}\partial B_{\xi} )$ as appearing on the right-hand side of \eqref{eq:nclb} is an upper bound for $\theta( \ba_* ,\xi)$. Now observe that $\{\capacity_N(\mathscr{C}^{\geq \ba})<\infty\}$ implies that $0$ is disconnected from $\infty$ in $\{\varphi \geq \ba \}$ (using the condition (Cap) in \cite{drewitz_cluster_2022}), and therefore $
\theta(\ba,N)$ is bounded from below by the probability on the left-hand side of \eqref{eq:nclb}. All in all, \eqref{eq:nclb-true} follows.\\
\end{comment}
\subsection{Proof of Theorem \ref{prop:lower_bound}}\label{sec:lbpf}
To handle the remaining case when $|\ba-\ba_*|\!\geq\!\! \lambda$, we estimate in Proposition \ref{prop:lb_general} the probability for $0$ to be connected to $\partial B_N$ inside a prescribed tube at a supercritical level. The estimate also includes a conditioning argument that will induce the critical cost $\theta_{\xi}(a_*)$ appearing in Theorem \ref{prop:lower_bound}.

For $\epsilon\!\in\!(0,1]$, set
\begin{equation}\label{def:xi_delta}
\xi_\epsilon=\xi_\epsilon(\ba)\defeq N{e^{-\epsilon (\ba-\ba_*)^2}}(\geq 1),
\end{equation}
where the lower bound by $1$ is justified by the condition $|a-a_*|\leq \Cr{gNc}$. And note that $\xi_1=\xi$ and $\xi\leq \xi_\epsilon\leq N$.
Let $\tilde{\mathscr{C}}^{\,\geqslant \ba}_{{\xi_\epsilon}}$ denote the cluster of $0$ in $\{ \tilde{\varphi} \geq \ba\} \cap \tB_{{\xi_\epsilon}}$.
Let $L\geq \xi_\delta$ with $\delta\in(0,1)$ soon to be fixed, consider, with $s_L:=\lfloor 2^{-1/2}L\rfloor$, $i$ ranging in $0,\dots,\lceil (N+L+1)/s_L\rceil$ below, the sets
\begin{align}
\begin{split}
\label{eq:defLell}
\mathcal{L} \defeq \bigcup_{i}{\tB}((is_L,0),4L),\quad \cL^\prime(\delta) \defeq (\tB_{100L}\cup \cL)\setminus \tB_{ \xi_\delta}.
\end{split}
\end{align}
Here with hopefully obvious notation $\tB(x,L)$ is the subset of $\tilde{\Z}^2$ obtained from the Euclidean ball $B(x,L) \subset \Z^2$ centred at $x \in \Z^2$ with radius $L$ by including the cables between any two neighbouring vertices in $B(x,L)$.
Recall the renormalised lattice $\Lambda(L)$ defined above \eqref{eq:conChain} and note that \(((is_L,0))_{0\leq i\leq \lceil (N+L+1)/s_L\rceil}\) is a
star-connected chain in \(\Lambda(L)\), and $B((\lceil (N+L+1)/s_L\rceil s_L,0),L)\subset B_N^c$.
The following connection estimate for $\tilde\varphi$ will be useful in the proof of Theorem~\ref{prop:lower_bound} in both the subcritical and supercritical regimes.

\begin{proposition}\label{prop:lb_general}
For all $\eta,\delta\in(0,1)$, $\epsilon\in(\delta,1)$ and $t>0$, there exists $\Cl[c]{a_prop}(\eta, \epsilon,\delta)<\infty$ such that for all $N\geq 1,a>0$ with $|a-a_*|\leq \Cr{gNc}$ and $|\ba-\ba_*|^2\geq \Cr{a_prop}$, the following holds.
For every $L$ satisfying
\begin{equation}\label{eq:supcrit_L}
\log(|\ba-\ba_*|)\leq N/L\leq |\ba-\ba_*|,
\end{equation}
set $p(\ba,\eta,t):=(1-\eta)(1-e^{-c'(\eta\ba)^2t}),$
we have that for all sets $\sK\subset \Z^2$ (possibly empty) such that $B_{100L}\cup \cL \subseteq \sK^c$,
\begin{align}\label{eq:lb_general}
\begin{split}
&\mathbb{P}_N^{\sK}\bigl(0 \leftrightarrow \partial B_N\text{ in }
(\tB_{100L}\cup \cL)\cap \{ \tilde\varphi \ge (1-\eta)\ba\}\bigr)\\
& \geq \P_N^{\sK}\Big(\capacity_N^{\sK}(\tilde{\mathscr{C}}^{\,\geqslant \ba}_{\xi_\epsilon\!(\ba)})\ge t\Big)\times p(\ba,\eta,t)\times \inf_K
\exp\left\{-\frac{(1+\eta)\ba^2\capacity_N^{K\cup\sK}(\mathcal{L}^{\prime}(\delta))+2e^{-1}}{2p(\ba,\eta,t)}\right\},
\end{split}
\end{align}
where the infimum is over all $K\subseteq \tB_{\xi_\epsilon}$ compact and connected with $\capacity_N^{\sK}(K)\geq t$.
\end{proposition}
%Note that the connection on the left-hand side of \eqref{eq:lb_general} occurs at the slightly lower level $(1-\eta)\ba$, whereas the cluster whose capacity is controlled on the right-hand side is at level $\ba$. This loss is inherent to the argument (see {\color{red}the beginning of the proof of Proposition~\ref{prop:lb_general}} and \eqref{eq:AK_lb} below) and will be compensated by applying Proposition~\ref{prop:lb_general} at a level slightly larger than the target one. Also, the condition $|a-a_*|\leq\Cr{gNc}$ is only used to ensure that $\xi_\delta(\ba)\geq \xi_\epsilon(\ba)\geq 1$, which more generally holds whenever $\epsilon|\ba-\ba_*|^2\leq \log(N)$.

The proof of Proposition \ref{prop:lb_general} is postponed to Section \ref{sec:lbProp}. We now explain the:
\begin{proof}[Proof of \eqref{eq:prop_lb_asymp} when $\ba-\ba_*\geq \lambda$]
In this regime, \eqref{eq:prop_lb_asymp} is a straightforward consequence of Proposition \ref{prop:lb_general}. Indeed, recall the sets defined in \eqref{eq:defLell} with scales $\epsilon,\delta$ and $\xi_\epsilon,\xi_\delta$ from \eqref{def:xi_delta}. Following the same reasoning as for \cite[(4.6)]{rz25a}, the event that $\tilde{\mathscr{C}}^{\,\geqslant \ba} \neq \tilde{\mathscr{C}}^{\,\geqslant \ba}_{{\xi_\epsilon}}$ implies that $\capacity_N(\tilde{\mathscr{C}}^{\,\geqslant \ba})$ must exceed $\capacity_N\big([0,{\xi_\epsilon}]\cap \Z)\times \{0\}\big)$.
Applying \cite[(3.6)]{rz25a} in the case $h_N=1$ to bound this capacity from below by $c\log(N/{\xi_\epsilon})^{-1}$, it follows that
the probability of the event $\{\capacity_N(\tilde{\mathscr{C}}^{\,\geqslant \ba}_{\xi_{\epsilon}})\ge s\log(N/{\xi_\epsilon})^{-1}\}$ is at least $\P_N\left(\capacity_N(\tilde{\mathscr{C}}^{\,\geqslant \ba})\ge s\log(N/{\xi_\epsilon})^{-1}\right)-\P_N\big(\tilde{\mathscr{C}}^{\,\geqslant \ba} \neq \tilde{\mathscr{C}}^{\,\geqslant \ba}_{{\xi_\epsilon}}\big)$, which in turn is bounded from below using \cite[(3.8)]{drewitz_cluster_2022} by
\begin{align}\label{eq:density_comp}
\begin{split}
&\P_N
\left(c\geq \log(N/{\xi_{\epsilon}}) \times \capacity_N(\tilde{\mathscr{C}}^{\,\geqslant \ba})\ge s\right)
\geq\frac{c}{\sqrt{g_N}} \int_{s\log(N/{\xi_\epsilon})^{-1}}^{c\log(N/{\xi_\epsilon})^{-1}} \frac{e^{-(\ba)^{ 2} x/2}}{x^{3/2}}\,dx \\
&\qquad \geq {c g_N^{-1/2}\exp\{-c'(\ba)^{ 2}\log(N/{\xi_\epsilon})^{-1}\}}\sqrt{\log(N/{\xi_\epsilon})}  \geq \exp\{-c'\epsilon^{-1}\}
\cdot\sqrt{\epsilon }
\sqrt{\frac{\log(N/\xi)}{\log(N)}},
\end{split}
\end{align}
where we chose $s=s_0>0$ small enough such that the second last inequality holds; and used  \eqref{eq:g_N}, \eqref{def:xi_delta} and $\log(N/\xi_\epsilon)=\epsilon\log(N/\xi)$ to justify the last bound. 
From now on take (cf.~\eqref{eq:ML_choice_ub}),
\begin{equation}\label{eq:ML_choices}
L=L(a,N)= {N}/{\log({|\ba-\ba_*|})},
\end{equation}
which satisfies \eqref{eq:supcrit_L}. Hence using \eqref{eq:density_comp} and Proposition \ref{prop:lb_general} with $\sK=\emptyset$ yields for all $|\ba-\ba_*|^2\geq \Cr{a_prop}$, $\delta\in(0,1)$ and $\epsilon\in(\delta,1)$ that with $t=s_0\log(N/\xi_\epsilon)^{-1}$,
\begin{equation}\label{eq:subcritApplyProp}
    \theta_N(a)\geq 
    e^{-c\epsilon^{-1}}\theta_{\xi}(a_*)\times p\big(\ba,\eta,t\big)\times \inf_{\substack{K\subseteq \tB_{\xi_{\epsilon}},\\ \capacity_N\!(K)\geq t}}
\!\!\exp\left\{-\frac{(1+\eta)\ba^2\capacity_N^{K}(\mathcal{L}^{\prime}(\delta))+2e^{-1}}{2p\big(a,\eta,t\big)}\right\},
\end{equation}
where we also used the upper bound in \eqref{eq:crit_cost_size}. 

 Henceforth assume $|\ba-\ba_*|\geq c(\eta,\delta)$ with $c(\eta,\delta)$ large enough so that $L\geq \xi_{\delta}$. 
We now bound the quantities on the right-hand side of \eqref{eq:subcritApplyProp} by picking $\delta$ and $\epsilon$.
\begin{lemma} \label{lem:cap_ll}
For all $\epsilon, \eta\in(0,1)$, there exists $\Cl[c]{delta_dist}(\epsilon, \eta)\in(0,\epsilon)$ and $\Cl[c]{eps_a}(\eta)<\infty$ such that for all $a \in(0,1)$, $N\geq 1$ satisfying $|\ba-\ba_*|^2 \geq \Cr{eps_a}$, with $L$ chosen as in \eqref{eq:ML_choices} and with the sup ranging over $K \subset \tB_{\xi_\epsilon}$ compact and connected,
\begin{equation}\label{eq:cap_ll}
\sup_K \capacity_N^K(\mathcal{L}^{\prime}(\delta))\big\vert_{\delta= \Cr{delta_dist}}\leq
(1+\eta) {2\vartheta}.
\end{equation}
\end{lemma}
Lemma \ref{lem:cap_ll} will be proved at the end of the section. Let us first conclude the proof of \eqref{eq:prop_lb_asymp} when $\ba-\ba_*\geq \lambda$.
Using $|\ba-\ba_*|^2\leq 2\ba^2$, which holds for $|\ba-\ba_*|\geq c$ by \eqref{eq:range_crit_pt}, $\epsilon=\epsilon_0(\eta)\in(0,1)$ may be chosen small enough so that
\[
    p(\ba,\eta,s_0\log(N/\xi_\epsilon)^{-1})
    \geq (1-\eta)(1-e^{-c'\frac{(\eta\ba)^2}{\epsilon|\ba-\ba_*|^2}})\geq  (1-\eta)(1-e^{-c'\eta^2/(2\epsilon)}) \geq (1-\eta)^2.
\]
By combining the above display with Lemma \ref{lem:cap_ll} and \eqref{eq:subcritApplyProp} with $\epsilon=\epsilon_0(\eta)$ and $\delta=\Cr{delta_dist}(\epsilon_0,\eta)$, one immediately obtains \eqref{eq:prop_lb_asymp} whenever $(\ba-\ba_*)^2 \geq \Cr{lb}(\eta)$ with $\Cr{lb}(\eta)\geq \Cr{a_prop}$ large enough upon a change of variable in $\eta$. 
%By symmetry, the above remains true when $a<0$ with the additional proviso that the cluster of the origin has finite capacity. One then argues similarly as below \eqref{eq:nclb-true} to conclude.
\end{proof}

We now explain the:
\begin{proof}[Proof of \eqref{eq:prop_lb_asymp} when $\ba-\ba_*< -\lambda$]
Since $a<a_*\leq 0$,
to avoid confusion with signs, we set $\na:=-a>0$ and $\bna:=-\ba>0$. Fix $\eta\in(0,1/4)$, let $\kappa=\kappa(\eta)\in(0,\eta]$ be given by Lemma~\ref{lem:supcritCap1} below, and set ${\eta'}\coloneqq\eta\kappa$.
We consider three scales $1\leq L_1$, $2L_1\leq L_2$, $2L_2\leq L_3$ to be fixed below and tubes $T_i=[-L_i,N+L_i]\times [-L_i,L_i]\cap \Z^2$ for $i\in\{1,2,3\}$.
Let $C^{\na}_{T_1}(\varphi)$ be the event that there exists a path in $T_1\cap\{\varphi\geq -(1-{\eta'}/2)\bna\}$ connecting $0$ to $\partial B_N$. And let $E(\varphi):=\{\inf_{x\in T_1}\varphi_x\geq -(1+{\eta'}/2)\bna\}$.
Since the disconnection event $\{T_2\!\!\! \stackrel{\geq -(1-{\eta'}/2)\bna}
{\mathrlap{\longleftrightarrow}{\mkern15mu{\text{\scriptsize{/}}}}}
\!\!\partial T_3\}$ is measurable with respect to $\sigma(\varphi_x:x\in T_2^c)$,
it follows from \eqref{eq:field_decomp} that,
\begin{align}\label{eq:supcritLbMarkov}
\begin{split}
\theta_N\big((1\!-\!\tfrac {\eta'} 2)a\big)
\!&\geq \!
\P_N
\Big(\{T_2\!\! \!\stackrel{\geq -(1\!-\!\frac {\eta'} 2)\bna}
{\mathrlap{\longleftrightarrow}{\mkern15mu{\text{\scriptsize{/}}}}}\!\!
\partial T_3\}\!\cap\! C^{\alpha}_{T_1}(\varphi)\Big)\\
&\geq\! \P_N
\Big(\{T_2\!\! \!\stackrel{\geq -(1\!-\!\frac {\eta'} 2)\bna}
{\mathrlap{\longleftrightarrow}{\mkern15mu{\text{\scriptsize{/}}}}}\!\!
\partial T_3\}\!\cap\!E(\eta^{T_2^c})
\Big)
\!\times\!
\P_N^{ T_2^c}\!\big(C^{\alpha}_{T_1}\big(\varphi\!-\!(1\!+\!\tfrac {\eta'} 2)\bna\big)\big).
\end{split}
\end{align}
As in \eqref{eq:proof2point3}, by changing the measure to $\tilde{\mathbb{P}}_N\equiv \tilde{\mathbb{P}}_{N; \, \na, L_2,L_3 }$, which is the law of
$(\varphi(x) -\bna P_x(H_{T_3\setminus T_2}<\infty))_{x\in \Z^2}$ under $\mathbb{P}_N$, we have by the Cameron-Martin formula that the first factor in the second line of \eqref{eq:supcritLbMarkov} is at least,
\begin{align}\label{eq:supercrit-first-CM}
\P_N
\big(\{T_2\!\! \stackrel{\geq {\eta'} \bna/2}
{\mathrlap{\longleftrightarrow}{\mkern15mu{\text{\scriptsize{/}}}}}
\,\,\partial T_3\}\!\cap\!E(\eta^{T_2^c}-\bna)\big)
\times
\exp\!\left\{\!\!
-\!\frac{\bna^2\capacity_N({T_3\setminus T_2})/2+1/e}
{\P_N
\big(\{T_2\!\! \stackrel{\geq {\eta'} \bna/2}
{\mathrlap{\longleftrightarrow}{\mkern15mu{\text{\scriptsize{/}}}}}
\,\,\partial T_3\}\!\cap\!E(\eta^{T_2^c}-\bna)\big)}
\right\}.
\end{align}
We now claim that with $u=({\eta'}\bna/2)^2/2$
\begin{align}\label{eq:disc_lb}
\begin{split}
\P_N
\big(\{T_2 \stackrel{\geq {\eta'}\bna/2}{\longleftrightarrow}
\partial T_3\}\big)
&\!\leq\! 1\!-\!
\bar{\P}_N\!
\left(
\begin{gathered}
\text{$\exists$ a star-connected path $\gamma\subseteq \I^u$ s.t.}\\
\text{$\gamma$ forms a closed circuit in $T_3\setminus T_2$}
\end{gathered}
\right)
\leq e^{-\!c{\eta'}^2\bna^2\capacity_N( T_2) e^{-\frac{c'N}{L_3}}}.
\end{split}
\end{align}
Indeed, the first inequality follows from \eqref{eqcouplingintergff}. For the second inequality, let $L_3'=L_3/200$ and note that we can find a path $\tilde \gamma=(x_1,\dots,x_n)\subset T_3\setminus T_2$ with $n\leq c N/L_3$ satisfying,
\[
\begin{array}{l}
|x_i - x_{i+1 \text{ mod }n}| < 2L_3', \, B(x_i,10L_3')\subset T_3 \setminus T_2, \, 1\leq i \leq n,\,
\text{ and } \\
\partial([-\frac34 L_3,N+\frac34L_3]\times [-\frac34L_3,\frac34L_3]\cap \Z^2)\subset \bigcup_{i=1}^n B(x_i,L_3').
\end{array}
\]
It then suffices to repeat the argument for \eqref{eq:Iu_closed_cont} with $\tilde \gamma$ instead of the path used to prove Lemma~\ref{lem:Iu_closed_cont}.
The argument leading to
\eqref{eq:lb_X} then yields $P_x(X\text{ forms a closed circuit in $T_3\setminus T_2$})\geq e^{-cN/L_3}$ for all $x\in T_2$.
As in \eqref{eq:intensity_lb}, this implies that intensity measure of trajectories that create a closed circuit in $T_3\setminus T_2$ is at least
$c e^{-c'N/L_3}\capacity_N( T_2)$, which gives
\eqref{eq:disc_lb}.
Moreover, by \eqref{eq:critHarm} and the fact that $\eta^{T_2^c}_\cdot$ is centred,
\begin{equation}\label{eq:harmAvgTube}
\P_N\Big(\!\inf_{x\in T_1}\! \eta^{T_2^c}_x< \!-{\eta'}\bna/2\Big)\leq \frac{cN^2}{d(T_1,T_2^c)^2}\times  \exp\Bigg\{\!-\!\frac{c\Big({\eta'}\bna-c'\sqrt{\log(\tfrac{N}{d(T_1,T_2^{c})}\vee 2)}\Big)_+^2}{\log\Big(\tfrac{N}{d(T_1,T_2^{c})}\vee 2\Big)}\Bigg\}.
\end{equation}
We now pick
\begin{equation}\label{eq:L123choices}
L_1={100N}/{\log|\bna-\bna_*|},\quad L_2={\kappa N}, \quad
L_3={2\kappa N} \quad \text{and}\quad L=L_1/100,
\end{equation}
and we set $\beta\coloneqq{\eta'}\bna/(1-\eta)$, the level at which Proposition~\ref{prop:lb_general} will be applied, so that the connection it provides holds at level $(1-\eta)\beta={\eta'}\bna$.
Recall the sets defined in \eqref{eq:defLell} with scales $\epsilon,\delta$ and $\xi_\epsilon(\beta),\xi_\delta(\beta)$ from \eqref{def:xi_delta}.
Henceforth assume $|\ba-\ba_*|\geq c(\eta,\epsilon)$ with $c(\eta,\epsilon)$ large enough so that $L_2/4\geq L_1\geq \xi_{\delta}(\beta)$.
Note that $L$, as chosen in \eqref{eq:L123choices}, satisfies $B_{100L}\cup \cL\subset T_1 \subset T_2$, as well as \eqref{eq:supcrit_L} with $\beta$ in place of $\ba$. Indeed, since $\bna=|\ba-\ba_*|+|\ba_*|$, we have $\beta\geq {\eta'}|\ba-\ba_*|$, while $\beta-\ba_*\leq 2{\eta'}\bna+|\ba_*|\leq |\ba-\ba_*|$ for $|\ba-\ba_*|\geq c(\eta)$ by \eqref{eq:range_crit_pt}; hence, for $|\ba-\ba_*|\geq c(\eta)$,
\[\log(\beta-\ba_*)\leq  N/L=\log|\ba-\ba_*|\leq \beta-\ba_*.\]
Hence using Proposition \ref{prop:lb_general} at level $\beta$ (i.e.~with $\beta$ in place of $\ba$) with $\sK=T_2^c$ yields for all $N\geq 1,a<0$ such that $|\beta-\ba_*|^2\geq \Cr{a_prop}$, $\delta\in(0,1)$ and $\epsilon\in(\delta,1)$ and $t>0$,
\begin{align}\label{eq:supcritApplyProp}
\begin{split}
    &\P_N^{ T_2^c}\!\big(C^{\alpha}_{T_1}\big(\varphi\!-\!(1\!+\!\tfrac {\eta'} 2)\bna\big)\big)
    \geq \mathbb{P}_N^{T_2^c}\bigl(0 \leftrightarrow \partial B_N\text{ in }
    (\tB_{100L}\cup \cL)\cap \{ \tilde\varphi \ge {\eta'}\bna\}\bigr)\\
    &\geq
    \P_N^{T_2^c}\Big(\capacity_N^{T_2^c}(\tilde{\mathscr{C}}^{\,\geqslant \beta}_{\xi_\epsilon\!(\beta)})\ge t\Big)\!\times\! p(\beta,\eta,t)\!\times\! \inf_K
\exp\left\{\!-\frac{(1\!+\!\eta)\beta^2\capacity_N^{K\cup T_2^c}\!(\mathcal{L}^{\prime}(\delta))\!+\!2e^{-1}}{2p(\beta,\eta,t)}\right\},
\end{split}
\end{align}
where the infimum is over all $K\subseteq \tB_{\xi_{\epsilon}(\beta)}$ compact and connected with $\capacity_N^{T_2^c}(K)\geq t$.

The next lemma, proved at the end of this subsection, ensures that the exponential factor in \eqref{eq:supcritApplyProp} is subleading, see \eqref{eq:lbCapSubLead}, and provides the sharp capacity bound \eqref{eq:lbCapLead} behind the leading cost in \eqref{eq:supercrit-first-CM}.
\begin{lemma}\label{lem:supcritCap1}
For all $\eta\in(0,1/4)$ there exists $\kappa=\kappa(\eta)\in(0,\eta]$ such that the following hold with $L,L_1,L_2,L_3$ as in \eqref{eq:L123choices}.
For all $\epsilon\in(0,1)$, $\cL'(\delta)$ as in \eqref{eq:defLell} and $\xi_\delta=\xi_\delta(\beta),\xi_\epsilon=\xi_\epsilon(\beta)$ therein,
there exist $\delta_0(\epsilon)\in(0,1)$ and $C(\eta,\epsilon)<\infty$ such that for all $a \in(-1,0)$, $N\geq 1$ satisfying $|\ba-\ba_*|^2 \geq C$, we have that for all $K \subset \tB_{\xi_\epsilon}$,
\begin{equation}\label{eq:lbCapSubLead}
\capacity_N^{K\cup T_2^c}(\cL'(\delta))\vert_{\delta =\delta_0 }
\leq c\kappa^{-1}.
\end{equation}
Moreover, there exists $C'(\eta)<\infty$ such that for all $a \in(-1,0)$, $N\geq 1$ satisfying $|\ba-\ba_*|^2 \geq C'$,
\begin{equation}\label{eq:lbCapLead}
  \capacity_N({T_3\setminus T_2})
\leq
(1+\eta)2\vartheta.
\end{equation}
\end{lemma}
We now claim that there exists $s_0>0$ such that with $\xi_\epsilon=\xi_\epsilon(\beta)$,
\begin{equation}\label{eq:supCritClaim}
\P_N^{T_2^c}
\Big(\capacity_N^{T_2^c}(\tilde{\mathscr{C}}^{\,\geqslant \beta}_{\xi_\epsilon})\!\ge\! s_0\log( N/{\xi_{\epsilon}})^{-1}\Big)\!\geq\!   \exp\{-c\epsilon^{-1}\}\times {\eta'} \, \theta_{\xi(\ba)}(a_*).
\end{equation}
Indeed, by \cite[(1.61)]{sznitman_topics_2012} and since $g_N^{T_2^c}\leq g_N$, one has $\capacity^{T_2^c}_N\big(([0,\xi_\epsilon]\cap\Z)\times \{0\}\big)\geq \capacity_N\big(([0,\xi_\epsilon]\cap\Z)\times \{0\}\big)\geq c\log(N/\xi_{\epsilon})^{-1}$. The argument leading to \eqref{eq:density_comp}, with $(\P_N^{T_2^c},\beta)$ in place of $(\P_N,\ba)$, thus yields that the left-hand side of \eqref{eq:supCritClaim} is at least $c g_N^{-1/2}\exp\{-c'\epsilon^{-1}\}\sqrt{\log( N/{\xi_\epsilon})}$, and \eqref{eq:supCritClaim} follows from \eqref{eq:crit_cost_size} since $\log(N/\xi_\epsilon)=\epsilon|\beta-\ba_*|^2\geq \epsilon{\eta'}^2|\ba-\ba_*|^2$.
With $t=s_0\log(N/\xi_{\epsilon}(\beta))^{-1}$, and since $\beta-\ba_*\leq 2\beta$ for $|\ba-\ba_*|\geq c(\eta)$ by \eqref{eq:range_crit_pt}, we may pick $\epsilon=\epsilon_0(\eta)\in(0,1)$ small enough such that
\begin{equation*}%\label{eq:supConlb}
    (1-\eta)^{-1}p(\beta,\eta,t)\geq  1-\exp\left\{-c'\eta^2\beta^2 s_0\log(N/\xi_{\epsilon})^{-1}\right\}=1-\exp\left\{-\frac{c's_0\eta^2\beta^2}{\epsilon|\beta-\ba_*|^2}\right\}\geq 1-\eta.
\end{equation*}
Combining \eqref{eq:supcritApplyProp}, \eqref{eq:supCritClaim} and \eqref{eq:lbCapSubLead} with $\epsilon=\epsilon_0(\eta), \delta=\delta_0(\epsilon_0)$ and $t=s_0\log( N/\xi_{\epsilon_0}(\beta))^{-1}$, and noting that $(1+\eta)\beta^2 c\kappa^{-1}\leq c'{\eta'}^2\bna^2\kappa^{-1}= c'\eta^2\kappa\bna^2\leq c'\eta\bna^2$, gives for $|\ba-\ba_*|^2\geq c(\eta)$ with $c(\eta)$ large enough so that $|\beta-\ba_*|^2\geq \Cr{a_prop}$,
\begin{align}\label{eq:supLbFinalCon}
\begin{split}
\P_N^{ T_2^c}\!\big(C^{\alpha}_{T_1}\big(\varphi\!-\!(1\!+\!{\eta'}/2)\bna\big)\big)
&\geq   c(\eta) \theta_{\xi(\ba)}(a_*)  \exp\{-c\eta\bna^2\}\geq \theta_{\xi(\ba)}(a_*)  \exp\{-2c\eta\bna^2\} ,
\end{split}
\end{align}
where the last inequality holds upon picking $|\ba-\ba_*|^2\geq c(\eta)$ with $c(\eta)$ large enough, which ensures that $\bna$ is large.

We next take care of the event $\{T_2\!\! \!\stackrel{\geq -(1\!-\!\frac {\eta'} 2)\bna}
{\mathrlap{\longleftrightarrow}{\mkern15mu{\text{\scriptsize{/}}}}}\!\!
\partial T_3\}\!\cap\!E(\eta^{T_2^c})$ in \eqref{eq:supcritLbMarkov}.
First observe that, by monotonicity of capacity and \cite[(3.6)]{rz25a} with $h_N=1$ therein, we have that
$
\capacity_N( T_2)
\geq\capacity_N([0,N]\times \{L_2\})\geq c.
$
Combining this with \eqref{eq:L123choices}, \eqref{eq:supercrit-first-CM}, \eqref{eq:disc_lb}, \eqref{eq:harmAvgTube}, \eqref{eq:lbCapLead} and the fact $d(T_1,T_2^c)\geq L_2/2=\kappa N/2$, we have that for $|\ba-\ba_*|\geq \Cr{lb}(\eta)$ with $\Cr{lb}(\eta)$ large enough,
\begin{align}\label{eq:finalLbDisc}
\begin{split}
&\P_N
\big(\{T_2 \stackrel{\geq {\eta'}\bna/2}
{\mathrlap{\longleftrightarrow}{\mkern15mu{\text{\scriptsize{/}}}}}
\,\,\partial T_3\}\cap E(\eta^{T_2^c}-\bna)\big)\geq 1-e^{-c(\eta)\bna^2}-e^{-c'(\eta)\bna^2}\geq 1-\eta
\quad\text{and} \\
& \P_N
(\{T_2\!\! \stackrel{\geq -(1-{\eta'}/2)\bna}
{\mathrlap{\longleftrightarrow}{\mkern15mu{\text{\scriptsize{/}}}}}
\partial T_3\}\cap E(\eta^{T_2^c}))
\geq c\exp\{-(1+\eta)^3\vartheta \bna^2\},
\end{split}
\end{align}
where we also used that $(1+\eta)/(1-\eta)\leq(1+\eta)^3$ for $\eta\in(0,1/4)$.
Finally, using $a<a_* \leq 0$, $\bna
=
|\ba-\ba_*|+\bna_*,$
and \eqref{eq:range_crit_pt}, we have $\bna^2\leq(1+\eta)|\ba-\ba_*|^2$ for $|\ba-\ba_*|\geq c(\eta)$. We may now combine this with
\eqref{eq:supcritLbMarkov},
\eqref{eq:finalLbDisc} and \eqref{eq:supLbFinalCon} to obtain
\begin{align*}
\theta_N\big((1-\tfrac{\eta'}2)a\big)
&\geq
c\,\theta_{\xi}(a_*)
\exp\left\{
-(1+\eta)^4\vartheta|\ba-\ba_*|^2
-{c'\eta|\ba-\ba_*|^2}
\right\}.
\end{align*}
We apply this bound with $a/(1-\frac{\eta'}2)$ in place of $a$. By \eqref{eq:range_crit_pt}, this changes $|\ba-\ba_*|^2$ by a factor $1+O(\eta)$ when $|\ba-\ba_*|\geq c(\eta)$, and hence, by \eqref{def:correlation_length} and \eqref{eq:crit_cost_size}, changes $\theta_\xi(a_*)$ by at most a constant factor. Upon a change of variable in $\eta$, the resulting right-hand side is seen to exceed $
\theta_{\xi}(a_*)\exp\{-(1+\eta)\vartheta|\ba-\ba_*|^2\}=\theta_{\xi}(a_*)\exp\{-(1+\eta)\vartheta\log(N/\xi)\}$, which is \eqref{eq:prop_lb_asymp}.
\end{proof}

\begin{proof}[Proof of Lemma \ref{lem:cap_ll}]
For all $x\in \cL^{\prime}(\delta)$, we have that by \eqref{eq:last_exit}, \eqref{eq:defLell}, \eqref{eq:capB}, \eqref{eq:g2-asymp-unif} and \eqref{eq:bessel_ub},
\begin{equation}\label{eq:hitting_K}
P_x(H_K<\infty)\leq c\frac{\log(N/\xi_{\delta})}{\log(N/\xi_\epsilon)}\leq c\delta \epsilon^{-1}\leq \eta,
\end{equation}
where the last inequality holds by picking $\delta=\Cr{delta_dist}(\epsilon, \eta)$ with $\Cr{delta_dist}$ small enough. We first claim that for all $A,K\subseteq \tilde{\Z}^2$ and $t\in(0,1)$ such that $\sup_{x\in A}P_x(H_K<\infty)\leq t$,
\begin{equation}\label{eq:killedCapClaim}
\capacity^K_N(A)\leq (1-t)^{-1}\capacity_N(A).
\end{equation}
Indeed, by \eqref{eq:last_exit} and \cite[(1.49)]{sznitman_topics_2012}, for all $x\in A$, $\sum_{y\in A} g_N^K(x,y)e_{A,N}(y)=1-E_x[P_{X_{H_K}}(H_A<\infty),H_K<\infty]$ is at least $1-P_x(H_K<\infty)\geq 1-t$. Hence $\capacity_N^K(A)
\!=\! \sum_{x\in A} e_{A,N}^K(x)\leq (1-t)^{-1}\sum_{y\in A}e_{A,N}(y)\sum_{x\in A} e_{A,N}^K(x) g_N^K(x,y)\!=\!(1-t)^{-1}\capacity_N(A)$, as claimed. Applying \eqref{eq:killedCapClaim} with $A=\cL'(\delta)$ and $t=\eta$, it follows that $\capacity^K_N(\cL'(\delta))\leq (1-\eta)^{-1}\capacity_N(\cL'(\delta))$.
Using
 \eqref{eq:defLell}, the subadditivity of capacity, \eqref{eq:capB} and \eqref{eq:ML_choices}, we further have that
\begin{align}
\begin{split}
\capacity_N(\cL'(\Cr{delta_dist})) &\leq \capacity_N(\cL)+\capacity_N(B_{100L})
\leq \capacity_N(\cL)+c\log(N/(L))^{-1}\\
&\leq \capacity_N(\cL)+c(\log\log|\ba-\ba_*|)^{-1}.
\end{split}
\end{align}
To bound $\capacity_N(\cL)$, using the upper bound in \ref{prop:tube} with $n=P=\lceil 2N/L\rceil$ therein (note that the condition for \eqref{eq:tube_ub_macro} holds with $\delta$ chosen to be $1/2$ therein), and taking ${|\ba-\ba_*|}^2\geq \Cr{eps_a}(\eta)$ with $\Cr{eps_a}(\eta)$ large enough we get that $\capacity_N(\cL)\leq (1+\eta){2\vartheta}$. Therefore \eqref{eq:cap_ll} follows upon a change of variable in $\eta$.
\end{proof}

\begin{proof}[Proof of Lemma \ref{lem:supcritCap1}]
We start by choosing $\kappa$ and proving \eqref{eq:lbCapLead}. Let $P\coloneqq\lceil\Cr{tube_N}(\eta,1/4)\rceil$, with $\Cr{tube_N}$ from Proposition~\ref{prop:tube}, and set $\kappa\coloneqq\eta\wedge(c/P)$ for a small absolute constant $c>0$ (in particular $T_3\setminus T_2\subseteq B_{2N}$ as $\kappa<1/4$). Cut $T_3\setminus T_2$ into the $P$ consecutive slabs $S_i=(T_3\setminus T_2)\cap\{x: \frac{(i-1)N}{P}-L_3\1_{\{i=1\}}\leq x_1\leq\frac{iN}{P}+L_3\1_{\{i=P\}}\}$, $1\leq i\leq P$. Since $L_3=2\kappa N\leq 2cN/P$, choosing $c$ small enough gives $|x-y|\leq(|i-j|+4)\frac NP$ for all $x\in S_i$, $y\in S_j$, so Proposition~\ref{prop:tube}(ii) with $n=P$, $\varepsilon=\eta$ and $\delta=1/4$ yields \eqref{eq:lbCapLead} for all $N\geq \Cr{tube_N}(\eta,1/4)\vee P$, which holds when $|\ba-\ba_*|^2\geq C'(\eta)$ with $C'$ large enough, since $|\ba-\ba_*|^2\leq c\log N$.

It remains to show \eqref{eq:lbCapSubLead}.
As in \eqref{eq:hitting_K}, $P_x^{T_2^c}(H_K<\infty)\leq P_x(H_K<\infty)\leq c\delta\epsilon^{-1}\leq\frac12$ for all $x\in\cL'(\delta)$, upon picking $\delta=c(\epsilon)$ small enough. Hence by a similar argument as for \eqref{eq:killedCapClaim}, with $P_x$, $g_N$ and $e_{A,N}$ replaced by their counterparts for the process killed on $T_2^c$, one has $\capacity_N^{K\cup T_2^c}( \cL'(\delta))
\leq
2\capacity_N^{T_2^c}( \cL'(\delta)).$
We partition $ \cL'(\delta)$ into disjoint sets
$S_1,\ldots,S_m$ where 
\[S_i:=\big([-L_1+\tfrac{(i-1)L_2}{4}, -L_1+\tfrac{iL_2}{4})\times [-L_1,L_1]\big)\cap \cL'(\delta), \quad 1 \leq i \leq  m\defeq\lceil \tfrac{4(N+2L_1)}{L_2}\rceil.\]
Hence for each $1\leq i\leq m$, $x_i\coloneqq ((-L_1+\tfrac{2i-1}{8}L_2,0)$ is the centre of each rectangle in the definition of $S_i$ and $S_i\subset B(x_i,L_2/2)\subset T_2$. By \eqref{eq:gNK} and \eqref{eq:killed_green_smallh} we have that uniformly in $x,y\in S_i$ and $1 \leq i \leq m$, $g_N^{T_2^c}(x,y) \geq g_N^{B(x_i,L_2/2)^c}(x,y) \geq c.$ For every probability measure $\mu$ supported on $\cL'(\delta)$, the Cauchy-Schwarz inequality gives $
1=(\sum_{i=1}^m \mu(S_i))^2\leq m\sum_{i=1}^m \mu(S_i)^2$ and hence
\[
\sum_{x,y\in \cL'(\delta)}g_N^{T_2^c}(x,y)\mu(x)\mu(y)\geq \sum_{i=1}^m\sum_{x,y\in S_i}g_N^{T_2^c}(x,y)\mu(x)\mu(y)
\geq
c\sum_{i=1}^m\mu(S_i)^2
\geq
\frac{c}{m}
\geq
c\frac{L_2}{N}.
\]
Combining this with
\cite[(1.61)]{sznitman_topics_2012} yields $\capacity_N^{T_2^c}(\cL'(\delta))\leq c\frac{N}{L_2}\leq c\kappa^{-1}$.
\end{proof}

\subsection{Proof of Proposition \ref{prop:lb_general}}\label{sec:lbProp}
\begin{proof}
Assume $|\ba-\ba_*|\geq c(\delta)$ with $c(\delta)$ large enough so that $L\geq \xi_{\delta}\geq \xi_{\epsilon}\geq 1$. 
For any $K\subseteq \tB_{\xi_{\epsilon}}$ compact and connected, let $v(x)=P_x(H_{K}<{H_{\sK}})$ and set, for $\eta\in(0,1)$, $ A_K$ the event that the excursion set $\{x\in \tilde \Z^2: \tilde \varphi_x \geq \ba(1-\eta- v(\cdot)) \}$ has a connected path in $\tB_{100L} \!\cup\! \cL$ intersecting both $K$ and $\partial B_N$.
Similarly as in \cite[(6.10)]{drewitz_critical_2023}, the strong Markov property yields that, for all $t,a>0$ and $\epsilon\in(0,1)$, abbreviating $\K=\tilde{\mathscr{C}}^{\,\geqslant \ba}_{\xi_\epsilon\!(\ba)}$, the probability $\P_N^{\sK}\big( 0\!\leftrightarrow\!\partial B_N \text{ in }
(\tB_{100L} \!\cup\! \cL)\! \cap \! \{ \tilde\varphi \!\ge\! (1\!-\!\eta)\ba\}\big)$ from \eqref{eq:lb_general} to be bounded from below is at least
\begin{align*}%\label{eq:proof2point1}
\begin{split}
%&\P_N^{\sK}\big( 0\!\leftrightarrow\!\partial B_N \text{ in }
%(\tB_{100L} \!\cup\! \cL)\! \cap \! \{ \tilde\varphi \!\ge\! (1\!-\!\eta)\ba\}\big)\\
&  \E_N^{\sK}
\Big[
\1\big\{\capacity_N^{\sK}(\K)\!\ge\! t\big\}
\!\times\! \P_N^{K\cup {\sK}}(A_K)\big\vert_{ K=\K}\!
\Big]
\!\geq\!
\P_N^{\sK}
\Big(
\capacity_N^{\sK}(\K)\!\ge\! t\Big)
\!\times\! \inf_{K}\P_N^{K\cup {\sK}}(A_K),
\end{split}
\end{align*}
where the infimum is over all $K\subseteq\tB_{\xi_\epsilon}$ compact and connected with $\capacity_N^{\sK}(K)\ge t$.
We now proceed to lower bound $\P^{K\cup\sK}_N(A_K)$.
Let $\tilde{\mathbb{P}}^{K\cup\sK}_N\equiv \tilde{\mathbb{P}}^{K\cup\sK}_{N; \, a, L,\delta }$ be the law of
$(\varphi(x) +\ba \tilde{v}(x))_{x\in \tilde{\Z}^2}$ under $\mathbb{P}^{K\cup\sK}_N$, where $\tilde{v}(x)=P_x(H_{\mathcal{L}^{\prime}(\delta)} < H_{K\cup\sK})$. Note that $\tilde{v}(\cdot)$ depends implicitly on the choice of $L, \delta$, along with $K, N, \epsilon$. As in \cite[(6.17)]{drewitz_critical_2023}, the Cameron-Martin formula yields
\begin{equation}
\label{eq:proof2point3}
\P^{K\cup\sK}_N( A_{K})
\ge
\tilde{\mathbb{P}}^{K\cup\sK}_N( A_{K})
\exp\Biggl\{-\frac{\ba^{2}\capacity_N^{K\cup\sK}(\mathcal{L}^{\prime}(\delta)) + 2/e}{2\,\tilde{\mathbb{P}}^{K\cup\sK}_N( A_{K})}\Biggr\}.
\end{equation}
%For $\epsilon$ to be fixed later, let $K\subset B_{\xi_\epsilon}$ be a compact and connected set satisfying $\capacity_N(K)\ge s_0\log(N/{{\xi_\epsilon}})^{-1}$. Recall $\Cr{delta_dist}$ from Lemma \ref{lem:cap_ll} and the functions $v(x)=P_x(H_K< \infty) $ and $\tilde{v}(x)= P_x(H_{\mathcal{L}^{\prime}(\delta)} < H_K)$ introduced around \eqref{eq:proof2point1} and \eqref{eq:proof2point3}, respectively.
By definition of $\mathcal{L}'(\delta)$ in \eqref{eq:defLell} and the fact that $L\geq \xi_\delta\geq \xi_{\epsilon}$, we know that $\tB_{100L}\cup \cL\subset \sK^c$.
Using \cite[(2.51)]{lawler_random_2010}, \eqref{eq:supcrit_L} and the tail of $\tau_N$, an exponential variable with rate $N^{-2}$, we obtain that for all $\delta\in(0,1)$ and all $x\in \tB_{100L}\cup \cL$,
\begin{align*}
v(x)+ \tilde{v}(x)
&\geq P_x(H_{K\cup \mathcal{L}^{\prime}(\delta)}<H_{\sK}) \ge P_x(H_{B(x,2\xi_{\delta})^{c}}<L^2)P(\tau_N>L^2) \\&\geq (1-ce^{-c^\prime L^2/\xi_{\delta}^2}) e^{-L^2/N^2}
\stackrel{\eqref{eq:supcrit_L}}{\geq} (1-ce^{-c^\prime \frac{e^{2\delta|\ba-\ba_*|^2}}{|\ba-\ba_*|^2}}) e^{-\frac{1}{\log(|\ba-\ba_*|)^2}} \geq 1-\eta/2 ,
\end{align*}
where the last two inequalities hold upon taking $\ba^2\geq c(\eta, \delta)$ with $c(\eta,\delta)$ large enough, which on account of Proposition~\ref{thm:crit_pt} also implies that we may assume $a \geq |a_*|$ without incurring loss of generality. The above display then implies that
\begin{equation}\label{eq:AK_lb}
\tilde{\P}_N^{K\cup \sK}(A_{K})\geq 
\P_N^{K\cup \sK}\Big(K \leftrightarrow \partial B_N\text{ in }
(\tB_{100L}\cup \cL)\cap \{ \varphi \ge -\eta\ba/2\}\Big).
\end{equation}
The estimate of the above supercritical connection probability is collected in the following Lemma.
\begin{lemma} \label{lem:li_connect}
For all $\eta\in(0,1),t>0, \epsilon\in(0,1)$, there exists $C(\eta)<\infty$ such that for all $ K\subset \tB_{\xi_{\epsilon}}$ compact and connected with $\capacity_N^{\sK}(K)\geq t$, and all $N\geq 1,a>0$ with $|a-a_*|\leq \Cr{gNc}$ and $|\ba-\ba_*|^2\geq C$, the following holds.
With $L$ as in \eqref{eq:supcrit_L},
we have
\begin{equation} \label{eq:li_connect}
\mathbb{P}^{K\cup\sK}_N\bigl(K \leftrightarrow \partial B_N\text{ in }
(\tB_{100L}\cup \cL)\cap \{\tilde \varphi \ge -\eta\ba/2\}\bigr)\geq (1-\eta)(1-e^{-c'(\eta\ba)^2t}).
\end{equation}
\end{lemma}
We postpone the proof of Lemma \ref{lem:li_connect} and first finish the proof of Proposition \ref{prop:lb_general}. Combining the first display of this proof, \eqref{eq:proof2point3}, \eqref{eq:AK_lb} and \eqref{eq:li_connect}, we have that with $p(\ba,\eta,t)$ as below \eqref{eq:lb_general}, the l.h.s. of \eqref{eq:lb_general} (i.e.~the probability of a connection at level $(1-\eta)\ba$) is at least
\[\P_N^{\sK}
\Big(
\capacity_N^{\sK}(\tilde{\mathscr{C}}^{\,\geqslant \ba}_{\xi_\epsilon\!(\ba)})\ge t\Big)\times p(\ba,\eta,t)\,\inf_K\exp\left\{-\frac{\ba^2\capacity_N^{K\cup\sK}(\mathcal{L}^{\prime}(\delta))+2e^{-1}}{2p(\ba,\eta,t)}\right\},\]
with the infimum as in \eqref{eq:lb_general}, which implies \eqref{eq:lb_general}.
\end{proof}
\medskip
\begin{comment}
\begin{lemma} \label{lem:li_connect}
For all $\eta\in(0,1)$, there exists $\Cl[c]{eps_eta}(\eta)\in(0,1),C(\eta)<\infty$ such that for all $K\subset \tB_{\xi_{\Cr{eps_eta}}}$ compact and connected with $\capacity_N(K)\geq s_0\log(N/\xi_{\Cr{eps_eta}})^{-1}$, and all $a\in(0,1)$ and $N\geq1$ such that $|\ba-\ba_*|^2\geq C$, the following holds.
With $L$ chosen as in \eqref{eq:ML_choices}, abbreviating $\tilde{\mathbb{P}}^K_{N} = \tilde{\mathbb{P}}^K_{N; \, a, L,\Cr{delta_dist} (\Cr{eps_eta})}$, we have that
\begin{equation} \label{eq:li_connect}
\tilde{\mathbb{P}}^K_{N } ( A_{K})\geq 1-\eta.
\end{equation}
\end{lemma}
\end{comment}
To complete the proof of Lemma \ref{lem:li_connect}, we first isolate the following result. %Exactly one of the two bounds in \eqref{eq:killed_survive_il} below will be needed in the proof of one of the two previous lemmas.
Consider the interlacement process on $\tilde \Z^2$ with killing on $K\cup\sK$ (we continue to denote its law by $\bar \P_N^{K\cup\sK}$). 
Similar to around~\eqref{def:interlacement},
under $\bar \P_N^{K\cup\sK}$, the trajectories contributing to $\I^u$ have their forward and backward parts exit $\tilde \Z^2\setminus ({K\cup\sK})$ either through the half-open cables attached to each vertex or upon hitting ${K\cup\sK}$. We call a trajectory $w$ \emph{killed-surviving} if its backward part hits $K$ (killed) while its forward part exits through one of the half-open cables or hits $\sK$ (survives). In the sequel, we denote by $W^*_{-,+}$ the set of (equivalence classes of) such killed-surviving trajectories.
The following result is similar to \cite[Lemma 7.2]{drewitz_critical_2023}, however, the statement is no longer an equality since we are working on a graph with non-vanishing killing measure.
\begin{lemma} \label{lem:killed_survive_il}
Let $c>0$ be arbitrary. For all compact and connected sets $ K\subset \sK^c$ and $K^\prime \subset \sK^c\setminus K $, where $K^\prime$ is such that, for all $x\in K$, every continuous path on $\tilde \Z^2$ starting in $x$ and exiting $B(x,cN)\cap \sK^c$ intersects $K^\prime$.
Then $$c^\prime \,\capacity_N^{\sK}(K)
\leq \nu^{K\cup \sK}_N(W^*_{-,+}(K^\prime))
\leq \capacity_N^{\sK}(K),$$ where $W^*_{-,+}(K^\prime)$ is the set of trajectories modulo time-shift in $W^*_{-,+}$ that hit $K^\prime$.
\end{lemma}
\begin{proof}
Since $\nu_N^{K\cup \sK}$ is invariant under time reversal, $\nu_N^{K\cup \sK}(W^*_{-,+}(K^\prime))$ equals the intensity of backwards surviving, forwards killed trajectories hitting $K^\prime$. By the same calculation as in the proof of \cite[Lemma 7.2]{drewitz_critical_2023}, using \eqref{eq:last_exit}, we get, abbreviating $e_K\equiv e^{\sK}_{K,N}$, that this is further equal to $\langle e_{K}, G_N^{\sK} e_{K^\prime}\rangle = P_{e_K} (H_{K^\prime}<H_{\sK}) $. The upper bound now follows since $P_{x} \left(H_{K^\prime}<H_{\sK}\right)\leq 1$ for $x \in \partial K$. For the lower bound, requiring that $H_{K^\prime}<N^2$ and that the killing time of $X$ exceeds $N^2,$ one readily finds using our condition on $K^\prime$ and \cite[(2.51)]{lawler_random_2010} that
$ P_{x^\prime} \left(H_{K^\prime}<H_{\sK}\right) =P_{x^\prime} \left(H_{K^\prime}<\infty\right)\geq c$ for all $x^\prime \in K$.
\end{proof}
\begin{proof}[Proof of Lemma \ref{lem:li_connect}]
We follow the proof of \cite[Lemma 6.4]{drewitz_critical_2023}, with modifications described below.
For $L\geq 4$, it's easy to see that there exists an integer $1\leq p\leq c$ and $(z_0,z_1,\dots ,z_p)\subset \Lambda(L)$, a star connected chain in $\Lambda(L)$ (which corresponds to the collection of vertices constructed in \cite[(7.14)]{drewitz_critical_2023}) such that, %
\begin{align} \label{eq:defS}
\begin{split}
&\begin{array}{r}
z_0=(\lfloor \frac{30L}{s_L}+\frac12\rfloor s_L,0), \,\partial B_{30L}\subset\bigcup_{i=0}^pB(z_i,L), \,B(z_i,L) \cap \partial B_{30L}\neq \emptyset,
\end{array} \\[0.3em]
&\begin{array}{r}
\text{and }d(z_i,B_{\xi_\epsilon})\geq 20L \text{ for all }i\in\{0,\dots,p\}.
\end{array} \\[0.3em]
\end{split}
\end{align}
In view of \eqref{eq:defLell}, let $\mathcal{P}=\{z_0,\dots ,z_p\}\cup \{(is_L,0):\lfloor\frac{30L}{s_L}+\frac12\rfloor\leq i\leq \lceil (N+L+1)/s_L \rceil\}$ and note that our choice of $z_0$ implies that $\mathcal{P}$ is a star-connected chain in $\Lambda(L)$.
We next define a suitable event implementing the desired connection in \eqref{eq:li_connect}, which will be similar to the one illustrated in \cite[Figure 1]{drewitz_critical_2023}.
The main difference from the setting of \cite{drewitz_critical_2023} is that additional care is required for connection between $K$ to $\partial B_{\sqrt{\sigma'}\ell}$ therein due to the presence of a positive killing measure in our setup.
On the other hand, planarity makes it possible to replace \cite[(7.15)]{drewitz_critical_2023} by the simpler event \eqref{eq:IuG}.
Let $u=(\eta \ba/2)^2/4$. By \eqref{eqcouplingintergff}, there exists a coupling $\P^{K\cup\sK}_N\otimes \bar\P^{K\cup\sK}_N$ of $\tilde \varphi$ under $\P^{K\cup\sK}_N$ with $\mathcal{I}^{2u}$ under $\bar\P^{K\cup\sK}_N$ such that $\{ \tilde \varphi \geq -\eta \ba/2 \} \supseteq \mathcal{I}^{2u}$, and $\I^{2u}$ splits as $\I^{2u}=\I^{u}\cup{\mathcal{J}}^{u}$ into two independent interlacement sets at level $u$.
We write $\I^{u}_{-}$ for the set of vertices visited by trajectories comprising $\I^{u}$ whose backward parts are killed on $K$, and $\I^{u, L}_{-} \subset \I^u_{-}$ for the set of vertices in $\I^{u}_{-}$ visited before the first exit time from $\tB_{100L}$. Recall the event $\Con$ from \eqref{eq:LU_event} and consider two (good) events
\begin{equation} \label{eq:IuG}
G_1\defeq \big\{\exists i\in\{0,\dots,p\}\text{ s.t. }B(z_i,L)\stackrel{\I^{u, L}_{-}}{\leftrightarrow}\partial B(z_i,4L)
\big\}, \quad 
G_2\defeq\bigcap_{x\in \mathcal{P}} \Con_{x,L}(\mathcal{J}^u) .
\end{equation}
We now argue that on the event $G_1\cap G_2$, there exists a connected path $\gamma\subseteq(\I^{u,L}_-\cup \mathcal{J}^u)\cap (\tB_{100L}\cup \cL)$ connecting $K$ and $\partial B_N$. Indeed, let $i_*\in\{0,\dots,p\}$ be such that $B(z_{i_*},L)\stackrel{\I^{u, L}_{-}}{\leftrightarrow}\partial B(z_{i_*},4L)$. It follows from \eqref{eq:LU_event} that the circuit formed by $\mathcal{J}^u$ in $\tB(z_{i_*},4L)\setminus\tB(z_{i_*},2L)$ must intersect $\I^{u, L}_{-}$ inside $\tB(z_{i_*},4L)$. The existence of $\gamma$ is then ensured by \eqref{eq:conChain}, \eqref{eq:defLell} and the definition of $\I^{u, L}_{-}$.
It then follows by \eqref{eqcouplingintergff}, independence and above reasoning that,
\begin{equation}
\label{eq:lb16}
\mathbb{P}^{K\cup\sK}_N\bigl(K \leftrightarrow \partial B_N\text{ in }
(\tB_{100L}\cup \cL)\cap \{ \tilde \varphi \ge -\eta\ba/2\}\bigr)\geq \bar\P^{K\cup\sK}_N(G_1)\cdot\bar\P^{K\cup\sK}_N(G_2).
\end{equation}
We first lower bound $\bar\P^{K\cup\sK}_N(G_1)$. Let $\I^{u}_{-+}\subset \I^{u}_{-} (\subset \I^{u})$ denote the killed-surviving interlacement set corresponding to $\I^{u}$. By the definition of $\{z_0,\dots,z_p\}$ in \eqref{eq:defS}, we have that
$\{\I^{u}_{-+} \cap \partial B_{30L+5L} \neq \emptyset\}\subseteq G_1$.
Since $10L\leq N$, {and $\partial B_{30L+5L}\subseteq \sK^c$}, Lemma \ref{lem:killed_survive_il} implies that the number of trajectories in the process underlying $\I^{u}_{-+}$ which hits $\partial B_{30L+5L}$ is a Poisson random variable with parameter bounded from below by $cu \capacity_N^{\sK}(K)$. Hence, using $u=(\eta \ba/2)^2/4$ and the fact that $\capacity_N^{\sK}(K)\ge t$, we have
\begin{equation}\label{eq:G'shaslargeproba}
\bar\P^{K\cup\sK}_N(G_1)
\geq1-e^{-c(\eta\ba)^2\capacity_N^{\sK}(K)}
\geq 1-e^{-c(\eta\ba)^2t}.
\end{equation}
It remains to control $\bar\P^{K\cup\sK}_N(G_2)$. By \eqref{eq:defLell} and \eqref{eq:defS}, we have that $B(x,8L)\subset \Z^2\setminus (B_{\xi_{\epsilon}}\cup \sK)$ for all $x\in \mathcal{P}$. And we necessarily have $N\geq L\geq 4$ upon asking $|\ba-\ba_*|^2\geq c$ for a large enough constant $c>0$. Thus, using the bound $|\mathcal{P}|\leq p+cN/L$ and applying \eqref{eq:Iu_closed_cont} we obtain that
\[\bar\P^{K\cup\sK}_N(G_2^c)\leq c'\frac NL \exp\{-cu\log(N/L)^{-1}\}\leq c\frac NL \exp\Big\{-\frac{c'(\eta\ba)^2}{\log(N/L)}\Big\}\leq \eta, \]
where we also used \eqref{eq:capB} and the bound $\capacity^{K\cup\sK}_N(A)\geq \capacity_N(A)$ for any set $A\subseteq\Z^2$ in the first inequality and last inequality holds for all $L$ as in \eqref{eq:supcrit_L} upon a change of variable in $\eta$.
Combining this with \eqref{eq:lb16} and \eqref{eq:G'shaslargeproba}, we obtain \eqref{eq:li_connect}.
\end{proof}

\section{Critical and off-critical harmonic averages} \label{sec:harm_avg}

This section gathers the calculations for the harmonic averages arising in the coarse-graining schemes of Sections~\ref{sec:ub}--\ref{sec:disc} (see Section~\ref{sec:ub} for why each of them arises), and contains the proofs of Lemmas~\ref{lem:critHarm},~\ref{lem:bad_box} and~\ref{lem:harm_avg}, which concern critical, and anchored, resp.~unrooted off-critical harmonic averages. The arguments are inspired by those of \cite{sznitman_disconnection_2015} in dimension $d\geq3$, with three main differences. First, the estimates are no longer of large deviation type: the averages have variances of logarithmic order, and the exact exponent $\vartheta$ of a polynomial decay must be retained. Second, to capture $\vartheta$, the variance of unrooted averages must be controlled with the sharp constant $(1+\eta)\capacity_N(\Sigma)^{-1}$, see Lemma~\ref{lem:supvar_harm_avg}, which constrains the boxes arising from coarse-graining. Third, the averages involve boxes at two different scales, stemming from Section~\ref{sec:disc}, and their regularity relies on a gradient estimate for harmonic functions in the presence of killing, see Lemma~\ref{L:grad-estim}.

To fit all purposes, we develop bounds on the expectation and variance of Gaussian processes indexed by a certain family $\mathscr{F}$. Let $\cC_1,\cC_2\subset \Z^2$ be disjoint, finite sets with $\cC_1\cup \cC_2 \neq \emptyset$. We write, for $i\in\{1,2\},$ $(D_x^i,U_x^i)_{x\in\cC_i}$ — not to be confused with the sets denoted by $U_x^L$ and $D_x^L$ in \eqref{eq:various_boxes} — for a collection of nonempty subsets of $%\tilde
{\mathbb{Z}}^{2}$ such that there exist $r_i \ge r'_i \ge 1$ with
\begin{equation}\label{eq:cond_harnack}
D_x^i \subseteq B(x,r'_i), \qquad B(x,4r_i) \subseteq U_x^i ,\qquad \text{for all }x\in\cC_i.
\end{equation}
By convention, if $\cC_i=\emptyset$, we set $r_i'=0$ and $r_i=\infty$.
For conciseness, we define 
\begin{equation}\label{eq:short_DU}
( D_x,U_x)\defeq
\begin{cases}
(D_x^1,U_x^1), & x\in\cC_1,\\
(D_x^2,U_x^2), & x\in\cC_2, 
\end{cases}
\qquad \text{and}\quad \cC\defeq\cC_1\cup\cC_2 (\neq \emptyset ).%\red \subset B_{kN} \text{ for some $k,N \geq 1$}.
\end{equation}
We further let, with $D_x$, $x\in \cC$ as in \eqref{eq:short_DU}, 
\begin{equation}\label{def:F}
    \mathscr{F}\defeq\big\{f\in (%\tilde
    {\Z}^2)^{\cC}:f(x)\in D_x \text{ for each } x\in \cC\big\}.
\end{equation}
For $x\in \cC$, we write
$\eta_x(\cdot)$ to mean $\eta^{U_x^{\text{c}}}(\cdot)$ (see \eqref{eq:field_decomp} for notation). Now
consider, for $\mu$  a probability measure on $\cC$, a Gaussian field indexed by $f \in \mathscr{F}$ of the form
\begin{equation}\label{def:auxilliary_field}
    Z_f=Z_f(\cC,\mu,(D_x,U_x)_{x\in\cC})\defeq \sum_{x\in \cC} \mu(x)\eta_x (f(x)),\quad \text{for }f\in \mathscr{F},
\end{equation}

\subsection{Sharp variance estimate}\label{sec:sharp_var}
We start with the key ingredient for the proof of Lemma~\ref{lem:harm_avg}, which is a sharp bound on the variance of $Z_f$ for a suitable choice of $\mu$. It plays the analogous role as Lemma~\ref{lem:var_conditioning} below in the proofs of Lemmas~\ref{lem:critHarm} and~\ref{lem:bad_box}. Its proof follows that of \cite[(4.16)]{sznitman_disconnection_2015}, the constraints \eqref{eq:harm_avg_cond} being needed to obtain the sharp prefactor $(1+\eta)$ in \eqref{eq:supvar_harm_avg}.
\begin{lemma} \label{lem:supvar_harm_avg}
Let $\eta\in(0,1)$, $k \geq 1$, and $L_1\geq  L_2 \geq 1$, $M_1\geq M_2 \ge 1$, $N\geq 100M_1L_1$ and
\begin{equation}\label{eq:harm_avg_cond}
    \max_{i\in\{1,2\}}\frac{\log(N/(M_iL_i))}{\log(N/L_i)}\leq \Cr{nml_eta}(\eta  )  \quad \text{and}\quad M_1^{-1}\leq \Cr{ml_eta}(\eta,  k).
\end{equation}
For any $(\cC,\mu,(D_x,U_x)_{x\in \cC})$ such that:
\begin{enumerate}[label={(\roman*)}]
    \item $d(\cC_1,\cC_2)\geq 15M_1L_1$, $\cC\cap B_{2M_1L_1}=\emptyset$, and $B(\cC_i, 4L_i ) \subset B_{kN}$, $i\in\{1,2\}$, \label{fam_i}
    \item $\cC_i$ (cf.~\eqref{def:C}) consists of vertices in  $\Lambda(L_i)$ at distance $\geq 16M_iL_i$, $ i\in\{1,2\}$,\label{fam_ii}
    \item $(D_x^i,U_x^i)=(%\t
    B(x,4L_i),%\t
    B(x,M_iL_i))$ for all $x\in \cC_i, i\in\{1,2\}$,\label{fam_iii}
    \item for all $i\in\{1,2\},x\in \cC_i$, abbreviating $\Sigma= \Sigma(\cC_1,\cC_2)$ (cf.~\eqref{def:C_boxes} for notation)\label{fam_iv}
    \[\mu(x) \defeq\frac{e_{\Sigma,N}(\bC_x)}{\capacity_N(\Sigma)}=\frac{\sum_{x^\prime\in \bC_x }e_{\Sigma,N}(x^\prime )}{\capacity_N(\Sigma)},\quad\text{where }C_x\defeq
B(x,L_i), \text{ if } x\in\cC_i,\]
\end{enumerate}
we have with $B=B_{L_1/2}$ that
\begin{equation} \label{eq:supvar_harm_avg}
    \sup_{f\in\mathscr{F}}\E_N^{B}[Z_f^2]\leq ({1+\eta}){\capacity_N(\Sigma)^{-1}} .
\end{equation}
\end{lemma}

\begin{proof}
A calculation based on \cite[(1.49)]{sznitman_topics_2012} (cf.~also \eqref{eq:var_calculation} below) and using properties \ref{fam_i}-\ref{fam_iii} of $(\cC,\mu,(D_x,U_x)_{x\in \cC})$ yields that
\[
\E^{B}_N\left[\eta_x (f(x)) \eta_{x^\prime}(f(x^\prime))\right]=E_{f(x)}\big[g_N^{B}(f(x^\prime),X_{H_{U_x^c}} ) 1\{H_{U_x^c}<H_B\}\big], \quad x,x^\prime \in \cC.
\]
If in addition $x\neq x^\prime \in \cC$ (so $f(x^\prime)\in U_x^{\text{c}}$), it follows using \cite[(1.49)]{sznitman_topics_2012} that the right-hand side equals $g^{B}_N(f(x),f(x^\prime))$.
Combining these observations and recalling \eqref{def:auxilliary_field}, we get that
\begin{equation} \label{eq:sup_var1}
    \E_N^{B}[Z_f^2]\leq \sum_{x\in \cC} \mu(x)^2 \sup_{u\in \partial U_x}g_N(f(x),u)
    + \sum_{x,x^\prime\in \cC,x\neq x^\prime} \mu(x)\mu(x^\prime) g_N(f(x),f(x^\prime)).
\end{equation}
Now note that by \ref{fam_iv} and \eqref{eq:capB}, we have that for all $i\in\{1,2\},x\in \cC_i$,
$ \capacity_N(\Sigma) \times \mu(x)\leq    {\capacity_N(\bC_x)}\leq {c}{\log(N/L_i)^{-1}}$.
Hence using this,  \eqref{eq:g2-asymp-unif} and \eqref{eq:harm_avg_cond}, one obtains that \begin{align}\label{sup_var_ondiag}
\begin{split}
    \sum_{x\in \cC} \mu(x)^2 \sup_{u\in \partial U_x}g_N(f(x),u) 
    &\leq
     \max_{i\in\{1,2\}}\sup_{x\in \cC,u\in \partial U_x}\frac{ g_N(f(x),u)}{\log(N/L_i)\capacity_N(\Sigma)}\times \sum_{y\in \cC}\mu(y)\\
    &\leq 
    \max_{i\in\{1,2\}}\frac{c\log(N/{M_iL_i})}{\log(N/L_i)\capacity_N(\Sigma)}\leq \frac{\eta}{\capacity_N(\Sigma)},
\end{split}
\end{align}
upon choosing $\Cr{nml_eta}(\eta)$ small enough. Moving on to the second term on the r.h.s of \eqref{eq:sup_var1}, let 
    \begin{equation*}
        \alpha(x,x^\prime)%=\alpha_{L_1,L_2,M_1,M_2}(x,x^\prime)
        \defeq \sup_{u, u^\prime} \sup_{v,v^\prime  }
        \frac{g_N(v, v^\prime)}{g_N(u,u^\prime)}(\geq 1),\quad \text{for } x,x^\prime \in \cC, x \neq x^\prime,
    \end{equation*}
    where the suprema are over $u\in \bC_x,u^\prime \in \bC_{x^\prime}$ and $v\in D_x,v^\prime \in D_{x^\prime}$.
By definition of $\alpha$ and of $\mathscr{F}$ in \eqref{def:F}, substituting for $\mu(x)$ and $\mu(x^\prime)$ from \ref{fam_iv}, we obtain that the term $\sum_{x \neq x^\prime} \mu(x)\mu(x^\prime) g_N(f(x),f(x^\prime))$ with $x,x^\prime \in \cC$ appearing in \eqref{eq:sup_var1} is bounded by
\begin{align}\label{sup_var_cross}
\begin{split}
 %   &\sum_{x,x^\prime\in \cC,x\neq x^\prime} \mu(x)\mu(x^\prime) g_N(f(x),f(x^\prime))\\
     &\frac{\sup_{x\neq x^\prime }\alpha(x,x^\prime)}{\capacity_N(\Sigma)^2}
    \sum_{x \neq x^\prime}\sum_{ u\in \bC_x,u^\prime\in \bC_{x^\prime}} e_{\Sigma,N}(u)e_{\Sigma,N}(u^\prime)g_N(u,u^\prime)  \\
    &\stackrel{\eqref{def:C_boxes}}{\leq} \frac{\sup_{x\neq x^\prime}\alpha(x,x^\prime)}{\capacity_N(\Sigma)^2}
    \sum_{x,x^\prime\in \Sigma} e_{\Sigma,N}(x)e_{\Sigma,N}(x^\prime)g_N(x,x^\prime)= \frac{\sup_{x\neq x^\prime}\alpha(x,x^\prime)}{\capacity_N(\Sigma)},
\end{split}
\end{align}
where the last identity follows from
\eqref{eq:last_exit} and the definition of capacity.
It remains to bound $\alpha(x,x^\prime)$ for all $x\neq x^\prime \in \cC$. As we now explain, upon taking $N \geq c$ for some large enough $c =c(\eta, k)$, which can always be arranged by choosing $\Cr{ml_eta}(\eta, k)>0$ sufficiently small (indeed, by assumption and using \eqref{eq:harm_avg_cond} we have that $N\geq 100M_1L_1 \geq 100 \Cr{ml_eta}^{-1}$), we can ensure that
\[
\alpha(x,x^\prime)\leq (1+\eta)
    \frac{K_0\big(2\frac{|x-x^\prime|-8L_i}{N}\big)}{K_0\big(2\frac{|x-x^\prime|+2L_i}{N}\big)}\leq (1+\eta)^2 , \text{ where $i=2$ if $x,x'\in \cC_2$ and $i=1$ otherwise;}
\]
here, the first inequality follows straightforwardly by combining \eqref{eq:g2-asymp-unif}, the first item in \ref{fam_i}, and \ref{fam_iii}, noting in particular that $g_N$ is evaluated at points all within $B_{kN}$. The second inequality holds upon using \ref{fam_i}, \ref{fam_ii} and applying Lemma \ref{lem:bessel_cont} with $t=2\frac{|x-x^\prime|+2L_i}{N}$ and $\zeta=\frac{20L_i}{N}$ and note that we can always make sure $\zeta/t=\frac{10L_i}{|x-x^\prime|+2L_i}\leq \frac{10L_i}{15M_iL_i}\leq \Cr{k0}(\eta)$ by \ref{fam_i}, \ref{fam_ii} and \eqref{eq:harm_avg_cond}. All in all, \eqref{eq:supvar_harm_avg} thus follows from \eqref{eq:sup_var1} and the above bounds on its individual terms after a change of variable in $\eta$.
\end{proof}

\subsection{Supremum bounds and proofs of Lemmas~\ref{lem:critHarm}, \ref{lem:bad_box} and~\ref{lem:harm_avg}}\label{sec:sup_bounds}
The next result controls the expected supremum $Z_f$ in terms of the maximal variance. 

\begin{lemma} \label{L:supharm}
For all $N\geq 1$ and $r_1\vee r_2\leq 100N$, all probability measures $\mu$ on $\cC$ and $(D_x)_{x\in\cC},(U_x)_{x\in\cC}$ satisfying \eqref{eq:cond_harnack}--\eqref{eq:short_DU}, all $K\subset {\Z}^2\setminus (\bigcup_{x\in \cC}U_x)$, we have with $r'=r_1'\vee r_2'$ and $r=r_1\wedge r_2$ that, 
\begin{equation} \label{eq:borelltis_supnorm}
    \E_N^K\big[\sup_{f\in\mathscr{F}} Z_f\big]\leq c \frac{r^\prime}{r}\Big(|\cC|\times \sup_{f\in \mathscr{F}}\E_N^K[Z_f^2] \Big)^{1/2}.
\end{equation}
\end{lemma}
Lemma~\ref{L:supharm} is proved in Section~\ref{App:entropy}, along the lines of \cite[Lem.~4.3]{sznitman_disconnection_2015}; the main new two-dimensional input is a gradient estimate (Lemma~\ref{L:grad-estim}), which follows from \eqref{eq:killed_green_smallh} and an argument from~\cite{grigoryan_sub-gaussian_2001}.

Next, we control the quantity  $\sup_{f\in\mathscr{F}}\E_N^K[Z_f^2]$ appearing on the right side of \eqref{eq:borelltis_supnorm}, with further specific choices of $\big(\cC,\mu,(D_x,U_x)_{x\in \cC}\big)$, which are relevant to Lemma~\ref{lem:critHarm} and Lemma \ref{lem:bad_box}.

\begin{lemma} \label{lem:var_conditioning}
\begin{enumerate}[label={(\roman*)}]
\item  
%Let $L\geq 10\xi$ (cf.~\eqref{def:correlation_length}), $N\geq 2L$ and 
Pick $\big(\cC,\mu,(D_x,U_x)_{x\in \cC}\big)$ to be such that $\cC_1=\{0\},\cC_2=\emptyset,\mu(\cdot)=\delta_{0}(\cdot)$, and
$(D_0,U_0)$ to be such that $\text{diam}(U)\leq 100N$.
%$(D_0,U_0)=(\tB_\xi,\tB_{L/2})$. 
Then with $d=\inf_{y\in D,z\in U^c}|y-z|$,
\begin{equation} \label{eq:bar_eta_var}
\sup_{f\in\mathscr{F}}\E_N[Z_f^2]\leq c\log((N/d)\vee 2).
\end{equation}
\item  Let $N,L,R\geq1$. For any $(\cC,\mu,(D_x,U_x)_{x\in \cC})$ such that $\cC\neq\emptyset$, $\mu(\cdot)=\frac{1}{|\cC|}  \delta_{\cC}(\cdot)$, and 
\[d(B(x,R),B_{L/2})\ge R,\quad
D_x\subseteq %\t
B(x,R),\quad
U_x=(%\t
B_{L/2})^c \quad\text{for all }x\in\cC,\]
we have that, with $g(N,R)$ as in Lemma \ref{lem:bad_box},
\begin{equation}\label{eq:outside_var}
   \sup_{f\in\mathscr{F}}\E_N[Z_f^2]\leq c\frac{g(N,R)}{\log(N/L)}. 
\end{equation}
\end{enumerate}
\end{lemma}

\begin{proof}
First note that for both \eqref{eq:bar_eta_var} and \eqref{eq:outside_var}, we have that by \cite[(1.49)]{sznitman_topics_2012}, 
%with $U =\partial B_{L/2}$,
\begin{equation} \label{eq:var_calculation}
    \E_N\left[Z_f^2\right]
    =\frac{1}{|\cC|^2}\sum_{x,x^\prime \in \cC}E_{f(x)}\left[g_N(f(x^\prime),X_{H_{U^c}})\1\big\{H_{U^c}<\infty \big\}\right].
\end{equation}
In the case of \eqref{eq:bar_eta_var}, this imply reduces with  $z=f(0) \in D_0$ to,
\begin{equation*} 
   \E_N\big[\big(\eta_0({z})\big)^2\big] 
   =E_{z}\big[g_N(z,X_{H_{U^c}})\1\big\{H_{U^c}<\infty \big\}\big]
   \leq  c\log((N/d)\vee 2),
\end{equation*}
where the last bound follows from \eqref{eq:bessel_ub} and  \eqref{eq:g2-asymp-unif}. For \eqref{eq:outside_var}, using bound $d(f(x),B_{L/2})\geq R$, which holds for all $f\in \mathscr{F}, x\in \cC$, and \eqref{eq:var_calculation}, we have that
\begin{align*}
    \E_N\left[Z_f^2\right]
   % &\leq \sup_{x,x^\prime \in \cC}\sup_{y \in U}g_N(f(x^\prime),y)P_{f(x)}(H_{U}<\infty)\\
    &\leq  \capacity_N(U^c)\times \sup_{x \in \cC,y \in U^c}g_N(f(x),y)^2 \leq c\frac{\sup_{x \in \cC,y \in U^c}g_N(f(x),y)^2}{\log(N/L)},
\end{align*}
where we used \eqref{eq:last_exit} and \eqref{eq:capB}. Therefore, when $N\geq 100R$, \eqref{eq:outside_var} follows from \eqref{eq:bessel_ub} and \eqref{eq:g2-asymp-unif}. When $N<100R$, applying \cite[Proposition A.2]{rz25a} with $r=1,d=2,M=0$ yields that for all $|x-y|\geq R$,
\[g_N(x,y)\leq \frac{c}{R^2} + \int_{0}^\infty\frac{ e^{-\frac{|x-y|^2}{t}-t/N^2}}{\pi t}\,dt ,\]
and moreover, the integral in the previous display is simply $\frac{2}{\pi}K_0(2\frac{|x-y|}{N})$ in view of \eqref{eq:K_0}. Using the $\cosh$ integral representation of $K_0$ (see \cite[(1.4)]{rz25a}) and the inequality $\cosh(x)\geq 1+x^2/2$ for all $x\geq 0$ we further see that for all $r\geq 10^{-3}$, $K_0(r)\leq \int_0^\infty\exp\{-r( 1+x^2/2)\}\,dx\leq ce^{-r}.$ Hence if $R>N/100$,
\[
\sup_{x \in \cC,y \in U}g_N(f(x),y)\leq c(R^{-2}+e^{-c'R/N})\leq c\big(R^2\wedge e^{R/N}\big)^{-1},
\]
which concludes the proof of \eqref{eq:outside_var}.
\end{proof}

We are now ready to finish the: 
\begin{proof}[Proofs of Lemmas \ref{lem:critHarm},~\ref{lem:bad_box} and~\ref{lem:harm_avg}]
The key is to note that by \eqref{def:F} and \eqref{def:auxilliary_field}, and since $\mu$ is a probability measure, for all $a \in \R$ one has the inclusion $\{ \sup_{x^\prime \in D_x} \eta_x(x^\prime) \geq a \text{ for all } x\in \cC \}\subset \{\sup_{f\in\mathscr{F}} Z_f \geq a\}$. We now show Lemma \ref{lem:harm_avg}. 
Consider $L_1=L,L_2=L^\prime$, $M_1=M,M_2=M^\prime$, $N\geq 1$. The conditions from Lemma~\ref{lem:harm_avg} ensure that condition \eqref{eq:harm_avg_cond} is satisfied and Lemma
~\ref{lem:supvar_harm_avg} applies to $\cC=\cC_1$ and $\cC^\prime=\cC_2$ and $(D_x,U_x)_{x\in \cC}$ as declared by \eqref{eq:short_DU} and \ref{fam_iii} of Lemma~\ref{lem:supvar_harm_avg}. Picking $\mu$ as in property \ref{fam_iv} of Lemma \ref{lem:supvar_harm_avg}, this defines the centred Gaussian process $(Z_f)_{f\in \mathscr{F}}$ under $\P^{L/2}_N$, where $Z_f=Z_f(\cC,\mu,(D_x,U_x)_{x\in \cC})$. Now note that the event $\mathsf{F}(\cC,\cC')$ introduced in \eqref{eq:EF_events} is of the form appearing on the left of this inclusion. Therefore, $\P_N^{L/2}\left(\mathsf{F}(\cC,\cC')\right) \leq \P^{L/2}_N(\sup_{f\in\mathscr{F}} Z_f \geq (1-\epsilon)|\ba|)$. Now, applying the BTIS inequality, \eqref{eq:borelltis_supnorm}, and \eqref{eq:supvar_harm_avg}, we deduce that the deviation probability for $\sup_{f\in\mathscr{F}} Z_f $ is bounded by the right-hand side of \eqref{eq:harm_avg}. This concludes the proof of Lemma \ref{lem:harm_avg}.
The proof of \eqref{eq:critHarm} when $D,U$ satisfies \eqref{eq:cond_harnack} and the proof Lemma~\ref{lem:bad_box} follow a similar reasoning but by applying Lemma \ref{lem:var_conditioning} instead of~\ref{lem:supvar_harm_avg}.

For general $D\subset U$ with $\diam(D)\leq \lambda d$, cover $D$ by $m\leq c\lambda^2$ balls $B(x_i,d/8)$, $x_i\in D$. Since $B(x_i,d/2)\subseteq U$, the case already treated applies to each pair $(D\cap B(x_i,d/8),U)$, and \eqref{eq:critHarm} follows by a union bound.

It remains to show \eqref{eq:critHarmExp} to conclude the proof of Lemma~\ref{lem:critHarm}. Writing $S=\sup_{x\in D}\eta^{U^c}_x$ and $\ell=\log((N/d)\vee 2)$, we have $\E_N[S^k\1\{S\geq t\}]=t^k\P_N(S\geq t)+\int_t^\infty ks^{k-1}\P_N(S\geq s)\,ds$, and \eqref{eq:critHarmExp} follows from \eqref{eq:critHarm} and the bound $\int_t^\infty s^{k-1}e^{-cs^2/\ell}\,ds\leq c(k)\,t^k e^{-ct^2/\ell}$, valid for $t^2/\ell\geq 1$.
\end{proof}

\subsection{Proof of Lemma~\ref{L:supharm}}\label{App:entropy}

We start with a gradient estimate for harmonic functions, which will be useful in the sequel. For $B \subset \Z^2$ let $\bar{B}= B \cup \partial (\Z^2 \setminus B)$. The function $u: \bar{B} \to \R$ is \emph{harmonic in $B$} if $Lu{|_B}=0$, where $L$ is the generator of the walk $X$, introduced in Section~\ref{sec:facts}. 

Via \cite[Corollary 5.2]{rz25a} (which, as it is derived from \cite[Proposition 5.1]{rz25a}, holds for all $R\leq C_{\textnormal{range}}N$ with a constant depending on the arbitrary parameter $C_{\textnormal{range}}\geq 1$ therein; we use it with $C_{\textnormal{range}}=400$, so that all constants below are absolute) and adapting the argument by Grigor'yan and Telcs in \cite[Proposition 11.1]{grigoryan_sub-gaussian_2001}, one obtains:

\begin{lemma}\label{L:grad-estim} There exists $C<\infty$ such that for all $N \geq 1$ and $1 \leq R \leq 200N$, and any function $u$ that's defined on $\bar B_{2R}$ and harmonic in $B_{2R}$,
\begin{equation} \label{eq:harm_grad}
    |u(x)-u(y)| \leq C \frac{ |x-y|}{R}  \sup_{z\in B_{2R}} |u(z)| \quad\quad \forall x,y \in B_{\frac12R}.
\end{equation} 
\end{lemma}
\begin{proof}[Proof of Lemma~\ref{L:supharm}]
    For $f,k\in \mathscr{F},y\in U_x$ and $u\in\partial U_x$, we abbreviate $h(y,u)=P_y(X_{H_{U_x^{\text{c}}}}=u,H_{U_x^{\text{c}}}<H_K)$. By definition of $\eta_x$ below \eqref{def:F}, this  implies that $\eta_x (y)= \sum_u h(y,u)\varphi_u$ for all $y \in D_x$. 
    Let $f,k\in \mathscr{F}$. We start by controlling 
    $\E^K_N[(Z_f-Z_k)^2]^{1/2}$,
    the so called ``canonical metric" on $\mathscr{F}$, %which is in fact only a pseudometric (
    cf.~\cite[(1.3.1)]{noauthor_random_2007}. 
    It follows from \eqref{def:auxilliary_field} and an easy calculation
 that $\E^K_N[(Z_f-Z_k)^2]$ equals
    \begin{align} \label{eq:canonical_metric_calculation}
    \begin{split} 
     &\sum_{x,x^\prime \in \mathcal{C}}
    \mu(x)\mu(x^\prime) \sum_{u,u^\prime} 
    \big(h(f(x), u)-h(k(x),u)
    \big)\big(h(f({x^\prime}), u^\prime)-h(k({x^\prime}), u^\prime)
    \big)\E_N^K[\varphi_u\varphi_{u^\prime}].
    \end{split}
    \end{align}
    By \eqref{eq:cond_harnack}, for $i\in\{1,2\}$, $x\in \cC_i$, $u\in \partial U_x$, $h(\cdot, u)$ is nonnegative and harmonic in $\tB(x,4r_i)$. It thus follows via \cite[Corollary 5.2]{rz25a} and \eqref{eq:harm_grad} that 
    there exists $C<\infty$ such that for $v,v^\prime \in D_x$ and $u\in \partial (U_x^c)$,
    \[\big| h(v,u)-h(v',u) \big| \leq {C}{r^{-1}} |v-v^\prime | \times
  h(v,u).\]
    Applying this to (\ref{eq:canonical_metric_calculation}) we obtain with $\|f-k\|_\infty := \sup_{x\in \mathcal{C}} |f(x)-k(x)|$ for $f,k\in\mathscr{F}$ that
    \begin{align}\label{eq:bdd_canonical_metric}
    \begin{split}
    \E^K_N[(Z_f-Z_k)^2] 
    &\leq \frac{(C)^2\|f-k\|_\infty^2}{r^2} \sum_{x,x^\prime \in \mathcal{C}}
    \mu(x)\mu(x^\prime) \sum_{u, u^\prime}h(f(x), u)\,
     h(f({x^\prime}), u^\prime)\E_N^K[\varphi_u\varphi_{u^\prime}] \\
    &\leq \frac{(C)^2\|f-k\|_\infty^2}{r^2} \E_N^K[Z_f^2] .
    \end{split}
    \end{align}
    Now consider the renormalised field $\tilde Z_f=\big(\sup_{f\in \mathscr{F}}\E_N^K[Z_f^2]\big)^{-1/2}Z_f$. By \eqref{def:F}, \eqref{eq:cond_harnack} and \eqref{eq:bdd_canonical_metric}, $\|f-k\|_\infty\leq \epsilon r/(C)$ implies that the canonical distance between $\tilde Z_f$ and $\tilde Z_k$ is at most $\epsilon$, and $\mathscr{F}$ has canonical diameter $\sup_{f,k\in\mathscr{F}}\E_N^K[(\tilde Z_f-\tilde Z_k)^2]^{1/2}$ (i.e.~the supremum over $f,k$ of the square root of the left-hand side of \eqref{eq:bdd_canonical_metric}, normalised by $\sup_{f}\E_N^K[Z_f^2]$) at most $2C\,r'/r$. Covering $\bigcup_{x\in\cC}D_x$ by $\ell^\infty$-balls, $\mathscr{F}$ is thus covered by $(c'r'/(\epsilon r))^{2|\cC|}$ balls of radius $\epsilon$, and \cite[Theorem 1.3.3]{noauthor_random_2007} yields
    \[
    \big(\sup_{f\in \mathscr{F}}\E_N^K[Z_f^2]\big)^{-1/2}\, \big|\E_N^K[\sup_{f\in \mathscr F}Z_f]\big|
    \leq c \sqrt{|\cC|} \int_0^{2Cr'/r} \sqrt{\log\big(c' r'/(\epsilon r)\big)} \,d\epsilon
    \leq \frac{c\,r'}{r}\sqrt{|\cC|}.
    \]
\end{proof}

\noindent\textbf{Acknowledgements.} The research of WZ is supported by the CDT in Mathematics of Random Systems and by the  President's PhD Scholarship scheme at Imperial College, London. The research of PFR is supported by the European Research Council (ERC) under the European Union’s Horizon Europe research and innovation program (grant agreement No 101171046).

{\normalsize
\makeatletter\renewcommand\@openbib@code{\setlength{\itemsep}{4pt plus 2pt minus 1pt}\setlength{\parsep}{0pt}}\makeatother
\phantomsection\addcontentsline{toc}{section}{References}
\bibliography{bibliography}

@book{fukushima_dirichlet_2011,
  author    = {Fukushima, Masatoshi and Oshima, Yoichi and Takeda, Masayoshi},
  title     = {Dirichlet Forms and Symmetric {M}arkov Processes},
  edition   = {Second},
  publisher = {De Gruyter},
  year      = {2011}
}

@article{fisher1972scaling,
  author    = {Fisher, Michael E. and Barber, Michael N.},
  title     = {Scaling Theory for Finite-Size Effects in the Critical Region},
  journal   = {Phys. Rev. Lett.},
  volume    = {28},
  issue     = {23},
  pages     = {1516--1519},
  numpages  = {4},
  year      = {1972},
  publisher = {American Physical Society},
  doi       = {10.1103/PhysRevLett.28.1516},
  url       = {https://aps.org}
}

@article{nitzschner_solidification_2020,
	title = {Solidification of porous interfaces and disconnection},
	volume = {22},
	issn = {1435-9855},
	url = {https://ems.press/journals/jems/articles/16901},
	doi = {10.4171/jems/973},
	language = {en},
	number = {8},
	urldate = {2026-01-15},
	journal = {J. Eur. Math. Soc.},
	author = {Nitzschner, Maximilian and Sznitman, Alain-Sol},
	year = {2020},
	pages = {2629--2672},
}

@article{lsw01,
	author = {Lawler, Gregory F. and Schramm, Oded and Werner, Wendelin},
	date = {2001/09/01},
	doi = {10.1007/BF02392619},
	id = {Lawler2001},
	isbn = {1871-2509},
	journal = {Acta Math.},
	number = {2},
	pages = {275--308},
	title = {Values of Brownian intersection exponents, II: Plane exponents},
	url = {https://doi.org/10.1007/BF02392619},
	volume = {187},
	year = {2001}}

@article{zbMATH07704060,
 author = {Sznitman, Alain-Sol},
 title = {On the cost of the bubble set for random interlacements},
 fjournal = {Inventiones Mathematicae},
 journal = {Invent. Math.},
 issn = {0020-9910},
 volume = {233},
 number = {2},
 pages = {903--950},
 year = {2023},
 language = {English},
 doi = {10.1007/s00222-023-01190-9},
 zbMATH = {7704060},
 Zbl = {1528.60105}
}

@Article{GRS24.1,
 Author = {Goswami, Subhajit and Rodriguez, Pierre-Fran{\c{c}}ois and Shulzhenko, Yuriy},
 title = {Strong local uniqueness for the vacant set of random interlacements},
 Journal = {Preprint},
Volume = {arXiv:2503.14497},
Year = {2025},
Language = {English},
}

@Article{GRS24.2,
 Author = {Goswami, Subhajit and Rodriguez, Pierre-Fran{\c{c}}ois and Shulzhenko, Yuriy},
 title = {Sharp connectivity bounds for the vacant set of random interlacements},
 Journal = {Ann. Appl. Probab.},
 Volume = {36},
 Year = {2026},
 Note = {arXiv:2504.02777},
}

@misc{NIST:DLMF,
         key = "{\relax DLMF}",
       title = "{\it NIST Digital Library of Mathematical Functions}",
howpublished = "\url{https://dlmf.nist.gov/}, Release 1.2.4 of 2025-03-15",
         url = "https://dlmf.nist.gov/",
        note = "F.~W.~J. Olver et al., eds."}

@article{rz25a,
	title = {On cable-graph percolation between dimensions 2 and 3},
  journal={arXiv preprint arXiv:2512.05947},
	author = {Rodriguez, Pierre-Fran\c{c}ois and Zhang, Wen},
	year = {2025},
}

@book{sznitman_brownian_1998,
	title = {Brownian {Motion}, {Obstacles} and {Random} {Media}},
	isbn = {978-3-642-08420-1 978-3-662-11281-6},
	url = {https://link.springer.com/10.1007/978-3-662-11281-6},
	language = {en},
	urldate = {2023-12-24},
	publisher = {Springer},
	author = {Sznitman, Alain-Sol},
	year = {1998},
	doi = {10.1007/978-3-662-11281-6},
}

@book{barlow_random_2017,
	title = {Random {Walks} and {Heat} {Kernels} on {Graphs}},
	isbn = {978-1-107-67442-4},
	url = {https://www.cambridge.org/core/books/random-walks-and-heat-kernels-on-graphs/5B375D343025BCE91C682D49CDDB3A1A},
	urldate = {2024-01-06},
	publisher = {Cambridge University Press},
	author = {Barlow, Martin T.},
	year = {2017},
	doi = {10.1017/9781107415690},
}

@book{lawler_random_2010,
	title = {Random {Walk}: {A} {Modern} {Introduction}},
	isbn = {978-0-521-51918-2},
	shorttitle = {Random {Walk}},
	url = {https://www.cambridge.org/core/books/random-walk-a-modern-introduction/7DA2A372B5FE450BB47C5DBD43D460D2},
	urldate = {2024-01-06},
	publisher = {Cambridge University Press},
	author = {Lawler, Gregory F. and Limic, Vlada},
	year = {2010},
	doi = {10.1017/CBO9780511750854},
}

@article{drewitz_critical_2023,
	title = {Critical exponents for a percolation model on transient graphs},
	volume = {232},
	issn = {1432-1297},
	url = {https://doi.org/10.1007/s00222-022-01168-z},
	doi = {10.1007/s00222-022-01168-z},
	language = {en},
	number = {1},
	urldate = {2024-01-27},
	journal = {Invent. Math.},
	author = {Drewitz, Alexander and Pr{\'e}vost, Alexis and Rodriguez, Pierre-Fran{\c c}ois},
	year = {2023},
	pages = {229--299},
}

@article{drewitz_cluster_2022,
	title = {Cluster capacity functionals and isomorphism theorems for {Gaussian} free fields},
	volume = {183},
	issn = {1432-2064},
	url = {https://doi.org/10.1007/s00440-021-01090-0},
	doi = {10.1007/s00440-021-01090-0},
	language = {en},
	number = {1},
	urldate = {2024-01-27},
	journal = {Probab. Theory Related Fields},
	author = {Drewitz, Alexander and Pr{\'e}vost, Alexis and Rodriguez, Pierre-Fran{\c c}ois},
	year = {2022},
	pages = {255--313},
}

@article{goswami_radius_2022,
	title = {On the radius of {Gaussian} free field excursion clusters},
	volume = {50},
	issn = {0091-1798, 2168-894X},
	url = {https://projecteuclid.org/journals/annals-of-probability/volume-50/issue-5/On-the-radius-of-Gaussian-free-field-excursion-clusters/10.1214/22-AOP1569.full},
	doi = {10.1214/22-AOP1569},
	number = {5},
	urldate = {2024-01-29},
	journal = {Ann. Probab.},
	author = {Goswami, Subhajit and Rodriguez, Pierre-Fran{\c c}ois and Severo, Franco},
	year = {2022},
	pages = {1675--1724},
}

@article{prevost_percolation_2023,
	title = {Percolation for the {Gaussian} free field on the cable system: counterexamples},
	volume = {28},
	issn = {1083-6489, 1083-6489},
	shorttitle = {Percolation for the {Gaussian} free field on the cable system},
	url = {https://projecteuclid.org/journals/electronic-journal-of-probability/volume-28/issue-none/Percolation-for-the-Gaussian-free-field-on-the-cable-system/10.1214/23-EJP949.full},
	doi = {10.1214/23-EJP949},
	urldate = {2024-02-05},
	journal = {Electron. J. Probab.},
	author = {Pr{\'e}vost, Alexis},
	year = {2023},
	pages = {1--43},
}

@misc{jego_crossing_2023,
	title = {Crossing exponent in the {Brownian} loop soup},
	url = {http://arxiv.org/abs/2303.03782},
	doi = {10.48550/arXiv.2303.03782},
	urldate = {2024-03-25},
	publisher = {arXiv},
	author = {Jego, Antoine and Lupu, Titus and Qian, Wei},
	year = {2023},
	note = {arXiv:2303.03782, to appear in Ann. Probab.},
}

@article{lupu_loop_2016,
	title = {From loop clusters and random interlacements to the free field},
	volume = {44},
	issn = {0091-1798, 2168-894X},
	url = {https://projecteuclid.org/journals/annals-of-probability/volume-44/issue-3/From-loop-clusters-and-random-interlacements-to-the-free-field/10.1214/15-AOP1019.full},
	doi = {10.1214/15-AOP1019},
	number = {3},
	urldate = {2024-04-02},
	journal = {Ann. Probab.},
	author = {Lupu, Titus},
	year = {2016},
	pages = {2117--2146},
}

@article{ding_percolation_2020,
	title = {Percolation for level-sets of {Gaussian} free fields on metric graphs},
	volume = {48},
	issn = {0091-1798, 2168-894X},
	url = {https://projecteuclid.org/journals/annals-of-probability/volume-48/issue-3/Percolation-for-level-sets-of-Gaussian-free-fields-on-metric/10.1214/19-AOP1397.full},
	doi = {10.1214/19-AOP1397},
	number = {3},
	urldate = {2024-04-02},
	journal = {Ann. Probab.},
	author = {Ding, Jian and Wirth, Mateo},
	year = {2020},
	pages = {1411--1435},
}

@article{ding_chemical_2018,
	title = {Chemical {Distances} for {Percolation} of {Planar} {Gaussian} {Free} {Fields} and {Critical} {Random} {Walk} {Loop} {Soups}},
	volume = {360},
	issn = {1432-0916},
	url = {https://doi.org/10.1007/s00220-018-3140-x},
	doi = {10.1007/s00220-018-3140-x},
	language = {en},
	number = {2},
	urldate = {2024-04-02},
	journal = {Comm. Math. Phys.},
	author = {Ding, Jian and Li, Li},
	year = {2018},
	pages = {523--553},
}

@article{rodriguez_phase_2013,
	title = {Phase {Transition} and {Level}-{Set} {Percolation} for the {Gaussian} {Free} {Field}},
	volume = {320},
	issn = {1432-0916},
	url = {https://doi.org/10.1007/s00220-012-1649-y},
	doi = {10.1007/s00220-012-1649-y},
	language = {en},
	number = {2},
	urldate = {2024-04-02},
	journal = {Comm. Math. Phys.},
	author = {Rodriguez, Pierre-Fran{\c c}ois and Sznitman, Alain-Sol},
	year = {2013},
	pages = {571--601},
}

@article{gandolfi_uniqueness_1988,
	title = {On the {Uniqueness} of the {Infinite} {Occupied} {Cluster} in {Dependent} {Two}- {Dimensional} {Site} {Percolation}},
	volume = {16},
	issn = {0091-1798, 2168-894X},
	url = {https://projecteuclid.org/journals/annals-of-probability/volume-16/issue-3/On-the-Uniqueness-of-the-Infinite-Occupied-Cluster-in-Dependent/10.1214/aop/1176991681.full},
	doi = {10.1214/aop/1176991681},
	number = {3},
	urldate = {2024-04-02},
	journal = {Ann. Probab.},
	author = {Gandolfi, A. and Keane, M. and Russo, L.},
	year = {1988},
	pages = {1147--1157},
}

@misc{drewitz_arm_2023,
	title = {Arm exponent for the {Gaussian} free field on metric graphs in intermediate dimensions},
	url = {http://arxiv.org/abs/2312.10030},
	doi = {10.48550/arXiv.2312.10030},
	urldate = {2024-04-18},
	publisher = {arXiv},
	author = {Drewitz, Alexander and Pr{\'e}vost, Alexis and Rodriguez, Pierre-Fran{\c c}ois},
	year = {2023},
	note = {arXiv:2312.10030},
}

@misc{cai_one-arm_2024,
	title = {One-arm {Probabilities} for {Metric} {Graph} {Gaussian} {Free} {Fields} below and at the {Critical} {Dimension}},
	url = {http://arxiv.org/abs/2406.02397},
	doi = {10.48550/arXiv.2406.02397},
	urldate = {2024-06-05},
	publisher = {arXiv},
	author = {Cai, Zhenhao and Ding, Jian},
	year = {2024},
	note = {arXiv:2406.02397},
}

@book{sznitman_topics_2012,
	title = {Topics in {Occupation} {Times} and {Gaussian} {Free} {Fields}},
	publisher = {Eur. Math. Soc.},
	url = {https://ems.press/books/zlam/111/contents},
	language = {en},
	urldate = {2024-08-24},
	author = {Sznitman, Alain-Sol},
	year = {2012},
	doi = {10.4171/109},
}

@article{sznitman_disconnection_2015,
	title = {Disconnection and level-set percolation for the {Gaussian} free field},
	volume = {67},
	issn = {0025-5645, 1881-1167},
	url = {https://projecteuclid.org/journals/journal-of-the-mathematical-society-of-japan/volume-67/issue-4/Disconnection-and-level-set-percolation-for-the-Gaussian-free-field/10.2969/jmsj/06741801.full},
	doi = {10.2969/jmsj/06741801},
	number = {4},
	urldate = {2024-12-20},
	journal = {J. Math. Soc. Japan},
	author = {Sznitman, Alain-Sol},
	year = {2015},
	pages = {1801--1843},
}

@misc{drewitz_cluster_2024,
	title = {Cluster volumes for the {Gaussian} free field on metric graphs},
	url = {http://arxiv.org/abs/2412.06772},
	doi = {10.48550/arXiv.2412.06772},
	urldate = {2025-01-06},
	publisher = {arXiv},
	author = {Drewitz, Alexander and Pr{\'e}vost, Alexis and Rodriguez, Pierre-Fran{\c c}ois},
	year = {2024},
	note = {arXiv:2412.06772},
}

@article{muirhead_percolation_2024-1,
	title = {Percolation of strongly correlated {Gaussian} fields {I}. {Decay} of subcritical connection probabilities},
	volume = {5},
	issn = {2690-1005, 2690-0998},
	url = {http://arxiv.org/abs/2206.10723},
	doi = {10.2140/pmp.2024.5.357},
	number = {2},
	urldate = {2025-02-03},
	journal = {Probab. Math. Phys.},
	author = {Muirhead, Stephen and Severo, Franco},
	year = {2024},
	pages = {357--412},
}

@book{noauthor_random_2007,
	author = {Adler, Robert J. and Taylor, Jonathan E.},
	title = {Random {Fields} and {Geometry}},
	copyright = {http://www.springer.com/tdm},
	isbn = {978-0-387-48112-8 978-0-387-48116-6},
	url = {http://link.springer.com/10.1007/978-0-387-48116-6},
	language = {en},
	urldate = {2025-02-03},
	publisher = {Springer},
	year = {2007},
	doi = {10.1007/978-0-387-48116-6},
}

@article{grigoryan_sub-gaussian_2001,
	title = {Sub-{Gaussian} estimates of heat kernels on infinite graphs},
	volume = {109},
	issn = {0012-7094, 1547-7398},
	url = {https://projecteuclid.org/journals/duke-mathematical-journal/volume-109/issue-3/Sub-Gaussian-estimates-of-heat-kernels-on-infinite-graphs/10.1215/S0012-7094-01-10932-0.full},
	doi = {10.1215/S0012-7094-01-10932-0},
	number = {3},
	urldate = {2025-02-06},
	journal = {Duke Math. J.},
	author = {Grigor'yan, Alexander and Telcs, Andras},
	year = {2001},
	pages = {451--510},
}

@article{nitzschner_disconnection_2018,
	title = {Disconnection by level sets of the discrete {Gaussian} free field and entropic repulsion},
	volume = {23},
	issn = {1083-6489, 1083-6489},
	url = {https://projecteuclid.org/journals/electronic-journal-of-probability/volume-23/issue-none/Disconnection-by-level-sets-of-the-discrete-Gaussian-free-field/10.1214/18-EJP226.full},
	doi = {10.1214/18-EJP226},
	urldate = {2025-02-27},
	journal = {Electron. J. Probab.},
	author = {Nitzschner, Maximilian},
	year = {2018},
	pages = {1--21},
}

@misc{cai_incipient_2025,
	title = {Incipient infinite clusters and volume growth for {Gaussian} free fields and loop soups on metric graphs},
	url = {http://arxiv.org/abs/2412.05709},
	doi = {10.48550/arXiv.2412.05709},
	urldate = {2025-04-02},
	publisher = {arXiv},
	author = {Cai, Zhenhao and Ding, Jian},
	year = {2025},
	note = {arXiv:2412.05709},
}

@article{aru_first_2020,
	title = {The {First} {Passage} {Sets} of the {2D} {Gaussian} {Free} {Field}: {Convergence} and {Isomorphisms}},
	volume = {375},
	issn = {1432-0916},
	shorttitle = {The {First} {Passage} {Sets} of the {2D} {Gaussian} {Free} {Field}},
	url = {https://doi.org/10.1007/s00220-020-03718-z},
	doi = {10.1007/s00220-020-03718-z},
	language = {en},
	number = {3},
	urldate = {2025-05-16},
	journal = {Comm. Math. Phys.},
	author = {Aru, Juhan and Lupu, Titus and Sep{\'u}lveda, Avelio},
	year = {2020},
	pages = {1885--1929},
}

@misc{gao_percolation_2024-2,
	title = {Percolation of discrete {GFF} in dimension two {I}. {Arm} events in the random walk loop soup},
	url = {http://arxiv.org/abs/2409.16230},
	doi = {10.48550/arXiv.2409.16230},
	urldate = {2025-09-29},
	publisher = {arXiv},
	author = {Gao, Yifan and Nolin, Pierre and Qian, Wei},
	year = {2024},
	note = {arXiv:2409.16230},
}

@article{kohler-schindler_crossing_2023,
	title = {Crossing probabilities for planar percolation},
	volume = {172},
	issn = {0012-7094, 1547-7398},
	url = {https://projecteuclid.org/journals/duke-mathematical-journal/volume-172/issue-4/Crossing-probabilities-for-planar-percolation/10.1215/00127094-2022-0015.full},
	doi = {10.1215/00127094-2022-0015},
	number = {4},
	urldate = {2025-09-29},
	journal = {Duke Math. J.},
	author = {K{\"o}hler-Schindler, Laurin and Tassion, Vincent},
	year = {2023},
	pages = {809--838},
}

@misc{duminil-copin_near_2021,
	title = {Near critical scaling relations for planar {Bernoulli} percolation without differential inequalities},
	url = {http://arxiv.org/abs/2111.14414},
	doi = {10.48550/arXiv.2111.14414},
	urldate = {2025-10-01},
	publisher = {arXiv},
	author = {Duminil-Copin, Hugo and Manolescu, Ioan and Tassion, Vincent},
	year = {2021},
	note = {arXiv:2111.14414},
}

@article{schramm_contour_2009,
	title = {Contour lines of the two-dimensional discrete {Gaussian} free field},
	volume = {202},
	issn = {0001-5962, 1871-2509},
	url = {https://projecteuclid.org/journals/acta-mathematica/volume-202/issue-1/Contour-lines-of-the-two-dimensional-discrete-Gaussian-free-field/10.1007/s11511-009-0034-y.full},
	doi = {10.1007/s11511-009-0034-y},
	number = {1},
	urldate = {2025-10-08},
	journal = {Acta Math.},
	author = {Schramm, Oded and Sheffield, Scott},
	year = {2009},
	pages = {21--137},
}

@misc{cai_heterochromatic_2025,
	title = {Heterochromatic two-arm probabilities for metric graph {Gaussian} free fields},
	url = {http://arxiv.org/abs/2510.20492},
	doi = {10.48550/arXiv.2510.20492},
	urldate = {2025-10-24},
	publisher = {arXiv},
	author = {Cai, Zhenhao and Ding, Jian},
	year = {2025},
	note = {arXiv:2510.20492},
}

@misc{cai_separation_2025,
	title = {Separation and cut edge in macroscopic clusters for metric graph {Gaussian} free fields},
	url = {http://arxiv.org/abs/2510.20516},
	doi = {10.48550/arXiv.2510.20516},
	urldate = {2025-10-24},
	publisher = {arXiv},
	author = {Cai, Zhenhao and Ding, Jian},
	year = {2025},
	note = {arXiv:2510.20516},
}

@article{prevost_first_2025,
	title = {First {Passage} {Percolation}, {Local} {Uniqueness} for {Interlacements} and {Capacity} of {Random} {Walk}},
	volume = {406},
	issn = {1432-0916},
	url = {https://doi.org/10.1007/s00220-024-05195-0},
	doi = {10.1007/s00220-024-05195-0},
	language = {en},
	number = {2},
	urldate = {2025-11-27},
	journal = {Comm. Math. Phys.},
	author = {Pr{\'e}vost, Alexis},
	year = {2025},
	pages = {34},
}

@article{sheffield_conformal_2012,
	title = {Conformal loop ensembles: the {Markovian} characterization and the loop-soup construction},
	volume = {176},
	number = {3},
	doi = {10.4007/annals.2012.176.3.8},
	journal = {Ann. of Math.},
	author = {Sheffield, Scott and Werner, Wendelin},
	year = {2012},
	pages = {1827--1917},
}

@article{cao_level_2021,
	title = {Level set percolation in the two-dimensional {Gaussian} free field},
	volume = {126},
	number = {12},
	doi = {10.1103/PhysRevLett.126.120601},
	journal = {Phys. Rev. Lett.},
	author = {Cao, Xiangyu and Santachiara, Raoul},
	year = {2021},
	pages = {120601},
}

@article{bauerschmidt_logarithmic_2015,
	title = {Logarithmic correction for the susceptibility of the 4-dimensional weakly self-avoiding walk: a renormalisation group analysis},
	volume = {337},
	number = {2},
	doi = {10.1007/s00220-015-2352-6},
	journal = {Comm. Math. Phys.},
	author = {Bauerschmidt, Roland and Brydges, David C. and Slade, Gordon},
	year = {2015},
	pages = {817--877},
}

@article{bauerschmidt_scaling_2014,
	title = {Scaling limits and critical behaviour of the 4-dimensional $n$-component $|\varphi|^4$ spin model},
	volume = {157},
	number = {4-5},
	doi = {10.1007/s10955-014-1060-5},
	journal = {J. Stat. Phys.},
	author = {Bauerschmidt, Roland and Brydges, David C. and Slade, Gordon},
	year = {2014},
	pages = {692--742},
}

@misc{vanengelenburg_one-arm_2025,
	title = {One-arm exponents of the high-dimensional {Ising} model},
	author = {van Engelenburg, Diederik and Garban, Christophe and Panis, Romain and Severo, Franco},
	year = {2025},
	note = {arXiv:2510.23423},
}

@misc{hutchcroft_critical_2025,
	title = {Critical long-range percolation {III}: the upper critical dimension},
	author = {Hutchcroft, Tom},
	year = {2025},
	note = {arXiv:2508.18809},
}

@article{li_lower_2014,
	title = {A lower bound for disconnection by random interlacements},
	volume = {19},
	number = {17},
	doi = {10.1214/EJP.v19-3067},
	journal = {Electron. J. Probab.},
	author = {Li, Xinyi and Sznitman, Alain-Sol},
	year = {2014},
	pages = {1--26},
}

@article{lupu_convergence_2019,
	title = {Convergence of the two-dimensional random walk loop-soup clusters to {CLE}},
	volume = {21},
	number = {4},
	doi = {10.4171/JEMS/859},
	journal = {J. Eur. Math. Soc.},
	author = {Lupu, Titus},
	year = {2019},
	pages = {1201--1227},
}

@misc{bi_arm_2025,
	title = {Arm events in critical planar loop soups},
	author = {Bi, Yijie and Gao, Yifan and Nolin, Pierre and Qian, Wei},
	year = {2025},
	note = {arXiv:2509.25024},
}

@article{hutchcroft_logarithmic_2023,
	title = {Logarithmic corrections to scaling in the four-dimensional uniform spanning tree},
	volume = {401},
	doi = {10.1007/s00220-023-04686-w},
	journal = {Comm. Math. Phys.},
	author = {Hutchcroft, Tom and Sousi, Perla},
	year = {2023},
	pages = {2115--2191},
}

@article{chiarini_entropic_2020,
	title = {Entropic repulsion for the {Gaussian} free field conditioned on disconnection by level-sets},
	volume = {177},
	number = {1-2},
	doi = {10.1007/s00440-019-00957-7},
	journal = {Probab. Theory Related Fields},
	author = {Chiarini, Alberto and Nitzschner, Maximilian},
	year = {2020},
	pages = {525--575},
}

@article{chiarini_solidification_2026,
	title = {Solidification estimates for random walks on supercritical percolation clusters},
	volume = {65},
	doi = {10.1007/s11118-026-10287-8},
	journal = {Potential Anal.},
	author = {Chiarini, Alberto and Liu, Zhizhou and Nitzschner, Maximilian},
	year = {2026},
	pages = {3},
}

@article{drewitz_critical_2025,
	title = {Critical one-arm probability for the metric {Gaussian} free field in low dimensions},
	doi = {10.1007/s00440-025-01392-7},
	journal = {Probab. Theory Related Fields},
	author = {Drewitz, Alexander and Pr{\'e}vost, Alexis and Rodriguez, Pierre-Fran{\c c}ois},
	year = {2025},
}

@misc{aru_renormalised_2026,
	title = {Renormalised two-point functions of {CLE}$_4$ gaskets},
	author = {Aru, Juhan and Lupu, Titus},
	year = {2026},
	note = {arXiv:2604.15146},
}

@article{nolin_near_2008,
	title = {Near-critical percolation in two dimensions},
	volume = {13},
	doi = {10.1214/EJP.v13-565},
	journal = {Electron. J. Probab.},
	author = {Nolin, Pierre},
	year = {2008},
	pages = {1562--1623},
}

@article{duminil-copin_planar_2022,
	title = {Planar random-cluster model: scaling relations},
	volume = {10},
	journal = {Forum Math. Pi},
	author = {Duminil-Copin, Hugo and Manolescu, Ioan},
	year = {2022},
	pages = {e23},
}

@misc{bi_2026_onear,
  title = {One-Arm Probabilities for the Two-Dimensional Metric-Graph and Discrete {{Gaussian}} Free Field},
  author = {Bi, Yijie and Gao, Yifan and Li, Xinyi},
  year = 2026,
  number = {arXiv:2603.25400},
  eprint = {2603.25400},
  primaryclass = {math},
  publisher = {arXiv},
  doi = {10.48550/arXiv.2603.25400},
  urldate = {2026-03-31},
  archiveprefix = {arXiv}
}

@article{beffara_percolation_2017,
 author = {Beffara, Vincent and Gayet, Damien},
 title = {Percolation of random nodal lines},
 journal = {Publ. Math. Inst. Hautes \'Etudes Sci.},
 volume = {126},
 pages = {131--176},
 year = {2017},
 doi = {10.1007/s10240-017-0093-0},
}

@article{muirhead_sharp_2020,
 author = {Muirhead, Stephen and Vanneuville, Hugo},
 title = {The sharp phase transition for level set percolation of smooth planar {G}aussian fields},
 journal = {Ann. Inst. Henri Poincar\'e Probab. Stat.},
 volume = {56},
 number = {2},
 pages = {1358--1390},
 year = {2020},
}

@book{kesten_percolation_1982,
 author = {Kesten, Harry},
 title = {Percolation theory for mathematicians},
 publisher = {Birkh{\"a}user},
 year = {1982},
}
\bibliographystyle{abbrv}
}

\end{document}